\documentclass[11pt,oneside,reqno]{amsart}
\usepackage{graphicx}
\usepackage{amsmath}
\usepackage{amssymb}
\usepackage{epstopdf}
\usepackage{xcolor}
\usepackage{subfigure}
\usepackage[pagewise]{lineno}\linenumbers
\usepackage{enumerate}
 \usepackage{amsaddr}
\usepackage{tikz-cd}
\usetikzlibrary{spy}
\usepackage{ctable}
\usepackage{pifont}

\newcommand{\eps}{\varepsilon}
\newcommand{\R}{\mathbb{R}}
\newcommand{\Z}{\mathbb{Z}}
\newcommand{\C}{\mathbb{C}}
\newcommand{\N}{\mathbb{N}}

\newcommand{\Id}{\mathrm{Id}}
\newcommand{\rmi}{\mathrm{i}}
\newcommand{\rme}{\mathrm{e}}
\newcommand{\rmO}{\mathrm{O}}

\newcommand{\rmo}{\mathrm{o}}

\newtheorem{Lemma}{Lemma}[section]
\newtheorem*{Lemma*}{Lemma}
\newtheorem{Theorem}{Theorem}
\newtheorem*{Theorem*}{Theorem}
\newtheorem{Proposition}[Lemma]{Proposition}
\newtheorem*{Proposition*}{Proposition}
\newtheorem{Corollary}[Lemma]{Corollary}
\newtheorem{Remark}[Lemma]{Remark}
\newtheorem{Definition}[Lemma]{Definition}
\newtheorem{Hypothesis}[Lemma]{Hypothesis}
\newtheorem*{Hypothesis*}{Hypothesis}

\title{
Existence of spiral waves in oscillatory media with nonlocal coupling}

\thanks{This work is supported by NSF DMS-1911742.}
\thanks{AMS subject classification: 45K05, 45G15, 46N20, 35Q56, 35Q92} 

\author{Gabriela Jaramillo}
\address{University of Houston, Department of Mathematics, 3551 Cullen Blvd., 
 Houston, TX 77204-3008, USA}

\begin{document}
\nolinenumbers
\maketitle
\begin{abstract}

We prove existence of spiral waves in oscillatory media with nonlocal coupling.
Our starting point is a nonlocal complex Ginzburg-Landau (cGL) equation, rigorously derived
as an amplitude equation for integro-differential equations undergoing a Hopf bifurcation.
Because this reduced equation includes higher order terms that are usually ignored in a formal derivation of the cGL, the solutions we find also correspond to solutions of the original nonlocal system.
To prove existence of these patterns we use perturbation methods together with the implicit function theorem. 
Within appropriate parameter regions, we find that spiral wave patterns
have wavenumbers, $\kappa$, with expansion $\kappa \sim C \rme^{-a/\eps}$, where $a$ is a positive constant, $\eps$ is the small bifurcation parameter, and the positive constant $C$ depends on the strength and spread of the nonlocal coupling.
The main difficulty we face comes from the linear operators appearing in our system of equations.
Due to the symmetries present in the system, and because the equations are posed on the plane, these maps
have a zero eigenvalue embedded in their essential spectrum. 
Therefore, they are not invertible when defined between standard Sobolev spaces and a straightforward application of the implicit function theorem is not possible. We surpass this difficulty by redefining the domain of these operators using doubly weighted Sobolev spaces. These spaces encode algebraic decay/growth properties of functions, near the origin and in the far field, and allow us to recover Fredholm properties for these maps.


\end{abstract}

\vspace*{0.2in}

{\small
{\bf Running head:} {Existence of Spiral Waves }

{\bf Keywords:} pattern formation, nonlocal diffusion, integro-differential equations, Fredholm operators.

{\bf AMS subject classification: 45K05, 45G15, 46N20, 35Q56, 35Q92} 
}


\vspace*{0.2in}

\section{Introduction}

The term oscillatory media describes systems which combine self-sustained time oscillations with mechanisms that allow for spatial interactions, or coupling. 
Examples include  electrochemical systems \cite{koper1998non, Krischer2002},  oscillating chemical reactions \cite{epstein1983, Taylor2002},
colonies of aggregating slime mold \cite{goldbeter2006},
and under certain assumptions, even heart \cite{zebrowski2007} and brain tissue \cite{ermentrout2016}. 
 Interest in these systems stems in part from their ability to generate beautiful spatio-temporal structures like target patterns, traveling waves, and spiral waves. While properties of these patterns have been extensively studied in the case of oscillatory media with local coupling, not many results  address the case of systems involving long-range interactions.
In this paper we take on this challenge, focusing on existence of spiral waves in planar spatially extended oscillatory media with nonlocal coupling.

Our interest in spiral waves comes from numerical experiments done by Kuramoto and coauthors, 
which focused on an abstract FitzHugh-Nagumo system
that incorporates a convolution term in place of the standard Laplacian, \cite{shima2004}.
Their simulations show  that, depending on system parameters, the nonlocal coupling described by this convolution operator can give rise to a new type of pattern known as a spiral chimera. As the name suggest, this novel structure looks very much like a spiral wave in the far field, but has a core which does not act in synchrony with the rest of the pattern,
 (see Figure \ref{f:chimera}). 
 As a first step towards understanding the formation of these new structures, here we address how nonlocal forms of coupling  
 affect the formation and shape of `regular' spiral waves.
 
 \begin{figure}[t] 
    \centering
    \includegraphics[width=2in]{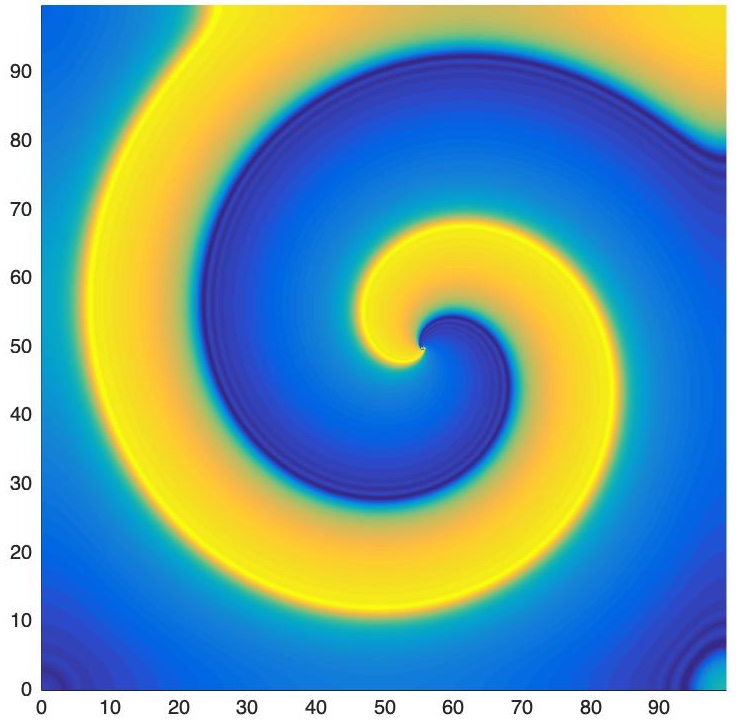} 
    \includegraphics[width=2in]{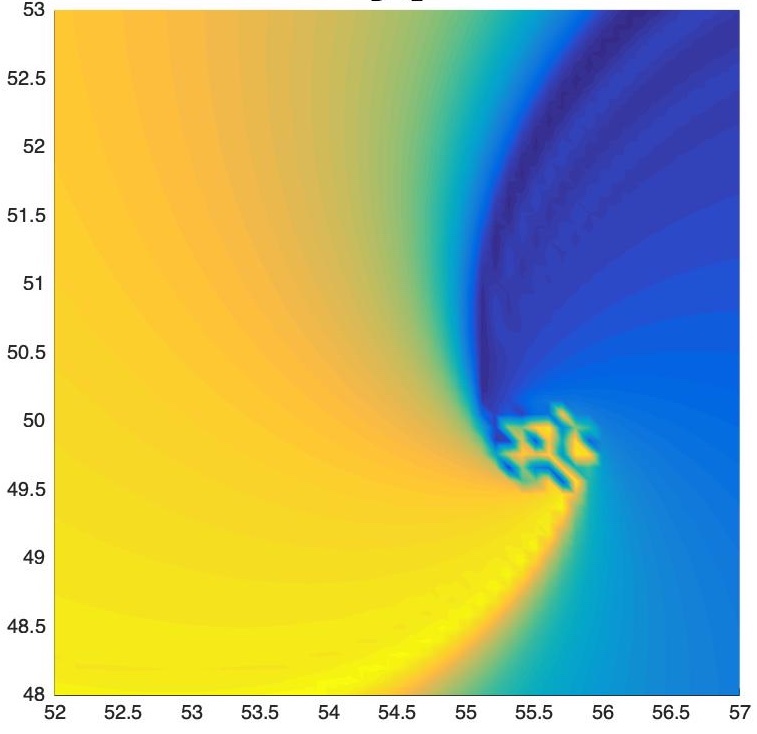} 
    
    \caption{Spiral chimeras.}
    \label{f:chimera}
 \end{figure}

To model nonlocal coupling, we use convolution operators of diffusive type. These operators are described by convolution kernels whose Fourier symbols are radially symmetric, uniformly bounded, analytic, and have a quadratic tangency  near the origin.  This choice of coupling then leads to
 model equations that are nonlocal and that take the form,
 \begin{equation}\label{e:model}
  U_t = D \mathcal{L} \ast U + F(U;\mu) \quad U\in \R^2, \quad x\in \R^2, \qquad \mu \in \R.
  \end{equation}
 Here $D$ is a matrix of diffusion coefficients, while $\mathcal{L} $ represents our choice of convolution operator. The symbol $F$ then describes reaction terms that undergo a Hopf bifurcation as the parameter $\mu$ crosses the origin.

While similar in structure to reaction-diffusion equations modeling oscillatory media with local coupling,  
 integro-differential equations like \eqref{e:model} are in general more challenging to analyze. 
 On the one hand, the convolution operator 
 is easier to handle if the equations are posed on the plane, since in this case one does not have to  
 worry about imposing boundary  constraints that are compatible with the operator, and 
 which in addition respect desired modeling assumptions.
  But this simplification comes at a price. When posed on $\R^2$, the linearization of equation \eqref{e:model} about a steady state is now an operator with real essential spectrum touching the origin. Moreover, due to the translational symmetry of the equation this operator has a non-trivial nullspace, or equivalently, a zero-eigenvalue embedded in the essential spectrum. Consequently, the linearization is not an invertible, nor a Fredholm operator, when viewed as a map between standard Sobolev spaces. As a result, one cannot immediately use perturbation methods together with the implicit function theorem, or Lyapunov-Schmidt reduction, to prove existence of solutions which bifurcate from a steady state.
 
A similar difficulty is encountered  when considering reaction-diffusion equations posed on unbounded domains. 
While in this case it is possible to prove existence of solutions  by reformulating the problem as an ordinary differential equation and employing methods from spatial dynamics \cite{kollar2007, kopell1981, scheel1998},
 this approach is not readily applicable for integro-differential equations (unless one assumes the convolution kernel has a particular form that allows one to write the equations as pde's, see \cite{bartpinto}). 
 An exception is the construction of a center manifold, which can be done without reference to a phase space using fixed point methods, see \cite{faye2018center, cannon2023}. 
However,  this approach implicitly assumes that the system is in a regime where the nonlocal coupling is well approximated by local interactions,  as evidenced by the fact that the resulting equations describing the 'flow' on the center manifold are differential equations. In contrast, here we are interested in the opposite regime, where interactions between oscillating elements are truly nonlocal.

Therefore, our starting point will be a {\it nonlocal} complex Ginzburg-Landau equation, rigorously derived in \cite{jaramillo2022rotating} as
an amplitude equation for rotating wave solutions of  integro-differential equations of the form \eqref{e:model}.
To prove existence of spiral waves we use
 perturbation methods together with the implicit function theorem.
To overcome the course of the zero-eigenvalue, we follow the approach taken in \cite{jaramillo2018, jaramillo2019, jaramillo2022can}, where it is shown that one can recover Fredholm properties (closed range, finite dimensional kernel and cokernel) for convolution operators  of diffusive type, and related maps, using {\it algebraically weighted} Sobolev spaces.
In the rest of this introduction, we briefly describe the derivation of the nonlocal amplitude equation, state our main  theorem, and give a short outline for the paper. We finish this introduction with a discussion of our results.

\subsection{The Nonlocal Amplitude Equation}

Close to the onset of oscillations and under the assumption of weak {\it local} coupling, oscillatory media may be described by the complex Ginzburg-Landau (cGL) equation. This reduced equation describes variation in the amplitude of oscillations which occur over long spatial and time scales, and can thus be formally derived using a multiple-scale analysis, see
 \cite{kuramoto2003, van1995complex}.
This method can also be extended  to account for other forms of coupling, including global and nonlocal coupling \cite{tanaka2003, garcia2008}, and to incorporate feedback mechanisms and forcing terms \cite{garcia2012}.
Since this is a formal approach, it is then necessary to  justify the validity of the equation. That is, one must prove that the approximate solutions obtained using the cGL are close to the actual solutions of the corresponding system in an appropriate metric. See for instance \cite{schneider1992, kuehn2018, schneider1996, vanharten1991} for works that address this question.

Instead, the work presented in \cite{jaramillo2022rotating} takes a different approach. There, the method of multiple-scales is given a rigorous treatment in order to derive, and validate, an amplitude equation for {\it rotating wave solutions} of \eqref{e:model}.
The result is the following {\it nonlocal} complex Ginzburg-Landau equation, 
  \begin{equation}\label{e:gl1}
  0 = K \ast w + (1+ \rmi \lambda) w - (1 + \rmi \beta) |w|^2 w + N(w;\eps), \qquad x = (r, \theta )  \in \R^2,
  \end{equation}
where the symbol $K$ represents a scaled version of the original operator describing the nonlocal coupling, the unknown $w= w(r)$ is a radial complex-valued function, $\lambda$ and $\beta $ are real parameters, and $\eps$ is a small number measuring how close the system is to the Hopf instability.
The term $N(w,\eps) $ then summarizes nonlinear higher order correction terms of size $\rmO(\eps)$ which, as explained below, are needed in order to rigorously prove the existence of spiral waves.

As in the formal derivation of the local cGL, the method used to arrive at equation \eqref{e:gl1}
focuses on  small  amplitude oscillations that emerge close to the Hopf bifurcation. 
In the formal derivation of the amplitude equation this difference in scales then allows one to expand  solutions to equation \eqref{e:model} in powers of $\eps$, i.e 
$$U = \eps U_1 + \eps^2 U_2 + \eps^3 U_3 +\cdots.$$
Gathering similar terms one then obtains a sequence of equations. Using an appropriate ansatz  on a co-rotating frame that moves with the rotational speed, $c$, of the wave (i.e. $U(r,\theta,t)= U(r, \theta - ct)$), one then finds that the order $\eps$ and $\eps^2$ equations can easily be solved, while  the  cGL equation then appears as a solvability condition at order $\eps^3$. As mentioned above, because  all higher order terms are then ignored it is then necessary to justify the validity of the equation.

In contrast, to place the multiple-scales method in a more rigorous setting, 
the approach in \cite{jaramillo2022rotating}
assumes solutions can be written using a finite expansion, i.e.
$U = \eps U_1 + \eps^2 U_2 + \eps^3 U_3.$
Again, one finds that the order $\eps$ and $\eps^2$ equation can readily be solved, while {\it all remaining terms} of order $\rmO(\eps^3)$ are now gathered into one main equation.
The results from \cite{jaramillo2022rotating} show that this main
equation can  be split into an invertible system and a reduced equation. 
Using the implicit function theorem one can then solve the invertible system, thus obtaining a family of solutions parametrized by the first order correction term, $U_1$. That is, one finds $U_i = \Psi_i(U_1;\eps)$ with $i =2,3$, for some $C^1 $ functions, $\Psi_i$.
After inserting this family of solutions into the reduced equation and  projecting onto the angular Fourier modes, $\rme^{ \pm \rmi \theta}$, one arrives at the nonlocal cGL equation \eqref{e:gl1},
where one now sees that the symbol $N(w;\eps)$ encodes all remaining terms of order $\rmO(\eps^4)$. 
As a result, solving the amplitude equation \eqref{e:gl1} is equivalent to solving the original integro-differential equation \eqref{e:model}.
Thus,  to rigorously prove existence of solution to \eqref{e:model}
 representing spiral waves, it is enough to prove this result for  the nonlocal cGL  equation given by \eqref{e:gl1}.

\subsection{Main Result}
Before stating our main result, let us point out a couple of properties of the amplitude equation \eqref{e:gl1} and of the 
solutions we seek.

First, because the reduce equation \eqref{e:gl1} comes from using a suitable projection onto an angular mode, and because it is based on an ansatz that moves in a co-rotating frame,
the unknown $w$ depends only on the radial variable $r$. While it is assumed that the solution, $U$, to the original
integro-differential equation is rotating with speed $c$, the value of this parameter is unknown. 
In the reduced equation \eqref{e:gl1}, the speed of the wave is captured by the parameter $\lambda$ via the relation $c = \omega + \eps^2 \lambda$, where $\omega$ represents the frequency of the time oscillations emerging from the Hopf bifurcation, see \cite{jaramillo2022rotating}.
Thus, using the gauge symmetry of the equation, we have that an equivalent formulation for our problem is to find solutions to
 \begin{equation}\label{e:gl2}
  \tilde{w}_t = K \ast \tilde{w} +  \tilde{w} - (1 + \rmi \beta) |\tilde{w}|^2 \tilde{w} + N(\tilde{w};\eps), \qquad x = (r, \theta )  \in \R^2,
  \end{equation}
  where $\tilde{w}(r,t) = w(r) \rme^{-\rmi \lambda t}$.
  
  Second, because in the co-rotating frame spiral waves look like target patterns, our goal is to find
  constants $\kappa$ and $A$, as well as solutions to \eqref{e:gl2} 
of the form $\tilde{w}(r,t) = \rho(r) \rme^{\rmi (\phi(r) - \lambda t)}$, such that $\phi(r) \to \kappa$ and $\rho(r) \to A$,  as $r$ goes to infinity. This is in essence the content of Theorem \ref{t:main}.

To complete the formulation of our problem, we now state our main assumptions regarding the 
the nonlinear terms $N(w;\eps)$ and convolution operator $K$.

\begin{Hypothesis*}[H1]
The nonlinear function $N(w; \eps)$ is order $\rmO(\eps\; |w|^4w)$, and every term in this expression is of the form $c |w|^{2n} w$, with $c\in \C$, and $n \in \{1,2,3, \cdots\}$.
\end{Hypothesis*}

A justification for Hypothesis [H1] is provided in Appendix \ref{a:nonlinear}.

\begin{Hypothesis*}[H2]
The convolution operator $K\ast$ is defined by a radially symmetric kernel given by
\[ K(|x|) = \frac{\eta}{\eps^2 D} \left ( K_1(|x|/\sqrt{\eps^2 D})-1 \right)\]
where $\eta, D$ are positive constants, $0<\eps<<1$, and $K_1(\xi)$ is the order one modified Bessel function of the second kind.
\end{Hypothesis*}

With these considerations in this paper we prove the following theorem.

\begin{Theorem}\label{t:main}
Let  $\eps, \beta$ be real numbers, with $\beta <0 $. Let $K$ denote a convolution operator of the form described by Hypothesis (H2) and take $N(w;\eps)$ to be higher order terms of the form stated in Hypothesis (H1).
Then, there exists a small positive number, $ \eps_0,$ and a family of solutions, $w(r;\eps)= \rho(r;\eps) \rme^{\rmi (\phi(r;\eps) - \lambda(\eps) t) } $,
 to the nonlocal complex Ginzburg-Landau equation \eqref{e:gl2},
which is valid and $C^1$ in the $\eps$- neighborhood $(-\eps_0,\eps_0)$.
Moreover, with $\delta  \sim \eps^{1/4}$, this family has expansions, 
\begin{align*}
\lambda(\eps) & = \beta + \delta^2 \Omega(\beta) \\
 \rho(r;\eps) & = \rho_0(\tilde{r}) + \delta^2(R_0(\delta \tilde{r}) + \delta R_1(\delta \tilde{r}))\\
\phi(r;\eps) & = \phi_0(\delta \tilde{r}) + \delta \phi_1(\delta \tilde{r}),\\
\tilde{r} &= r/ \sqrt{ \eta -\eps^2 D},
\end{align*}
where $\eta$ and $D$ are coefficients appearing in the definition of the operator $K$, and 
\[ \Omega(\beta) =  4 \tilde{C}(\beta ) \exp(-M/\beta^2),\]
for some constant $M>0$ and a $C^1(\R)$ function $\tilde{C}(\beta)$.
  In addition, the lower order correction terms satisfy,
\begin{itemize}
\setlength \itemsep{1ex} 
\item $\phi_0(\delta \tilde{r}) \sim \frac{1}{\beta}  \log( K_0(\Lambda \delta \tilde{r} ))  + c_1 \quad $ as $S = \delta \tilde{r} \to \infty \quad $, for some $c_1 \in \R$,
\item $R_0(\delta \tilde{r}) = - \frac{1 }{2} \rho_0(\tilde{r}) (\partial_S \phi_0( \delta \tilde{r})) ^2$,
\end{itemize}
while
\begin{itemize}
\setlength \itemsep{1ex} 
\item $\rho(r;\eps)  \longrightarrow  1 + \delta^2 c_2, \quad $  for some $c_2 \in \R$, and
 \item $\partial_r \phi(r;\eps) \longrightarrow \kappa = - \dfrac{\Lambda}{\beta} \dfrac{\delta}{ \sqrt{ \eta - \eps^2 D}} \qquad $ with $\Lambda = \sqrt{ - \beta \Omega(\beta) }$,
\end{itemize}
as $r \to \infty$.
\end{Theorem}

Before continuing let us highlight how the results from Theorem \ref{t:main} 
establish a relation between properties of the 
operator $K$ and the shape of the spiral wave. 
From Hypothesis (H2) we see that the parameters $\eta$ and $D$ in the definition of the operator $K$
control the strength and the spread of the convolution kernel, respectively.
These parameters then appear in the far-field approximation of the spiral's wavenumber,
$$\kappa = - \dfrac{\Lambda}{\beta} \dfrac{\delta}{ \sqrt{ \eta - \eps^2 D}}.$$
The above expression shows that as the strength, $\eta$, of the convolution kernel increases,
the spiral's wavenumber, $\kappa$, decreases. On the other hand, if the spread, $D$, of the operator increases past a certain 
threshold, the approximation for $\kappa$ breaks down. 
These results are in good agreement with our simulations, see Figure \ref{f:spirals_eta} and Figure \ref{f:spirals_D}.
In particular, notice how when the parameter $D$ is too large, we no longer obtain spiral waves but rather spiral chimeras.

\begin{figure}[ht] 
   \centering
   \includegraphics[width=0.2\textwidth]{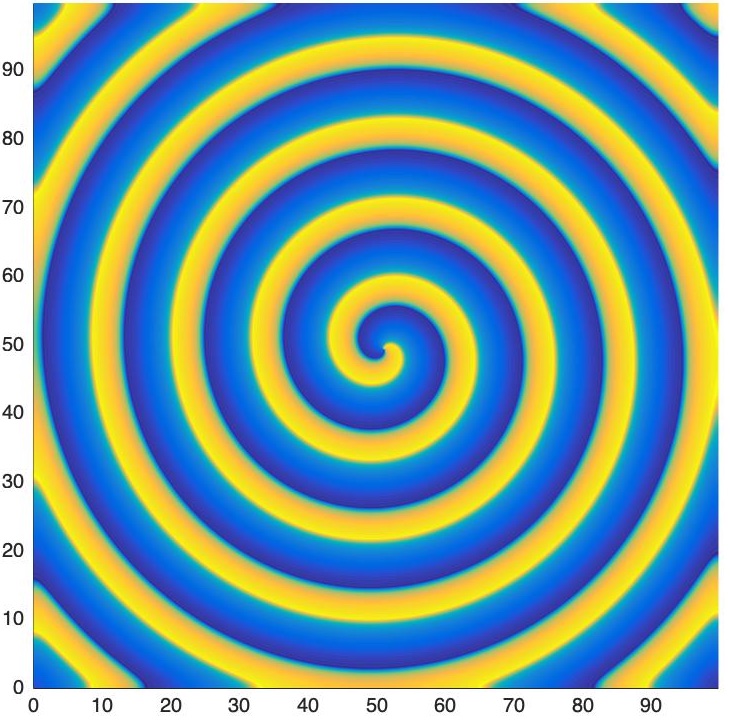} 
       \includegraphics[width=0.2\textwidth]{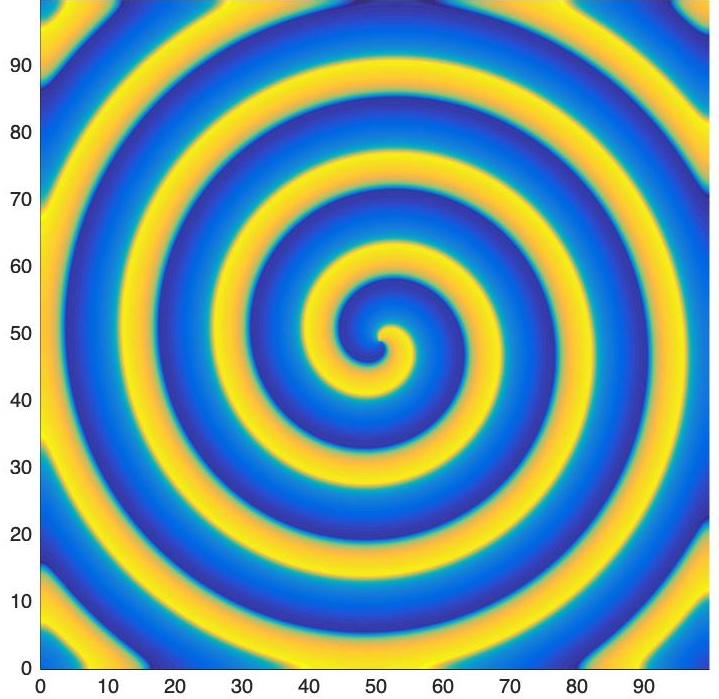} 
       \includegraphics[width=0.2\textwidth]{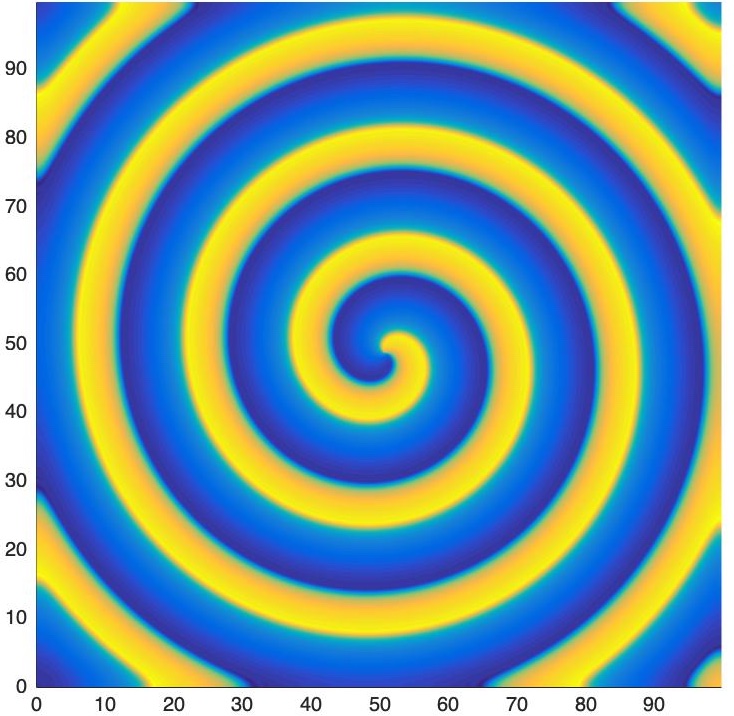} 
       \includegraphics[width=0.2\textwidth]{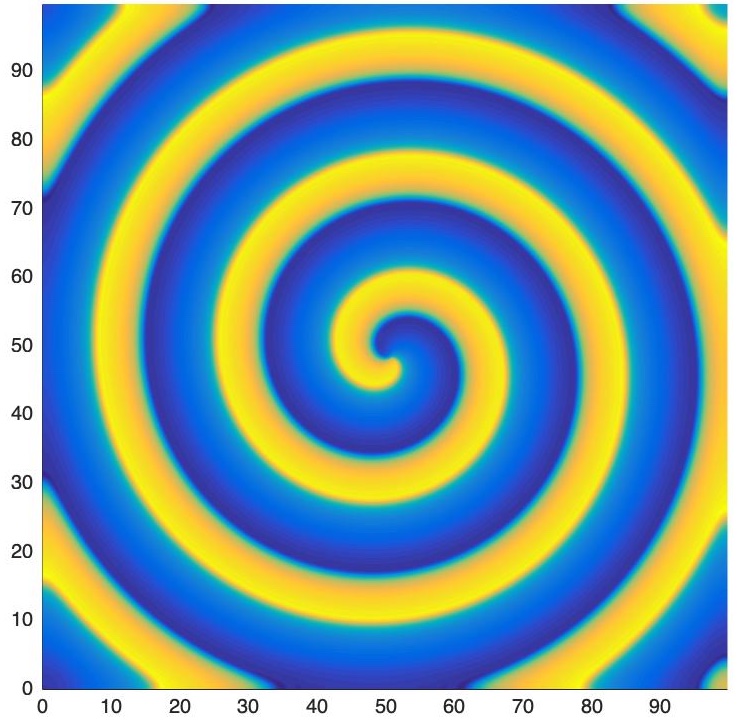} 
   \caption{Simulation of FitzHugh-Nagumo system appearing in \cite{shima2004} on a square domain of length $L =100$, using a cosine spectral method and an implicit Euler time stepping scheme with $N = 1024$ nodes and a time step $h= 0.05$. Convolution kernel used has Fourier symbol $\hat{K}(\xi) = \frac{ -\eta | \xi|^2}{ 1+ \eps^2 D |\xi|^2} $ with fixed $\eps^2 D =0.1 $ and varying $\eta$. From left to right we have $\eta =0.5, \eta= 1.0, \eta=1.5,$ and $\eta= 2.0$.}
   \label{f:spirals_eta}
\end{figure}

\subsection{Outline/ Sketch of Proof for Theorem 1}\label{ss:outline}
Although both formulations for our problem, i.e. equation \eqref{e:gl1} and equation \eqref{e:gl2}, are equivalent,  to prove Theorem \ref{t:main} we use the first equation.
 We now give a short summary of how we arrive at our main result.

First, notice that the form of the convolution operator $K$ stated in Hypothesis [H2] comes from projecting the  map 
\[ \mathcal{L} = \eta ( 1- \eps^2 D \Delta)^{-1} \Delta\]
onto the angular Fourier modes $\rme^{\pm \rmi \theta}$.
As a result we can formally write $K$ as 
\begin{equation}\label{e:kdecomposed}
K \ast w = \eta( 1- \eps^2 D \Delta_1)^{-1} \Delta_1 w, 
\end{equation}
where $\Delta_1 = \partial_{rr} + \frac{1}{r} \partial_r - \frac{1}{r^2}$.
Notice also that the operator $\mathcal{L}$ has a radially symmetric Fourier symbol $\hat{\mathcal{L}}(\xi)= - \eta |\xi|^2/(1 + \eps^2 D |\xi|^2)$.
Because the Fourier Transform commutes with orthogonal transformations, we have that the Fourier symbol of $K \ast \cdot + \Id $ is also radially symmetric \cite{stein1971}, and is thus given by
\[ \mathcal{F} ( K \ast w + w) = \left[\frac{- \eta |\xi|^2}{1 + \eps^2 D |\xi|^2}  + 1\right] \hat{w} = \left[ - (\eta - \eps^2 D) |\xi|^2 + \frac{ \eta \eps^2 D |\xi|^4}{ 1+ \eps^2 D |\xi|^2}\right] \hat{w}.\]
Therefore, equation \eqref{e:gl1} can also be written as
\begin{equation}\label{e:aux1}
0 = ( \eta - \eps^2 D) \Delta_1 w + (1+ \rmi \lambda) w - (1 + \rmi \beta) |w|^2 w + \eps^2 J \ast w + N(w;\eps),
\end{equation}
with
$\hat{J}(\xi)  = \frac{\eta D |\xi|^4}{1 + \eps^2 D |\xi|^2}.$

\begin{figure}[ht] 
   \centering
 \includegraphics[width=0.2\textwidth]{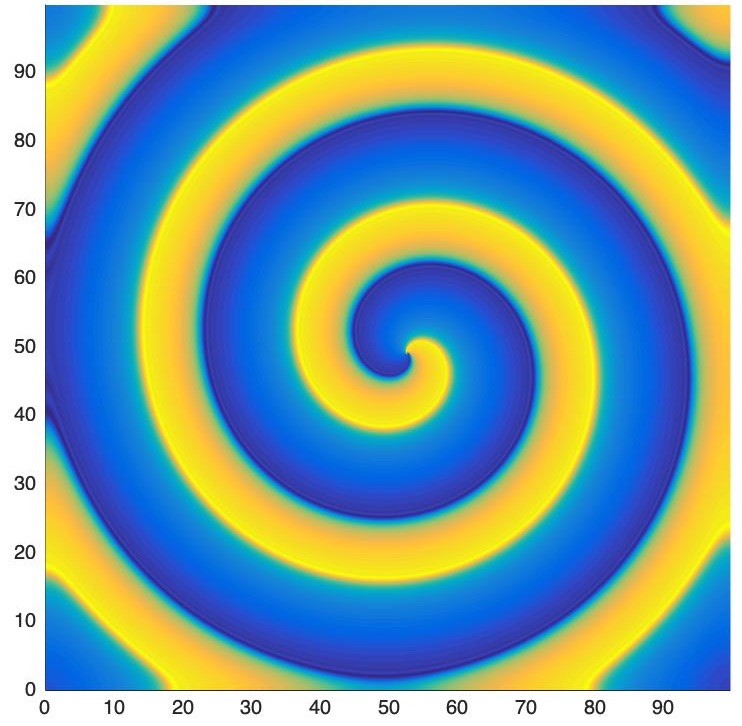} 
       \includegraphics[width=0.2\textwidth]{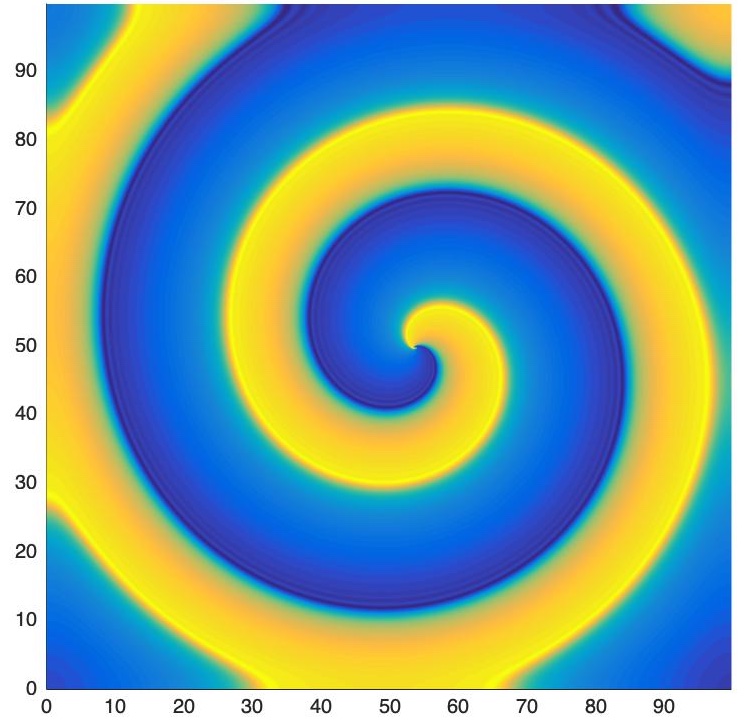}  
       \includegraphics[width=0.2\textwidth]{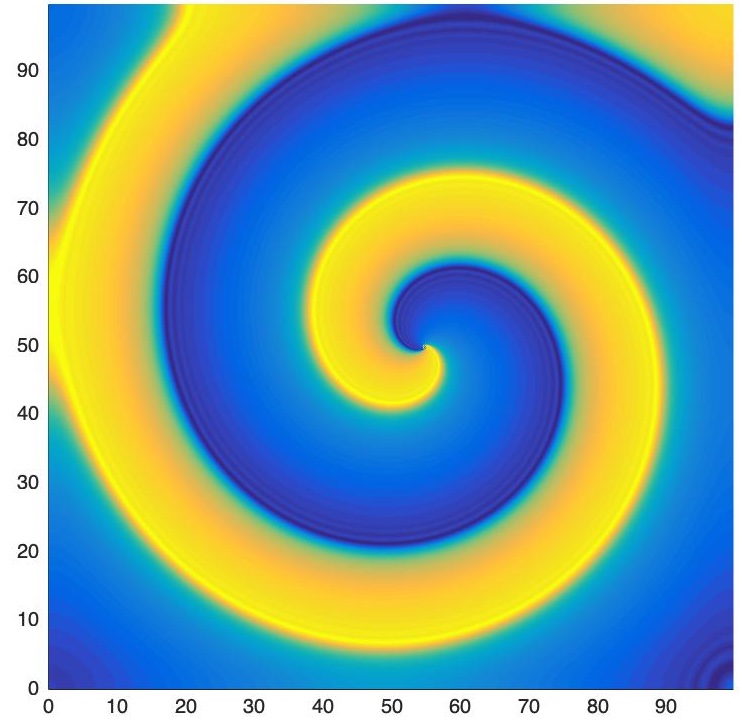} 
       \includegraphics[width=0.2\textwidth]{Figures/SpiralD_2_0.jpg}
      
       \includegraphics[width=0.2\textwidth]{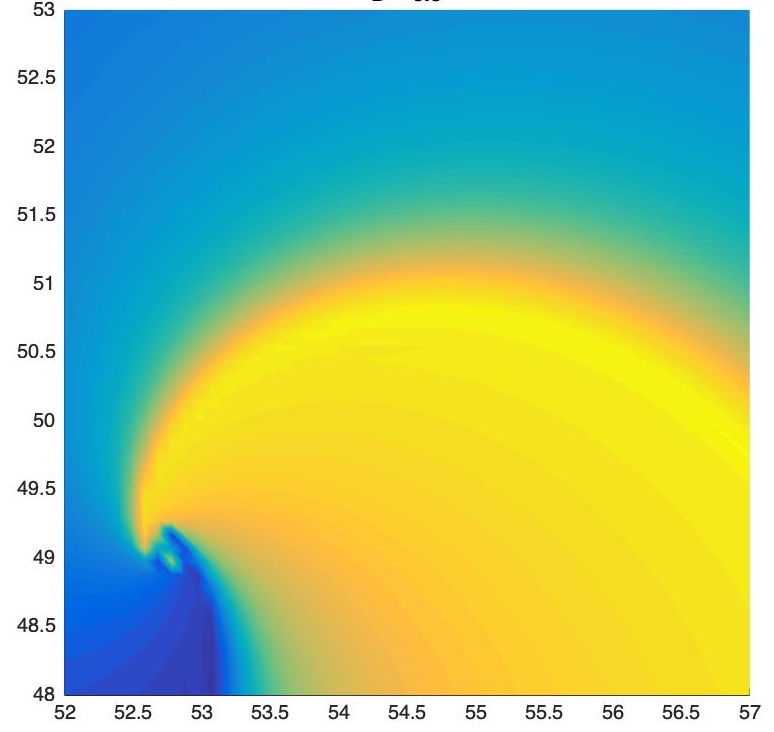} 
        \includegraphics[width=0.2\textwidth]{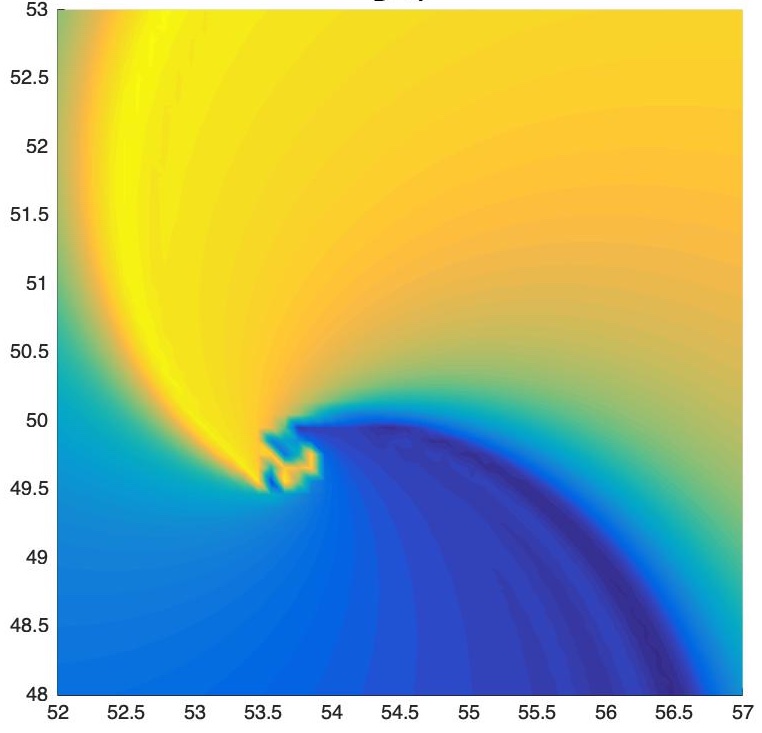} 
       \includegraphics[width=0.2\textwidth]{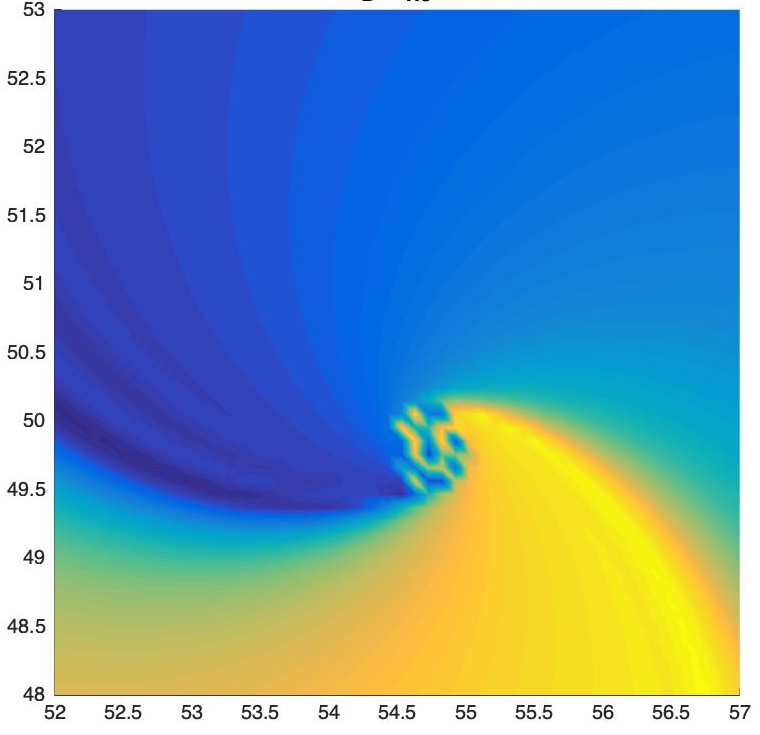} 
      \includegraphics[width=0.2\textwidth]{Figures/SpiralzoomD_2_0.jpg}  
   
     \caption{Simulation of FitzHugh-Nagumo system appearing in \cite{shima2004} on a square domain of length $L =100$, using a cosine spectral method and an implicit Euler time stepping scheme with $N = 1024$ nodes and a time step $h= 0.05$. Convolution kernel used has Fourier symbol $\hat{K}(\xi) = \frac{ -\eta | \xi|^2}{ 1+ \eps^2 D |\xi|^2} $ with fixed $\eta =1 $ and varying $\tilde{D} =\eps^2 D$. From left to right we have $\tilde{D} =0.5, \tilde{D}= 1.0, \tilde{D}=1.5,$ and $\tilde{D}= 2.0$. First row depicts full spiral, second row zooms in into core of spiral.}
   \label{f:spirals_D}
\end{figure}

At the same time, we can take advantage of the formal definition of the operator $K$ given in  \eqref{e:kdecomposed} and precondition equation  \eqref{e:gl1} by $(1 - \eps^2 D \Delta_1)$ to arrive at 
\begin{equation}\label{e:aux2}
0 = ( \eta - \eps^2 D) \Delta_1 w + (1+ \rmi \lambda) w + ( 1 - \eps^2 D \Delta_1) [ - (1 + \rmi \beta) |w|^2 w + N(w;\eps)].
 \end{equation}
 
Since both formulations, equations \eqref{e:aux1} and \eqref{e:aux2}, share similar first order terms, the first order approximations to both systems will be the same.
Thus, we choose to work with the second formulation for ease of exposition. However, we point out  that the ideas used to arrive at equation \eqref{e:aux1} can be extended to more general convolution operators. We plan to tackle this more general framework using this approach in a future paper.

Continuing with the sketch of the proof of Theorem \ref{t:main}, to show existence of solutions to equation \eqref{e:aux2}, we use polar coordinates to represent the unknown variable as $w(r) = \rho(r) \rme^{\rmi \phi(r)}$, and then separate the equation into its real and imaginary parts. 
Following the results from \cite{doelman2005}, in Section \ref{s:normalform} we pick appropriate scalings and, using a regular expansion for both $\rho$ and $\phi$, we obtain a hierarchy of equations at different powers of the small parameter $\delta \sim \eps^{1/4}$. Our goal is to then show that this sequence of equations has solutions representing spiral waves. This is done in Section \ref{s:approximations}.

 In Subsection \ref{ss:one} we work with the order $\rmO(1)$ equation and show that the first order correction to the complex amplitude, $\rho_0$, converges quickly to a constant. Thus, the spiral wave pattern is described by the complex phase, $\phi$. In Subsection \ref{ss:delta2} we show that, not surprisingly, the first order correction term for the phase variable, $\phi_0$, satisfies the following viscous eikonal equation,
\begin{equation}\label{e:visc}
-\Omega = \partial_{rr} \phi_0 + \frac{1}{r} \partial_r \phi_0  + \beta ( \partial_r \phi_0)^2   +\beta g(r).
\end{equation}
which is a known phase dynamics approximation for
target patterns in oscillatory media.
While in the case of target patterns the inhomogeneity $g$ represents the impurity, or defect, that gives rise  these structures,
 in the case of spiral waves this perturbation comes from the first order correction to the amplitude, $\rho_0$.
  We show  that $g$ decays algebraically in the far field and satisfies $g(r) \sim \rmO(1/r^2)$ as $r \to \infty$.
We are then able to use the results from \cite{jaramillo2022can}, 
where it is shown that equation \eqref{e:visc} admits solutions representing target patterns.
That is, solutions satisfying $\nabla \phi_0 \to \kappa$ as $r$ goes to infinity. We show that the constant $\kappa>0$, indicating that the impurity $g$ acts as a pacemaker, producing traveling waves that move away from the core. Because we are working in a co-rotating frame, we conclude that these same solutions correspond to spiral waves patterns in our setting.

To complete our proof, we gather all terms of order $\rmO(\delta^3)$  into one system of equations, which we refer to as the {\it closing equations}. 
Following a similar approach as in \cite{jaramillo2022rotating}, in Section \ref{s:existence} we use the implicit function theorem to prove existence of solutions. As mentioned before, the  difficulty with using this strategy comes from the fact that the linear operators appearing in our equations are not invertible when viewed as  maps between standard Sobolev spaces. We overcome this difficulty by establishing Fredholm properties for these, and related operators, using carefully selected algebraically weighted spaces. This is done in Section \ref{s:fredholm}, where we also give a precise definition for the spaces we will be working with. A bordering lemma then allows us to recover the invertibility of the operator.

\subsection{Discussion}

Spiral wave patterns have been extensively studied experimentally \cite{muller1987,mueller1990, winfree1974, belmonte1997}, analytically \cite{kopell1981, greenberg1980, greenberg1981, hagan1982, scheel1998, kuramoto2003, ermentrout1994, bernoff1991, barkley2002,tyson1988}, and through simulations \cite{barkley1991, laing2005, kilpatrick2010, ermentrout1994} since the 1970s, and have been shown to exist in both  excitable and oscillatory media.
In this paper we focus on the latter case, where one can assume that the intrinsic dynamics of the system allows for the formation of a limit cycle via a Hopf bifurcation. While most analysis of spiral waves concentrate on systems where these intrinsic `oscillators' interact via local forms coupling, 
here we assume that these connections are long-ranged and can be modeled using a convolution operator.
Since the length scale of these patterns is small compared to the experimental set up, we pose the model equations on the whole plane.  

Within this context, meaning spatially extended oscillatory media, past results on existence of spiral waves assume that
coupling 
is well represented by the Laplace operator, even in the case of nonlocal coupling, see \cite{kopell1981, greenberg1980, greenberg1981, hagan1982, scheel1998}.
As mentioned above, this assumption then allows one to use tools from spatial dynamics to analyze these problems.
While these techniques have been extended to study nonlocal neural field models \cite{laing2003, coombes2005, bressloff2012}, these results apply only to nonlocal operators that have 
  fractional Fourier symbols. The key idea is that in this case, the model equations can be transformed into partial differential equations by preconditioning the system with an appropriate differential operator.
Because this assumption is very restrictive, our goal for this paper has been to develop an alternative method based on functional analysis, which can be adapted to more general convolution operators.

Our efforts, summarized in Theorem \ref{t:main}, together with the work presented in reference \cite{jaramillo2022rotating},  provide the first rigorous proof for the existence of small amplitude spiral waves in oscillatory media with nonlocal coupling. Theorem  \ref{t:main} also provides  first order approximations for the amplitude and wavenumber of the pattern, and is the first work to rigorously establish a connection between
 the wavenumber, $\kappa$, and  properties of the nonlocal coupling. In addition, notice that our expansion for  $\kappa$  matches  the results obtained by Hagan \cite{hagan1982} and Aguareles et al \cite{aguareles2023rigorous} in the limiting case when the nonlocal operator reduces to the Laplacian. More precisely, we find that the wavenumber is small beyond all orders of the bifurcation parameter, a result that, due to the connection between the wavenumer and speed of the wave, also applies to the parameter $\Omega$.

Regarding our first order approximations, notice that because we are working in the weak coupling regime, spiral waves solutions
to our nonlocal cGL equation \eqref{e:gl1} are characterized mainly by the complex phase $\phi$. Thus, the viscous eikonal equation \eqref{e:visc}, which describes first order corrections for this variable, plays a central role in our proof. 
Because the  parameter $\Omega$ is an unknown, this equation can be interpreted as a nonlinear eigenvalue problem.  The difficulty in solving this equation then comes from the fact that $\Omega$, as pointed out above, is small beyond all orders of the
bifurcation parameter. As a result one cannot use regular perturbation techniques to find solutions, and instead one has to find approximations to the pattern, both near the core  and in the far field.
This analysis was done in \cite{jaramillo2022can}, where it is shown that one can find the value of the parameter $\Omega$, and thus solve the nonlinear eigenvalue problem,  by matching these two approximations.

This split between near and far field  approximations is also necessary at the level of the {\it closing equations} (i.e. the  order $\rmO(\delta^3)$ system analyzed in Section \ref{s:existence}). However, instead of finding and then matching
these approximations,
 as was the strategy in references \cite{aguareles2023rigorous} and \cite{jaramillo2022can}, here we use a special class of doubly weighted Sobolev spaces, which are able to capture the behavior of solutions in these two regimes.
In particular, we use spaces that encode algebraic decay or growth properties of functions, both at infinity and near the origin. With these spaces we are then able to 1) describe the higher order correction terms to our spiral wave solution, 2) show that the closing equations define a bounded map, and 3) prove that the corresponding linear operators are invertible. As a result, we are then able to prove existence of solutions using the implicit function theorem.

Although the main goal of this paper was to show existence of spiral waves,
 a large portion of this article is devoted to establishing Fredholm properties for elliptic operators using weighted Sobolev spaces. These results are new and interesting in and on themselves and, as shown in this article and in references \cite{jaramillo2016, jaramillo2019, jaramillo2018, jaramillo2022can}, they are the key ingredient to proving existence of patterned solutions in spatially extended systems via perturbation methods.

Finally, let us point out that even though our assumptions on $K$ narrow down the type of nonlocal coupling covered by our proof,  this choice of convolution map  does fall under the broader family of operators that are of diffusive type.
We concentrate on kernels like those defined by Hypothesis [H2] for simplicity of exposition.
Using the approach given in the derivation of equation \eqref{e:aux2}, the methods developed here can be adapted to more general convolution operators, with Fourier symbols that are uniformly bounded, analytic, and that have a quadratic (or even higher order) tangency near the origin. These operators naturally appear in neural field models, so it is possible to modify the techniques presented here to these systems, provided one can write the modeling equations in the form \eqref{e:model}.
In particular this means assuming that the firing rate function is smooth (i.e. sigmoidal) and that the space-clamped system admits a Hopf bifurcation, see for instance \cite{kilpatrick2010} for an example of such a model. 
However, notice that although these assumptions are reasonable mathematically, it is not clear if the resulting system is physically relevant.

\section{Preliminaries}\label{s:fredholm}
In this section we first introduce weighted and doubly-weighted Sobolev spaces.
We then consider the main linear operators appearing in our proofs of existence (Section \ref{s:existence}) and establish their Fredholm properties.
In particular, in Subsection \ref{ss:fredholm} we show the connection between their Fredholm index and
the type of weighted Sobolev spaces used to define their domain and range.
On a first reading, one may skip  Subsection \ref{ss:fredholm} and refer back to it when needed.

\subsection{Sobolev Spaces}\label{ss:spaces}
Throughout this section the letters $d$ and $s$ represent non-negative integers, while $\gamma$ and $\sigma$ are real numbers. To define the norm of our weighted Sobolev spaces we use the symbols
\[\langle x \rangle = (1+ |x|^2)^{1/2}\quad \mbox{and} \quad m(|x|)=|x| (1- \chi(x)),\]
where $\chi(x) \in C^\infty(\R^d)$
denotes a smooth radial cut-off function satisfying $\chi(x) = 0$ for $|x|<1$ and $\chi(x) = 1$ for $|x|>2$.

\begin{Definition}
We define the weighted space $H^{s}_{\gamma}(\R^d)$ as the completion of $C_0^\infty(\R^d, \C)$ with respect to the norm
\[ \| u\|_{H^s_\gamma(\R^d)} = \sum_{|\alpha| \leq s} \|  D^\alpha u(x) \; \langle x \rangle^\gamma\|_{L^2(\R^d)}. \]
\end{Definition}

It is clear from this definition that $H^{s}_{\gamma}(\R^d)$ is a Hilbert space. Its inner product is given by
\[ \langle f, g\rangle = \sum_{|\alpha|<s} \int_{\R^d} D^\alpha f(x) D^\alpha \bar{g}(x) \langle x \rangle^{2\gamma} \;dx,\]
where the over-bar denotes complex conjugation. Since $H^{s}_{\gamma}(\R^d)$ is a reflexive space, one may also  identify its dual with  $ H^{-s}_{-\gamma}(\R^d)$.

Notice  that depending on the weight $\gamma$, the functions $H^{s}_{\gamma}(\R^d)$ are either allowed to grow ($\gamma< -1$) or decay ($\gamma>-1$) at infinity, see Figure \ref{f:decay}. Moreover, the embedding $H^s_{\gamma_1}(\R^d) \subset H^s_{\gamma_2} (\R^d)$ holds, provided $\gamma_1>\gamma_2$, while $H^s_\gamma(\R^d) \subset H^k_\gamma (\R^d)$ is valid whenever $s>k$.
In terms of notation, when $s=0$ we write $L^2_\gamma(\R^d)$ instead of $H^0_\gamma(\R^d)$.

\begin{Definition}\label{d:doubly-weighted}
We define the doubly-weighted space $ H^{s}_{\gamma, \sigma}(\R^d)$  as the completion of $C^\infty_0(\R^d,\C)$ with respect to the norm
\[ \| u(x) \|_{H^s_{\gamma,\sigma}(\R^d)} = \sum_{|\alpha|<s} \| m(x)^{\sigma+|\alpha| }  D^\alpha u(x) \;\langle x \rangle^\gamma  \|_{L^2(\R^d)}.\]
\end{Definition}
As in the previous case, the space $ H^{s}_{\gamma, \sigma}(\R^d)$ is a Hilbert space with inner product
\[ \langle f, g\rangle = \sum_{|\alpha|<s}  \int_{\R^d} m(x)^{2(\sigma + |\alpha|) } \; D^\alpha f(x) D^\alpha \bar{g}(x)\; \langle x \rangle^{2 \gamma} \;dx.\]
Its dual space satisfies $( H^{s}_{\gamma, \sigma}(\R^d))^* =  H^{-s}_{-\gamma, -\sigma}(\R^d)$, and the 
 embeddings  $H^s_{\gamma_1,\sigma} (\R^d) \subset H^s_{\gamma_2,\sigma} (\R^d)$ and $H^s_{\gamma,\sigma} (\R^d) \subset H^k_{\gamma,\sigma} (\R^d)$  hold, provided $\gamma_1>\gamma_2$ and $s>k$, respectively, while $H^s_{\gamma,\sigma_1} (\R^d) \subset H^s_{\gamma,\sigma_2} (\R^d)$ is valid whenever $\sigma_1< \sigma_2$.

As before, functions in these doubly-weighted spaces have a level of growth or decay  at infinity that is controlled by the weight $\gamma$.
However, in contrast to $H^s_\gamma(\R^d)$, functions in $ H^{s}_{\gamma, \sigma}(\R^d)$  are also allowed to grow near the origin, see Figure \ref{f:decay}. In particular,  elements in $L^2_{\gamma, \sigma}(\R^d)$ are allowed to have a singularity at the origin of order $\rmO(1/|x|^\alpha)$, with $ \alpha \in (0,\sigma + d/2)$, see Lemma \ref{l:decay_zero}. %

In addition to the above weighted Sobolev spaces, in this article we also use their restriction  to radially symmetric functions. This is summarized in the following definition.
\begin{Definition}
We denote by $H^s_{r,\gamma}(\R^d)$ and $H^s_{r, \gamma, \sigma}(\R^d)$ the subspaces of $H^s_{\gamma}(\R^d)$ and $H^s_{ \gamma, \sigma}(\R^d)$, respectively, consisting of radially symmetric functions.
\end{Definition}

\begin{Remark}
To avoid confusion, throughout the paper we always use $\gamma$ to denote the strength of the weight $\langle x \rangle$, while the symbol $\sigma$ will always represent the strength of the weight $m(|x|)$.
As a result, we will use $\gamma$ to encode growth or decay properties of functions at infinity, while the value of $\sigma$ will indicate 
the behavior  of functions near the origin.
 In addition, in what follows we will  often approximate $m(|x|) \sim |x|=$ when defining the norm $\| \cdot \|_{L^2_{\gamma, \sigma}(B_1)}$. Here $B_1$ represents the unit ball in $\R^2$.
\end{Remark}

Next, we present four lemmas summarizing  properties of functions in $H^k_{\sigma}(\R^2)$ and in $H^k_{\gamma, \sigma}(\R^2)$.

\begin{Lemma}\label{l:decay}
Let $\gamma \in \R$ and $d \in \N$. A function $f$ is in $L^2_\gamma(\R^d)$ if and only if there is a number $\alpha < - \gamma - d/2$ and a positive constant $C$, such that for a.e. $ x \in \R^d$
 $$|f(x)|\leq C |x|^{\alpha} \quad \mbox{as} \quad x \to \infty.$$
\end{Lemma}

\begin{proof}
Let $f: \R^d \longrightarrow \C$, and suppose there is a constant $C$, such that $|f(x)| < C |x|^{\alpha}$, with $\alpha $ as stated in the Lemma.  We may then write
\[
\| f\|_{L^2_\gamma(\R^d)}  = \int_{\R^d} |f(x)|^2 \langle x \rangle^{2 \gamma} \;dx
\leq C \int_1^\infty  r^{2\alpha} r^{ 2\gamma}\; r^{ d-1} \;dr.\]
Since this last integral is finite when $\alpha < - \gamma -d/2$, it then follows that $f \in L^2_\gamma(\R^2)$.

To prove the second implication we use its contrapositive. To that end, suppose now that $f(x)$ is such that $|f(x)|> |x|^{\beta}$ with $\beta = - \gamma -d/2+\eps$, and that this holds for all $\eps>0$. Then, there is a constant $c>0$ such that
\[ \infty = \int_1^\infty  r^{2\beta} r^{ 2\gamma}\; r^{ d-1} \;dr < c \int_{\R^d} |f(x)|^2 \langle x \rangle^{2 \gamma} \;dx= c \| f\|_{L^2_\gamma(\R^d)},\]
Consequently, $f$ is not in $L^2_\gamma(\R^d)$, and the results of the lemma then follow.
\end{proof}

A similar argument as in the proof of the last lemma gives us decay properties near the origin for elements in the space $L^2_{\gamma, \sigma}(\R^2)$.  This result is then summarized in Lemma \ref{l:decay_zero}. Examples illustrating the results of Lemma \ref{l:decay} and Lemma \ref{l:decay_zero}, for the case of functions defined in $\R^2$, are also summarized in Figure \ref{f:decay}.

\begin{figure}[t] 
   \centering
   \includegraphics[width=5in]{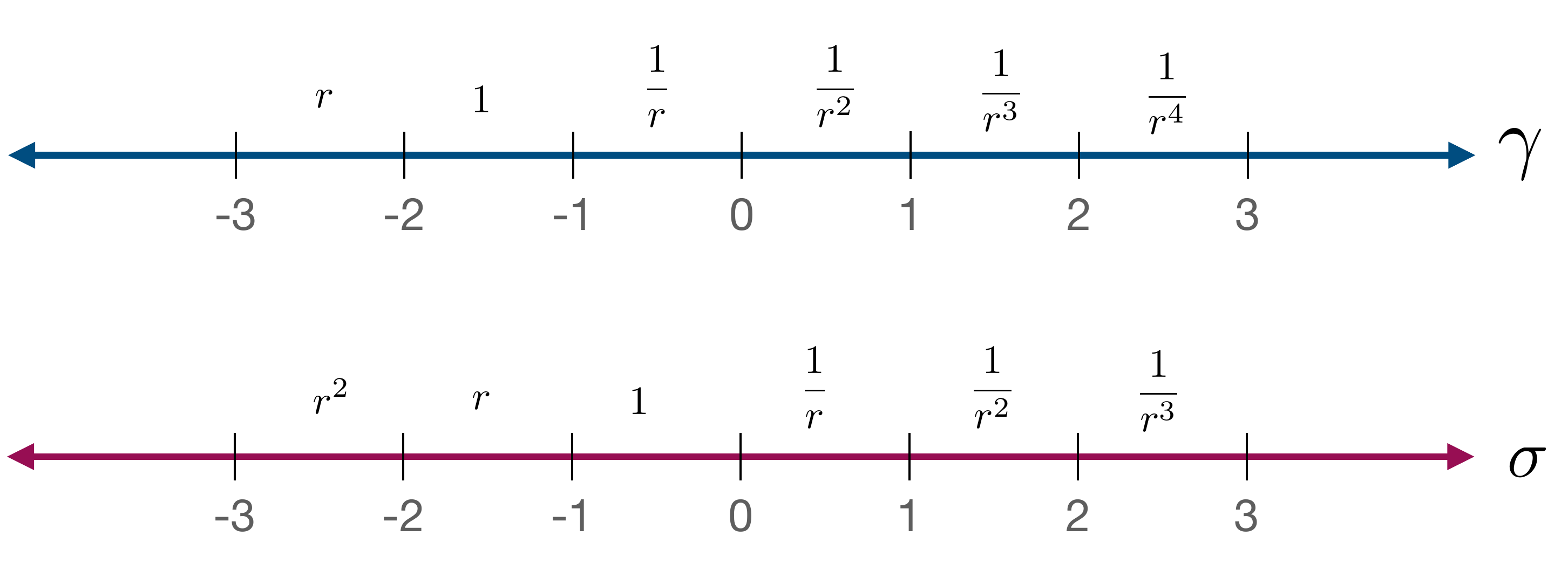} 
   \caption{Examples of algebraic decay/growth for weighted, $L^2_{\gamma}(\R^2)$, and doubly-weighted, $L^2_{\gamma, \sigma}(\R^2)$, Sobolev spaces. Parameter $\gamma$ encodes decay/growth properties of functions at infinity, while parameter $\sigma$ encodes decay/growth rates of functions near the origin. }
   \label{f:decay}
\end{figure}

\begin{Lemma}\label{l:decay_zero}
Let $\gamma, \sigma \in \R$, and $d \in \N$. A function $f$ is in $L^2_{\gamma,\sigma} (\R^2)$
 if and only if there is a number $\alpha > - \sigma - d/2$ and a positive constant $C$, such that for a.e. $x \in \R^d$,
 $$|f(x)|\leq C |x|^{\alpha} \quad \mbox{as} \quad |x| \to 0.$$

\end{Lemma}

The next lemma gives assumptions under which  the space $H^k_{\gamma, \sigma}(\R^2)$ is a Banach algebra.

\begin{Lemma}\label{l:banach_alg}
Let $k$ be an integer, and take $\gamma, \sigma \in \R$ with $\gamma>-1$ and $\sigma <-1$. Then $H^k_{\gamma, \sigma}(\R^2)$ is a Banach algebra.
\end{Lemma}

\begin{proof}
Suppose that $f$ and $g$ are in $H^k_{\gamma, \sigma}(\R^2)$. Since $\gamma> -1$, from Lemma \ref{l:decay} we know that functions in this space decay at infinity. Therefore, the inequality $|f g| < C |f|$ holds for large values of $|x|$ and some $C>0$. As a result, the product $fg \in L^2_{\gamma, \sigma}(\R^2\setminus B_1)$, where $B_1\subset \R^2$ represents the unit ball. A similar argument then shows that $D^\alpha(fg)$ is also in $L^2_{\gamma, \sigma +|\alpha|}(\R^2\setminus B_1)$, for all  indices $\alpha $ with $| \alpha| \leq k$.

On the other hand, to show that expression $fg$ satisfies the desired properties near the origin, we look at the
derivative $D^s(fg)$. Taking $s \in \Z \cap [0,k]$ and letting $D^sf $ any partial derivative of order $s$, we obtain
\[ D^s (fg) = \sum_{\ell =0}^s { s \choose \ell} D^\ell f D^{s -\ell} g.\]
From the definition of the spaces $H^k_{\gamma, \sigma}(\R^2)$ and Lemma \ref{l:decay_zero} we know that
near the origin,
\begin{align*}
| D^\ell f| \leq C |x|^{\alpha_1} & \qquad \alpha_1> -(\sigma+\ell +1),\\
| D^{s-\ell} g| \leq C |x|^{\alpha_2} & \qquad \alpha_2> -(\sigma+s -\ell +1),
\end{align*}
and as a result,
\[ \| D^\ell f D^{s -\ell} g\|_{L^2_{\gamma, \sigma+s}(B_1)}\leq \int_0^1 r^{2(\alpha_1+\alpha_2)} r^{2(\sigma+s)} r\;dr.\]
Since $\sigma<-1$, the inequality $2(\alpha_1+ \alpha_2) + 2(\sigma+s) +2 >0$ is satisfied. Consequently, this last integral is finite and we obtain $D^s(fg) \in L^2_{\gamma,\sigma+s}(\R^2)$, as desired. The results of the Lemma then follow.
\end{proof}

\begin{Lemma}\label{l:product_rule}
Let $k,s \in \N \cup \{0\}$ satisfying $s\leq k$,  take $\gamma_1, \gamma_2>-1$ and pick $\sigma_1, \sigma_2 >-1$. Suppose that $f \in H^k_{\gamma_1, \sigma_1}(\R^2)$ and $g \in H^s_{\gamma_2,\sigma_2}(\R^2)$, then the product
\[ fg \in H^s_{\gamma,\sigma}(\R^2)\]
where $\gamma = \min(\gamma_1,\gamma_2)$ and $\sigma = \sigma_1+ \sigma_2+1$.
\end{Lemma}

\begin{proof}
As in the previous lemma, since for $i \in \{1,2\}$ we have that $\gamma_i <-1$, both $f$ and $g$ decay at infinity. If, without loss of generality, we assume $\gamma_1< \gamma_2$, then the inequality $|fg|< |f|$ holds, and we conclude that the product
$fg \in L^2_{\gamma_1,\sigma}(\R^2\setminus B_1)$, where $\sigma$ is arbitrary for now. A similar reasoning also shows that $D^\alpha (fg) \in L^2_{\gamma_1,\sigma+|\alpha|}(\R^2 \setminus B_1)$.

On the other hand, if we now take $p \in \Z \cap [0,s]$ and let $D^p f $ denote any partial derivative of order $p$, 
we may write
\[ D^p (fg) = \sum_{\ell =0}^p { p \choose \ell} D^\ell f D^{p -\ell} g.\]
Using Lemma \ref{l:decay_zero} we again find that for $|x| \sim 0$,
\begin{align*}
| D^\ell f| \leq C |x|^{\alpha_1} & \qquad \alpha_1> -(\sigma_1+\ell +1),\\
| D^{s-\ell} g| \leq C |x|^{\alpha_2} & \qquad \alpha_2> -(\sigma_2+s -\ell +1).
\end{align*}
Letting $\sigma = \sigma_1 + \sigma_2 +1$, the above expressions, along with a similar reasoning as in the proof of Lemma \ref{l:decay}, then imply that the norm $| D^s(fg)\|_{L^2_{\gamma, \sigma+s}(B_1)}$ is finite. This follows since 
the inequality $2(\alpha_1 + \alpha_2) + 2(\sigma +s) +2 >0$ is satisfied.
\end{proof}

\subsection{Fredholm Operators}\label{ss:fredholm}

In this subsection we prove Fredholm properties for the
various differential operators that appear in this article.
Recall that a linear operator, $L$, is Fredholm if it has closed range, finite dimensional kernel, $\mathrm{Ker}(L)$, and finite dimensional co-kernel, $\mathrm{coKer}(L)$. Its index is then given by 
$i = $ dim $(\mathrm{Ker}(L)) -$ dim $ ( \mathrm{coKer}(L) )$.

We start with the operator $\Delta -\mathrm{Id}$, which is known to be invertible from $H^s(\R^2)$ into $H^{s-2}(\R^2)$. Lemma \ref{l:Delta-c}, shown below, extends this result to the weighted Sobolev spaces defined in Subsection \ref{ss:spaces}. A proof of this result can be found in \cite{jaramillo2018}, but for convenience we reproduce it in Appendix \ref{a:Fredholm}.

\begin{Lemma}\label{l:Delta-c}
Let $s\geq 2$ and suppose $\gamma \in \R$. Then, the operator
\[
  \Delta-\mathrm{Id}:  H^s_\gamma(\R^2) \longrightarrow H^{s-2}_\gamma(\R^2)\]
 is invertible.
\end{Lemma}

In the following corollary we view the Hilbert spaces $H^s_\gamma(\R^2)$  as  direct sums given by,
\[ H^s_\gamma(\R^2) = \oplus_n h^s_{n,\gamma},\]
where $n \in \Z$ and the spaces $h^s_{n,\gamma}$  are defined as
\[
 h^s_{n,\gamma} =  \;  \{ u \in H^s_\gamma(\R^2) : u(r, \theta) = u_n(r) \rme^{ \rmi n \theta}, \quad u_n(r) \in  H^s_{r,\gamma} (\R^2)\},\]

Heuristically, the above decomposition comes from a change of coordinates into polar coordinates and
a Fourier series decomposition of $u$ in the $\theta$ variable. For a complete description of this formulation,
see \cite{stein1971}. Recall also that we have assumed our weighted Sobolev spaces are composed of complex-valued functions.

\begin{Corollary}\label{c:Delta-c}
Let $n \in \N \cup\{0\}$, and take $\gamma \in \R$. Then, the operator
\[
 \Delta_n- \Id : H^s_{r,\gamma}(\R^2) \longrightarrow H^{s-2}_{r,\gamma}(\R^2)\]
 with $\Delta_n = \partial_{rr} + \frac{1}{r} \partial_r -\frac{n^2}{r^2} $
 is invertible.
\end{Corollary}
\begin{proof}
The result follows from the fact that the Laplacian commutes with rotations, so that if we view $H^s_\gamma(\R^2) = \oplus_n h^s_{n,\gamma}$, then the invertible operator $\Delta- \Id : H^s_\gamma(\R^2)\longrightarrow H^{s-2}_{\gamma}(\R^2)$ is a diagonal operator in this last space. That is, 
\begin{align*}
 (\Delta- \Id) u = & f\\
\sum_n (\Delta_n - \Id) u_n \rme^{\rmi n \theta} = &\sum_n f_n  \rme^{\rmi n \theta},
\end{align*}
where  by the definition of $h^s_{n, \gamma}$, the functions $u_n(r) \in H^s_{r,\gamma}(\R^2)$ and $f_n (r) \in H^{s-2}_{r,
\gamma}(\R^2)$.
It then follows that each operator $\Delta_n- \Id: H^s_{r,\gamma}(\R^2) \longrightarrow H^{s-2}_{r,\gamma}(\R^2)$ is invertible.
\end{proof}

In the next two lemmas we establish Fredholm properties for the radial operator $\mathcal{L} = \partial_r + \frac{1}{r} -\lambda$ with $\lambda>0$. At this stage it is convenient to restrict attention to real-valued functions. So in the rest of this section we assume that $H^s_{r,\gamma}(\R^2) = H^s_{r,\gamma}(\R^2,\R)$ and $H^s_{r,\gamma,\sigma}(\R^2) = H^s_{r,\gamma,\sigma}(\R^2,\R)$.

\begin{Lemma}\label{l:fredholm1}
Let $\gamma$ be a real number, $s$ an integer satisfying $s\geq 1$, and take $\lambda>0$.
Then, the operator $\mathcal{L}: H^s_{r,\gamma}(\R^2) \longrightarrow H^{s-1}_{r,\gamma}(\R^2)$, defined by
\[ \mathcal{L} u = \partial_r u + \frac{1}{r} u - \lambda u,\]
is Fredholm with index $i =-1$ and cokernel spanned by $\rme^{-\lambda r}$.
\end{Lemma}
\begin{proof}
We first show the results of the lemma for the case when $s =1$. 

We start by proving that the operator $\mathcal{L}: H^1_{r,\gamma}(\R^2) \longrightarrow L^2_{r,\gamma}(\R^2)$ has a trivial nullspace.
Indeed, one can check that the function $ u_0 = \frac{1}{r} \rme^{\lambda r}$ is the only function that satisfies $\mathcal{L} u_0 =0$.  However, because of its singularity at the origin and its exponential growth, this function does not belong to the space $H^1_{r,\gamma}(\R^2)$ for any weight $\gamma \in \R$.

A short calculation using integration by parts and the pairing between elements in $L^2_{r,\gamma}(\R^2)$ and its dual, $L^2_{r,-\gamma}(\R^2)$, shows that the adjoint of $\mathcal{L}$ is the operator $\mathcal{L^*}: L^2_{r,-\gamma}(\R^2) \longrightarrow H^{-1}_{r,-\gamma}(\R^2) $, given by $\mathcal{L}^* u = -( \partial_r u + \lambda u)$. The kernel of this operator is then spanned by $\rme^{-\lambda r}$, which is an element of the space $L^2_{r,-\gamma}(\R^2)$ for all $\gamma \in \R$. As a result the cokernel of $\mathcal{L}$ is spanned by $\rme^{-\lambda r}$ and therefore has dimension one.

To prove the results of the lemma, we are left with showing that the range of the operator $\mathcal{L} :H^1_{r,\gamma}(\R^2) \longrightarrow L^2_{r,\gamma}(\R^2)$ is given by 
\[ R = \{ u \in L^2_{r,\gamma}(\R^2) \mid \int_0^\infty u(r)\rme^{-\lambda r}  \; r \;dr = 0\}.\] 
Notice that because the function $\rme^{-\lambda r}$ is an element of  $L^2_{r,-\gamma}(\R^2) $, it defines a bounded linear functional, $\ell : L^2_{r, \gamma}(\R^2) \longrightarrow \R$, via the integral
\[ \ell(u) = \int_0^\infty u(r)\rme^{-\lambda r}  \; r \;dr.\]
Since the space $R$ corresponds to the nullspace of $\ell$, we immediately obtain that it is a closed subspace of $L^2_{r,\gamma}(\R^2) $.

To prove the above claim, we show that the inverse of $\mathcal{L}$, defined by the operator
\[\begin{array}{c c c}
R & \longrightarrow & H^1_{r, \gamma}(\R^2)\\
f(r) & \mapsto & u(r) =\frac{1}{r} \int_\infty^r  \rme^{-\lambda( s-r)} f(s) s \;ds
\end{array}
\]
 is bounded. We start by proving the bound $\|u \|_{L^2_{r,\gamma}(\R^2)} \leq \| f\|_{L^2_{r,\gamma}(\R^2)}$ using the inequality
\[ \|u \|_{L^2_{r,\gamma}(\R^2)} \leq \|u \|_{L^2_{r,\gamma}(B_1) } + \|u \|_{L^2_{r,\gamma}(\R^2\backslash B_1)}, \]
where $B_1$ is the unit ball in $\R^2$. First, notice that
\begin{align*}  
\|u \|_{L^2_{r,\gamma}(\R^2\backslash B_1)}  
= & \left [  \int_1^\infty \left| \int_\infty ^r f(s) \rme^{-\lambda(s-r)} \frac{s}{r} \;ds   \right|^2  \langle r \rangle^{2 \gamma } r\;dr \right]^{1/2}  \\
\leq &  2^{|\gamma -1 |}  \left [  \int_1^\infty \left| \int_\infty^r f(s)  \langle s \rangle^{\gamma  -1 } s \quad   \rme^{-\lambda(s -r)}  \langle s -r \rangle^{| \gamma -1  |}  r^{1/2} \;ds \right|^2  \;dr \right]^{1/2},\\
\leq &   2^{|\gamma -1 |}   \left [  \int_1^\infty \left|  \int_\infty^0 |f(z+r )| \langle z+r \rangle^{\gamma-1} (z+r )\;
  \rme^{-\lambda z}  \langle z \rangle^{|\gamma-1|} r^{1/2}
\;dz \right|^2\;dr \right]^{1/2}\\
\leq &  2^{|\gamma -1|} \int_0^\infty e^{-\lambda z} \langle z \rangle^{|\gamma-1|}
\left[   \int_1^\infty |f(z+r )|^2 \langle z+r \rangle^{2\gamma} r \;dr \right]^{1/2}  \;dz   \\
\leq & C_1(\lambda,\gamma)\| f\|_{L^2_{r,\gamma}(\R^2)},
\end{align*}
where the second line follows from 
the relation $\langle s \rangle^{-\eta} \langle r \rangle^ \eta \leq 2^{|\eta|} \langle s-r\rangle^{|\eta|}$,
and the approximation $\langle r\rangle^{2\gamma}/r^2 \sim \langle r\rangle^{2(\gamma-1)} $, which is valid given that $ r>1$.
The inequality on the third line comes from the change of variables $ z= s-r$, while the fourth line follows from Minkowski's inequality for integrals \cite[Theorem 6.19]{folland1999}. In the last line, we let $C_1(\gamma, \lambda) = 2^{|\gamma -1|} \int_0^\infty e^{-\lambda z} \langle z \rangle^{|\gamma-1|} \;dz$.

Next, we use the solvability condition $\int_0^\infty f(r) \rme^{-\lambda r} \;r\;dr =0$ to write the inverse as
\[ u(r) = \frac{1}{r} \int_0^r \rme^{-\lambda(s-r)} f(s) \;s \;ds\]
and bound the norm in $B_1$ as follows:
\begin{align*}
\|u \|_{L^2_{r,\gamma}( B_1)} 
= & \left [ \int_0^1 \left | \frac{1}{r} \int_0^r \rme^{-\lambda(s-r)} f(s) \;s \;ds \right|^2 \langle r \rangle^{2 \gamma} \;r \;dr \right]^{1/2}\\
\leq &  \left [ \int_0^1 \left | \frac{1}{r} \int_0^r \rme^{\lambda r} f(s) \;s \;ds \right|^2 \langle r \rangle^{2 \gamma} \;r \;dr \right]^{1/2}\\
\leq & \int_0^1 \left [  \int_s^1    |f(s)|^2 s^2 \; \rme^{2\lambda r} \langle r \rangle^{2\gamma} \frac{1}{r} \; dr    \right]^{1/2} \;ds\\
\leq &   \int_0^1 \rme^{ \lambda}  2^\gamma \left [  \int_s^1   \frac{1}{r} \;dr \right]^{1/2} |f(s)| \;s  \; ds\\
\leq &  2^\gamma \rme^{ \lambda}   \int_0^1 \left [  -\log(s) \right]^{1/2} |f(s)| \;s  \; ds\\
\leq &  2^\gamma \rme^{ \lambda}   \| f\|_{L^2_r(B_1)}\left( \int_0^1 |\log(s)| s\;ds\right)^{1/2}\\
\leq & C_2(\lambda, \gamma) \| f\|_{L^2_{r,\gamma}(B_1)},
\end{align*}
In this case, the second line comes from the inequality $0< -(s-r) < r$, the third line follows from applying Minkowski's inequality for integrals, and the second to last line from H\"older's inequality. In the last line, we also let $C_2(\lambda,\gamma) = 2^{\gamma -1}\rme^\lambda $. The result is,
\[ \|u \|_{L^2_{r,\gamma}(\R^2)} \leq ( C_1(\lambda, \gamma) + C_2(\lambda, \gamma) ) \| f\|_{L^2_{r,\gamma}(\R^2)}.\]

Our next step is to show that $\|\partial_r u \|_{L^2_{r,\gamma}(\R^2)} \leq \| f\|_{L^2_{r,\gamma}(\R^2)}$ using again the inequality
\[ \|\partial_r u \|_{L^2_{r,\gamma}(\R^2)} \leq \|\partial_r u \|_{L^2_{r,\gamma}(B_1) } + \| \partial_r u \|_{L^2_{r,\gamma}(\R^2\backslash B_1)}. \]
From  the relation $\partial_r u + \frac{1}{r} u - \lambda u = f$, it is easy to see that the second term on the right hand side of the inequality satisfies,
\[ \| \partial_r u \|_{L^2_{r,\gamma}(\R^2\backslash B_1)} \leq [(1+\lambda)( C_1(\lambda, \gamma) + C_2(\lambda, \gamma) +1 ]
\| f\|_{L^2_{r,\gamma}(\R^2)}.\]
To bound the term $\| \partial_r u \|_{L^2_{r,\gamma}(B_1) } $, we first observe that
\[  \frac{|u(r)|}{r}  \leq \frac{1}{r^2} \int_0^r |f(s)| s\;ds \leq  \frac{1}{r^2} \int_0^r |f(s)- f(y)|s \;ds + \frac{1}{2} |f(y)|,\]
where $y \in [0,r]$. 
Letting $B(y,2r)$ denote the ball of radius $2r$ centered at $y$, and  defining $\tilde{B} (y,2r)  = B(y,2r) \cap B(0,r) \subset \R^2$, the above inequality can also be written as,
\begin{equation}\label{e:boundu}
 \frac{|u(r)|}{r} \leq \frac{|\tilde{B}(y,2r)|}{2 \pi r^2} \left( \frac{1}{|\tilde{B}(y,2r)|} \int_{\tilde{B}(y,2r)} |f(s) -f(y)| s \;ds\;d \theta \right) + \frac{1}{2} |f(y)|.
 \end{equation}
with $|\tilde{B}(y,2r)|$ representing the measure of the set $\tilde{B}(y,2r)$.
Because of the embedding $L^2(B(0,M)) \subset L^1(B(0,M))$, which holds  for any bounded ball $B(0,M) \subset \R^2$, the function $f$ is in $L^1_{loc}(B(0,M))$.
As a result, by the Lebesgue Differentiation Theorem \cite[Theorem 3.21]{folland1999}, 
the expression in parenthesis approaches zero as $r\to 0$, while the fraction in front remains bounded,
since $|\tilde{B}(y,2r)|\geq 2\pi r^2$. 
Thus, near  the origin, the function $|u(r)|/r$ is bounded above by $|f(r)|$. From the equation $\partial_r u = f - \lambda u - u/r$, it then follows that
 \[ \| \partial_r u \|_{L^2_{r,\gamma}(B_1)} \leq (2 +\lambda C_1(\lambda, \gamma))\; \| f\|_{L^2_{r,\gamma}(\R^2)}.\]
Consequently, for $f \in R$, the solution to $(\partial_r + \frac{1}{r} - \lambda) u = f$ satisfies,
\[ \| \partial_r u\|_{L^2_{r,\gamma}(\R^2)} \leq  (3 + 2(1+\lambda)(C_1(\gamma, \lambda) + C_2(\gamma))) \| f\|_{L^2_{r,\gamma}(\R^2)}.\]

To prove that the general operator $\partial_r + \frac{1}{r} -\lambda : H^s_{r,\gamma}(\R^2) \longrightarrow H^{s-1}_{r,\gamma}(\R^2)$ is Fredholm index $i =-1$, one can proceed by induction:
 Assuming that $f$ and $u$ are in $ H^{s-1}_{r,\gamma}(\R^2)$  one shows that $\partial_r^s u $ is in $L^2_{r,\gamma}(\R^2)$ using the relation $\partial^s_r u = \partial_r^{s-1} (f +\lambda u - \frac{u}{r})$. The fact that $\partial_r^{s-1} \left(\dfrac{u}{r}\right)$ is in the correct space follows by a similar argument as the one done above to prove $u/r \in L^2_{r,\gamma}(\R^2)$.
\end{proof}

\begin{Lemma}\label{l:fredholm2}
Let $\gamma , \sigma$ be real numbers, $s$ and integer satisfying $s \geq 1$, and take $\lambda >0$. Then, the operator $\mathcal{L}: H^s_{r,\gamma,\sigma}(\R^2) \longrightarrow H^{s-1}_{r,\gamma,\sigma+1}(\R^2)$, defined by
\[ \mathcal{L} u = \partial_r u + \frac{1}{r} u - \lambda u,\]
is Fredholm. Moreover,
\begin{enumerate}[i)]
\item if $\sigma <0$, it has index $i =-1$. It is injective and its cokernel is spanned by $\rme^{-\lambda r}$;
\item if $\sigma >0$, it has index $i =0$ and it is invertible.
\end{enumerate}
\end{Lemma}
\begin{proof}
We first concentrate in the case when $s =1$. As in  Lemma \ref{l:fredholm1}, the kernel and cokernel of the operator are spanned by $e^{\lambda r}/r$ and $e^{-\lambda r}$, respectively. 
Because the function 
$e^{\lambda r}/r$ grows exponentially, it does not belong to the domain of the operator no matter what the values of $\sigma$ and $\gamma$ are. Therefore the kernel of the operator is trivial. On the other hand, the function $e^{-\lambda r}$ belongs to  $L^2_{r,-\gamma,-(\sigma +1)}(\R^2)$ and  thus, it is  in the co-kernel of the operator provided $\sigma <0$, see Lemma \ref{l:decay_zero} and Figure \ref{f:decay}.

In case $i)$ the results of this lemma follow a similar argument as the ones used to prove Lemma \ref{l:fredholm1} above. 
The main difference comes from bounding the solution $u$ near the origin. Therefore, we concentrate only on proving this result.

Suppose then that  $f \in L^2_{r,\gamma,\sigma+1}(\R^2)$ with $\sigma<0$. We first want to establish the bound $\| u\|_{L^2_{r,\gamma,\sigma}(B1)} \leq \| f\|_{L^2_{r,\gamma,\sigma+1}(\R^2)}$. Using the solvability condition, we may  write the inverse as
\[ u(r) = \frac{1}{r} \int_0^r \rme^{-\lambda(s-r)} f(s) \;s \;ds.\]
The expression in \eqref{e:boundu} then shows that the function $|u(r)|/r$ is bounded above by $|f(r)|$ provided $f(r)  \in L^1(B_1)$. One can then show that these last condition is satisfied if $\sigma <0$. 

Next, we show that the derivative, $\partial_r u$, satisfies the bound $\| \partial_r u\|_{L^2_{r,\gamma,\sigma+1}(B_1)} \leq \| f\|_{L^2_{r,\gamma,\sigma+1}(\R^2)}$. Using the equation we may write $\partial_r u = f - \frac{1}{r} u + \lambda u$, to obtain,
\begin{align*} 
| \partial_r u| & \leq | f|  +  \frac{1}{r} | u|  + \lambda | u|\\
\| \partial_r u\|_{L^2_{r,\gamma,\sigma+1}(B_1)}  
& \leq  \| f \|_{L^2_{r,\gamma,\sigma+1}(B_1)}   + \|  u\|_{L^2_{r,\gamma,\sigma }(B_1)} + \|  u\|_{L^2_{r,\gamma,\sigma+1 }(B_1)}\\
 & \leq (2+\lambda) \| f \|_{L^2_{r,\gamma,\sigma+1}(\R^2)},
\end{align*}
where in the last line we used the embedding $L^2_{r,\gamma,\sigma }(B_1) \subset L^2_{r,\gamma,\sigma+1}(B_1)$ and the previous bound for the norm of $u$.

In case $ii)$ we no longer have a solvability condition, and the inverse is given by
\[ u(r) = \frac{1}{r} \int_\infty^r \rme^{-\lambda(s-r)} f(s) \;s \;ds,\]
As in the previous case, the same arguments as in Lemma \ref{l:fredholm1} give us the bounds for 
$\| u\|_{L^2_{r, \gamma, \sigma}(\R^2\setminus B_1)}$. 
To bound the solution near the origin we take a different approach. Letting
$g(r) = \rme^{-\lambda s} f(s)s$ and writing $\partial_r v = g$, 
our goal is to show that $\| v\|_{L^2_{r,\gamma, \sigma-1}(B_1) } \leq \| g\|_{L^2_{r,\gamma,\sigma}(\R^2)}$. Then,
\[\| u\|_{L^2_{r,\gamma, \sigma}(B_1) } \leq  \| v/r\|_{L^2_{r,\gamma, \sigma}(B_1) } =   \| v\|_{L^2_{r,\gamma, \sigma-1 }(B_1) } \leq \| g\|_{L^2_{r,\gamma,\sigma}(\R^2)} \leq 
 \| f\|_{L^2_{r,\gamma,\sigma+1}(\R^2)},\]
as desired.

Consider then the change of variables $\tau = \ln r$ for $r \in (0,1)$ and define 
\[ w(\tau) = v(e^\tau) e^{-\beta \tau}\qquad h(\tau) = g(e^\tau) e^{-\beta \tau} e^\tau,\]
with $\beta = -\sigma <0$. Then, $\partial_\tau( w(\tau) e^{\beta \tau} ) = h(\tau) e^{\beta \tau}$, and a straightforward computation using the above change of variables shows that  $\| v\|_{L^2_{r,\gamma,\sigma-1}(B_1)} = \| w\|_{L^2(-\infty,0)}$. In addition,
\[w(\tau) = \int_{\infty}^\tau h(s) e^{\beta(s -\tau)} \;ds.\]
To bound the $L^2$ norm of $w$, we can then write
\begin{align*}
\| w\|^2_{L^2(-\infty,0)} & =
\int_{-\infty}^0 \left[   \int_{\infty}^\tau h(s) e^{\beta(s -\tau)} \;ds  \right]^2 \;d\tau\\
& \leq
\int_{-\infty}^0 \left[   \int_{\infty}^0 h(z+\tau) e^{\beta z } \;dz   \right]^2 \;d\tau\\
\| w\|_{L^2(-\infty,0)} & \leq
 \int^{\infty}_0 \left[   \int_{-\infty}^0 |h(z+\tau)|^2 e^{2\beta z } \;d\tau   \right]^{1/2} \;dz\\
& \leq  \int^{\infty}_0 e^{\beta z} \| h\|_{L^2(\R^2)} \;dz \\
& \leq C(\beta)  \| h\|_{L^2(\R^2)}.
\end{align*}
Here, the second line comes from the change of coordinates $ z = s+\tau$, the third line follows from Minkowski's inequality for integrals  \cite[Theorem 6.19]{folland1999}, and the fourth line is a consequence $\beta = - \sigma <0$.

To bound $\| h\|_{L^2(\R)}$ notice that
\[ \| h\|_{L^2(\R)}  = \int_{-\infty}^\infty |g(e^\tau)|^2 e^{-2\beta \tau} e^{2 \tau} \;d\tau\\
\leq  \int_{0}^\infty e^{-2\lambda r} |f(r)|^2 r^{2\sigma+2} r \;dr 
\leq C(\lambda,\gamma) \| f\|_{L^2_{r,\gamma,\sigma+1}(\R^2)},\]
and we may indeed conclude that $\| u\|_{L^2_{r,\gamma, \sigma}(B_1) } \leq \| f\|_{L^2_{r,\gamma,\sigma+1}(\R^2)}$. 
To bound the derivative 
 $\partial_r u$, one follows the approach used  in case $i)$.

Finally, as in Lemma \ref{l:fredholm1}, to prove that the general operator $\partial_r + \frac{1}{r} -\lambda : H^s_{r,\gamma,\sigma}(\R^2) \longrightarrow H^{s-1}_{r,\gamma,\sigma+1}(\R^2)$ is Fredholm index $i =-1$, one can proceed by induction.
\end{proof}

 
In the Section \ref{s:normalform} we find that the normal form describing the evolution of one-armed  spiral waves involves the linear radial operator
$(\Delta_1 - \Id) u= \partial_{rr} u + \frac{1}{r} \partial_ru - \frac{1}{r^2} u - u$.
 Corollary \ref{c:Delta-c} showed that the operator is invertible when defined over the weighted spaces $H^s_{r,\gamma}(\R^2)$.  The following proposition shows that this operator is also invertible over the double weighted Sobolev spaces $H^s_{r,\gamma, \sigma}(\R^2)$, provided the weights $\sigma $ are chosen appropriately.

\begin{Proposition}\label{p:fredholm1}
Let $\gamma \in \R$, $\sigma \in (-2,0)$ and take $s\geq 2$ to be an integer. Then, the operator 
$ \Delta_1 - \Id : H^s_{r, \gamma, \sigma}(\R^2) \longrightarrow H^{s-2}_{r,\gamma, \sigma +2}(\R^2)$, defined by
\[ (\Delta_1 - \Id) u = \partial_{rr} u + \frac{1}{r} \partial_ru - \frac{1}{r^2} u - u,\]
is invertible.
\end{Proposition}

\begin{proof}
The proof of this proposition is carried out in a number of steps. 
In Step 1 we first determine the elements in the kernel and cokernel of the operator.
In Step 2 we find the Green's function for the operator.
Finally, in Steps 3 and 4 we show that for values of $ \sigma \in (-2,0) $, and for $s =2$, the operator has a bounded inverse from $H^{s-2}_{r,\gamma, \sigma}(\R^2)$ into $H^s_{r,\gamma, \sigma}(\R^2)$. In Step 5 we use an induction argument to prove the results for integer values $s >2$.

{\bf Step 1:} 
Notice first that the equation $(\Delta_1 -\Id)u =0$, represents the modified Bessel equation of order one. It has as solutions the Modified Bessel functions $K_1(r), I_1(r)$, which satisfy the decay/growth properties summarized in Table \ref{t:bessel2}.
Since $I_1(r)$ grows exponentially, this function is not in any of the spaces $H^2_{r,\gamma, \sigma}(\R^2)$. On the other hand, since $K_1(r) \sim \rmO(1/r)$ near the origin, a short calculations shows that this function belongs to $H^2_{r,\gamma, \sigma}(\R^2)$
provided $\sigma >0$. See also Lemma \ref{l:decay_zero} and Figure \ref{f:decay}.

\begin{table}[t]
\begin{center}
\begin{tabular}{ m{2cm} m{5cm} m{4.5cm}  } 
\specialrule{.1em}{.05em}{.05em} 
  & $z \to 0$ & $z \to \infty $\\
  \hline
  \vspace{1ex}
$ K_\nu(z) $ &$ \sim \frac{1}{2}\Gamma(\nu)\left[\frac{z}{2} \right]^{-\nu}$ &$ \sim \sqrt{ \frac{\pi}{2 z} } \rme^{-z} $\\ 
$ I_\nu (z) $ & $ \sim \frac{1}{\Gamma(\nu+1)} \left[ \frac{z}{2} \right]^\nu $ &$\sim \sqrt{ \frac{1}{2 z \pi} } \rme^{z}  $\\      
\specialrule{.1em}{.05em}{.05em}
\vspace{1ex} 
  $\frac{d}{dz} \mathcal{Z}_\nu(z) = $ &$ \mathcal{Z}_{\nu+1}(z) +   \frac{\nu}{z} \mathcal{Z}_\nu(z)$&  \\
\specialrule{.1em}{.05em}{.05em} 

\end{tabular}
\end{center}
\caption{ 
Asymptotic behavior for the first-order Modified Bessel functions of the first and second kind, and their derivatives, taken from \cite[(9.6.8), (9.6.9), (9.7.2)]{abramowitz}. Here, the integer $\nu\geq1$, $\Gamma(\nu)$ is the Euler-Gamma function, and $\mathcal{Z}_\nu =\{ I_\nu(z), e^{i \pi \nu} K_\nu(z)\}$.}
\label{t:bessel2}
\end{table}

Using integration by parts, one can check that the adjoint of the operator
 is given by 
 \[
 \begin{array}{ r c l}
 \Delta_1 - \Id : L^{2}_{r,-\gamma, -(\sigma +2)}(\R^2)& \longrightarrow & H^{-2}_{r, -\gamma, -\sigma}(\R^2)\\
 \end{array}
 \]
It then follows that only $K_1(r)$ is an element of the co-kernel, and this holds for values
of $\sigma< -2$.

{\bf Step 2:}
The Green's function, $G(r, \rho)$, for the operator $\Delta_1 -\Id$
satisfying $G(r, \rho) \to 0$ as $r \to \infty$, and $G(0,\rho) =0$  on the interval $[0,\infty)$ is 
\[ G(r, \rho) = \left \{
\begin{array}{c c c}
\frac{I_1(r) K_1(\rho)}{W(r)} & \mbox{for} & 0< r< \rho,\\
\frac{K_1(r) I_1(\rho)}{W(r)} & \mbox{for} & \rho< r< \infty.
\end{array}
\right.
\]
Here, the function $W$ represents the Wronskian $W(r) = I'_1(r) K_1(r)- I_1(r) K'_1(r) = \dfrac{1}{r}$.
It then follows that the formal inverse of the operator $\Delta_1 -\Id$ is given by
\[
 \begin{array}{ c c c}
 L^{2}_{r,\gamma, (\sigma +2)}(\R^2)& \longrightarrow & H^{2}_{r, \gamma, \sigma}(\R^2)\\
 f(r) & \mapsto & u(r) = \int_0^\infty G(r; \rho) f(\rho) \;d\rho .
 \end{array}
 \]
More precisely,
\begin{equation}\label{e:inverseu}
 u(r) = I_1(r) \int_r^\infty K_1(\rho) f(\rho) \rho \;d\rho +
  K_1(r) \int_0^r I_1(\rho) f(\rho) \rho \;d\rho.
  \end{equation}

In Steps 3 and 4, we use expression \eqref{e:inverseu} to show that the bound $\| u\|_{H^2_{r,\gamma, \sigma}(\R^2)} \leq \|f \|_{L^2_{r,\gamma,\sigma +2}(\R^2)}$ holds for $-2< \sigma <0 $. Throughout, we use the inequality
\[ \| \cdot \|_{L^2_{\gamma, \sigma}(\R^2)} \leq \| \cdot \|_{L^2_{\gamma, \sigma}(B_1)}  + \| \cdot \|_{L^2_{\gamma, \sigma}(\R^2 \setminus B_1)}.\]
where $B_1$ represents the unit ball centered at the origin.
In this case,
$ \| \cdot \|_{L^2_{\gamma, \sigma}(\R^2 \setminus B_1)}  =  \| \cdot \|_{L^2_{\gamma}(\R^2 \setminus B_1)}.$ 

{\bf Step 3:} We start by bounding the $L^2_{r,\gamma,\sigma} (\R^2\setminus B_1)$ norm of $ u, \partial_r u,$ and $\partial_{rr} u$. Using expression \eqref{e:inverseu}
we have that
\[ \| u \|^2_{L^2_{\gamma, \sigma}(\R^2 \setminus B_1)} \leq A + B\]
where
\begin{align*}
A = & 
 \int_1^\infty \left[ I_1(r) \int_r^\infty K_1(\rho) f(\rho) \rho \;d\rho \right]^2 \langle r \rangle^{2\gamma} \;r dr  \\
B = &   \int_1^\infty \left[  K_1(r) \int_0^r I_1(\rho) f(\rho) \rho \;d\rho \right]^2 \langle r \rangle^{2\gamma} \;r dr 
\end{align*}
 Since $1<r<\rho$, we can use the decay/growth properties of $I_1(r)$ and $K_1(\rho)$, as $r, \rho \to \infty$,
and write
\begin{align*}
A & \leq C \int_1^\infty \left[ \int_r^\infty \frac{ \rme^{-( \rho - r)} }{\sqrt{r} \sqrt{\rho}}  f(\rho) \rho \;d\rho \right]^2 \langle r \rangle^{2\gamma} \;r dr  \\
&\leq C \int_1^\infty \left[ \int_r^\infty  \rme^{-( \rho - r)}   f(\rho)  \langle r -\rho \rangle^{|\gamma|} \langle \rho \rangle^\gamma \rho^{1/2} \;d \rho \right]^2  dr  \\
&\leq C \int_1^\infty \left[ \int_0^\infty  \rme^{-z}   f( z+r)  \langle z \rangle^{|\gamma|} \langle z+ r \rangle^\gamma (z +r) ^{1/2} \;d z \right]^2  dr,
\end{align*}
where the second line follows from the inequality $\langle r \rangle^{\gamma} \langle \rho \rangle^{-\gamma} \leq \langle r- \rho \rangle^{|\gamma|}$, and the third line comes from
the change of coordinates $z = \rho -r$.
Applying Minkowski's inequality for integrals, we then obtain
\begin{align*}
A^{1/2} &\leq C \int_0^\infty \left[ \int_1^\infty  \rme^{-2z}  | f( z+r)|^2  \langle z \rangle^{2 |\gamma|} \langle z+ r \rangle^{2\gamma} (z +r)  \;d r \right]^{1/2}  dz  \\
& \leq  C \int_0^\infty e^{-z} \langle z \rangle^{|z|} \| f\|_{L^2_{r, \gamma}(\R^2\setminus B_1)} \;d z\\
& \leq C(\gamma)  \| f\|_{L^2_{r, \gamma, \sigma+2}(\R^2)}.
\end{align*}

A similar analysis also shows that
$ B^{1/2} \leq C(\gamma) \| f\|_{L^2_{r, \gamma, \sigma+2}(\R^2)}.$
It therefore follows that 
$$\| u\|_{L^2_{r,\gamma, \sigma}(\R^2 \setminus B_1)} \leq C(\gamma) \| f\|_{L^2_{r,\gamma,\sigma+2}(\R^2)}.$$

To bound $\|\partial_r u\|_{L^2_{r,\gamma, \sigma +1}(\R^2\setminus B_1)}$,
we first compute this derivative using expression \eqref{e:inverseu}:
\begin{equation}\label{e:u_der}
 \partial_r u = 
 I'_1(r) \int_r^\infty K_1(\rho) f(\rho) \rho \;d\rho   +  K'_1(r) \int_0^r I_1(\rho) f(\rho) \rho \;d\rho. 
 \end{equation}
Since $I'_1(r) $ and $K'_1(r)$ have the same growth/decay properties as $I_1(r)$ and $K_1(r)$, respectively, the same computations as above show that 
$$\| \partial_r u\|_{L^2_{r,\gamma, \sigma}(\R^2 \setminus B_1)} \leq C(\gamma) \| f\|_{L^2_{r,\gamma,\sigma+2}(\R^2)}.$$

To bound the second derivative $\partial_{rr} u$ in the same norm, we use the equation and write
\[ \partial_{rr} u = f - \frac{1}{r} \partial_r u + \frac{1}{r^2} u + u,\]
arriving at the desired result,
$$\| \partial_{rr} u\|_{L^2_{r,\gamma, \sigma}(\R^2 \setminus B_1)} \leq 4 C(\gamma) \| f\|_{L^2_{r,\gamma,\sigma+2}(\R^2)}.$$

{\bf Step 4:} To complete the proof of the proposition, we need to now bound the norms $\| u\|_{L^2_{r,\gamma, \sigma}(B_1)} $, $\|\partial_r u\|_{L^2_{r,\gamma, \sigma+1}(B_1)} $, and $\| \partial_{rr} u\|_{L^2_{r,\gamma, \sigma+2}(B_1)} $.

We start with the bound $\| u\|_{L^2_{r,\gamma, \sigma}(B_1)} \leq \|f\|_{L^2_{r,\gamma, \sigma+2}(\R^2)}$.
Recalling the expression for $u$, along with the decay properties for $I_1(r)$ and $K_1(r)$ near the origin, one can approximate the first integral in \eqref{e:inverseu}
by
\[ A(r)  =  I_1(r) \int_r^\infty K_1(\rho) f(\rho) \rho \;d\rho  \leq C r \int_r^\infty  K_1(\rho) f(\rho) \rho \;d\rho . \]
With $g(\rho) = K_1(\rho) f(\rho) \rho$,  in what follows we show that $v(r) = \int^{\infty}_r g(\rho) \;d\rho$ satisfies $\| v\|_{L^2_{r,\gamma, \sigma+1}(B_1)} \leq C \| f\|_{L^2_{r,\gamma, \sigma+2}(\R^2)}$ for some generic positive constant $C$. 
As a result
\begin{equation}\label{e:boundu_origin}
 \| A \|_{L^2_{r,\gamma,\sigma}(B_1) } \leq C \|  r v(r) \|_{L^2_{r,\gamma,\sigma}(B_1)} \leq  C \|  v(r) \|_{L^2_{r,\gamma,\sigma+1}(B_1)} \leq C  \|  f(r) \|_{L^2_{r,\gamma,\sigma+2}(\R^2)}.
 \end{equation}

Recalling that $\partial_r v = g$, letting $ \tau = \ln r$ with $r \in(0,1)$, and defining
\[ w(\tau)= v(e^\tau) e^{-\beta \tau} \qquad h(\tau) = g(e^\tau)e^{-\beta \tau} e^\tau\]
we find that $\partial_\tau (w(\tau)e^{\beta \tau} ) = h(\tau) e^{\beta \tau} $. Taking $\beta = -(\sigma +2)<0$, a short calculation shows that
\[ w(\tau ) = \int_\infty^\tau h(s) e^{\beta(s-\tau)}\;ds\]
and that $\| w\|_{L^2(-\infty,0)} = \| v\|_{L^2_{r,\gamma,\sigma+1}(B_1)}$.
We can then bound
\begin{align*}
\| w\|^2_{L^2(-\infty,0)} & =
\int_{-\infty}^0 \left[ \int_\infty^\tau h(s) e^{\beta(s-\tau)}\;ds \right]^2 \; d\tau\\
& \leq
\int_{-\infty}^0 \left[ - \int^\infty_0 h(z+\tau) e^{\beta z}\;dz \right]^2 \; d\tau\\
\| w\|_{L^2(-\infty,0)} & \leq
\int^{\infty}_0 e^{\beta z} \left[ \int_{-\infty}^0 |h(z+\tau)|^2 \;d\tau \right]^{1/2} \; dz\\
&\leq
C(\beta) \| h\|_{L^2(\R)}
\end{align*}
where the second inequality comes from the change of coordinates $z = s- \tau$,
the third line is a consequence of Minkowski inequality for integrals \cite{}, and
the last line follows from our assumption $\beta = -(\sigma +2) <0$.

Lastly, notice that
\begin{align*}
 \| h\|_{L^2(\R)}  & = 
 \int_{-\infty}^\infty |h(\tau)|^2 \; d\tau\\
 & =
 \int_{-\infty}^\infty |g(e^\tau)|^2 e^{-2\beta \tau} e^{2 \tau} \; d\tau\\
 & \leq
 \int_0^1 |f(r)|^2 r^{2(\sigma+2)} r\;dr + \int_1^\infty |g(r)|^2 \langle r \rangle ^{2(\sigma+2)} r\;dr\\
 & \leq
 C \| f\|_{L^2_{r,\gamma, \sigma+2}(\R^2)}.
  \end{align*}
Here the third line follows from the change of variables $ r = e^\tau$ and the fact
that $g(r) \sim \rmO(f(r))$ for $r \sim 0$, while 
the last line follows from the definition of $g(r)$ and the decay properties of the
Bessel function $K_1(r)$ as $r \to \infty$.

On the other hand, for the second integral on expression \eqref{e:inverseu} we may write
\begin{equation}\label{e:B}
 B = K_1(r) \int_0^r I_1(\rho) f(\rho) \rho \;d\rho \leq C \frac{1}{r} \int_0^r  f(\rho) \rho^2 \;d\rho. 
 \end{equation}
As in the proof of Lemma \ref{l:fredholm1} we have that for $s \in [0,r]$
\[ |B|/ r \leq \frac{1}{r^2} \int_0^r | f(\rho) \rho - f(s) s| \rho \;d\rho + \frac{1}{2} |f(s) s|. \]
Letting $B(s,2r)$ denote the ball  of radius $2r$ centered at $s$, and writing 
$\tilde{B}(s, 2r) =  B(s,2r) \cap B(0,r)$, we then obtain
\begin{equation}\label{e:inequ_origin}
 |B|/ r \leq \frac{| \tilde{B}(s, 2r) |}{2 \pi r^2} \left( \frac{1}{| \tilde{B}(s, 2r)|} \int_{\tilde{B}(s, 2r)} | f(\rho) \rho - f(s) s| \rho \;d\rho \right) + \frac{1}{2} |f(s) s|. 
 \end{equation}
Since $f \in L^2_{r,\gamma,\sigma+2}(\R^2)$ with $\sigma<0 $, 
it follows $f(\rho)  \in L^1(B_1)$.
We may therefore apply the Lebesgue Differentiation Theorem to conclude that as $r$ goes to the origin the
term inside the parenthesis goes to zero, while the fraction in front of it remains bounded. 
Consequently $|B(r)|< r^2 f(r)$ and it follows that
$ \| B \|_{L^2_{r,\gamma,\sigma}(B_1) } \leq C  \|  f(r) \|_{L^2_{r,\gamma,\sigma+2}(\R^2)}.$
This bound together with \eqref{e:boundu_origin} leads to
\[ \| u\|_{L^2_{r,\gamma, \sigma}(B_1)} \leq \|f\|_{L^2_{r,\gamma, \sigma+2}(\R^2)}.\]

To obtain the bound $\| \partial_r u\|_{L^2_{r,\gamma, \sigma}(B_1)} \leq \|f\|_{L^2_{r,\gamma, \sigma+2}(\R^2)},$ recall first the expression for $\partial_r u$ written in \eqref{e:u_der}. Given that 
  $$\frac{d}{dz} \mathcal{Z}_\nu(z) =  \mathcal{Z}_{\nu+1}(z) +   \frac{\nu}{z} \mathcal{Z}_\nu(z)$$ 
where $\mathcal{Z}_\nu(z)$ represents either Bessel function, $I_\nu(z)$ or $K_\nu(z)$, it follows from 
the decay properties summarized in Table \ref{t:bessel2} that
$I'_1(r) \sim \rmO(1)$ and $K'_1(r) \sim \rmO(1/r^2)$ near the origin.
Therefore, we may bound
\[ | \partial_r u|  \leq C \int_r^\infty g(\rho) \;d \rho + \frac{1}{r^2} \int_0^r f(\rho) \rho^2 \; d\rho,
\]
where again $g(\rho) = K_1(\rho) f(\rho) \rho$.

To complete our argument, we define
\[ E(r) = C \int_r^\infty g(\rho) \;d \rho, \qquad F(r)=  \frac{1}{r^2} \int_0^r f(\rho) \rho^2 \; d\rho.\]
To bound $\| E\|_{L^2_{r,\gamma,\sigma+1}(\R^2)}$ one can follow the same approach as was
done in the case of $A(r)$, see inequality \eqref{e:boundu_origin}. While the bound for $\| F\|_{L^2_{r,\gamma,\sigma+1}(\R^2)}$
comes from a similar approach as inthe case of $B(r)$, see inequalities \eqref{e:B} and \eqref{e:inequ_origin}.
It then follows that
\[  \| \partial_r u \|_{L^2_{r,\gamma, \sigma +1}(B_1) } \leq C  \| f\|_{L^2_{r,\gamma, \sigma +2}(R^2) }.\]

Finally, to derive the bounds for the second derivative $\partial_{rr}u$, we use the equation 
 to write
 \begin{align*}
| \partial_{rr} u | & \leq |f| + \frac{1}{r} | \partial_r u| + \frac{1}{r^2}|u| + |u|\\ 
 \| \partial_{rr} u \|_{L^2_{r,\gamma,\sigma+2}(B_1) } & \leq 
\| f \|_{L^2_{r,\gamma,\sigma+2}(B_1) } + \| \partial_r u \|_{L^2_{r,\gamma,\sigma+1}(B_1) } 
+ \|  u \|_{L^2_{r,\gamma,\sigma }(B_1) } + \| u \|_{L^2_{r,\gamma,\sigma+2}(B_1) } \\
& \leq 4 C \| f \|_{L^2_{r,\gamma,\sigma+2}(R^2) }.
\end{align*}
Here the last inequality follows from our previous bounds for $u$ and $\partial_r u$, and the embedding $L^2_{r,\gamma, \sigma }(B_1) \subset L^2_{r, \gamma, \sigma+2}(B_1)$.

{\bf Step 5:} Lastly, to prove that the general operator $(\Delta_1 - \Id) : H^{s}_{r,\gamma, \sigma}(\R^2) \longrightarrow H^{s-2}_{r,\gamma, \sigma+2}(\R^2)$ is invertible one can use induction:
Suppose that $f \in H^{s-2}_{r,\gamma, \sigma+2}(\R^2)$ and that $u \in H^{s-1}_{r,\gamma, \sigma}(\R^2)$. Then,  the result that $\partial^s_r u \in L^2_{r,\gamma, \sigma+s}(\R^2)$, follows from the equation
\[ \partial^s_r u = \partial^{s-2}_r( f + u + \frac{1}{r^2} u - \frac{1}{r} \partial_r u),\]
and the embedding $H^{s-1}_{r,\gamma, \sigma}(\R^2) \subset H^{s-1}_{r,\gamma, \sigma+k}(\R^2)$, which hold for $k>0$.
\end{proof}


\section{Normal Form}\label{s:normalform}

As mentioned in the introduction, the following nonlocal version of the complex Ginzburg-Landau equation was rigorously 
derived  in reference \cite{jaramillo2022rotating} as an amplitude equation for spiral waves patterns in 
oscillatory media with nonlocal coupling:
\begin{equation}\label{e:main}
0 = K \ast w + (1+ \rmi \lambda) w - (1 + \rmi \beta) |w|^2 w + N(w;\eps), \qquad x = (r, \theta )  \in \R^2
\end{equation}
Notice that when compared to the original complex Ginzburg-Landau (cGL) equation, the above expression uses a convolution operator, $K \ast$, in place of the Laplacian operator. In addition, although the equation 
is posed on the plane, $\R^2$, the unknown function, $w$, depends only on the radial variable 
$r$, that is $w=w(r)$. Another difference with the standard cGL equation is the term $N(w,\eps) \sim \rmO(\eps \; |w|^2w )$, which captures all higher order terms that are usually ignored in a formal derivation of this equation, see Appendix \ref{a:nonlinear}. As shown in \cite{jaramillo2022rotating}, these terms need to be taken into account in order to conclude that our results are valid approximations to the solutions of the original system.

In this Section we prepare our normal form so that the analysis can go more smoothly. 
First, recalling Hypothesis (H2) and  the discussion from Subsection \ref{ss:outline}, we  write the convolution $K \ast $ as 
\[ K \ast u =\eta  (1 - \eps^2 D \Delta_1)^{-1} \Delta_1 u,\]
where, as already stated, we use the symbol $\Delta_n$ to denote the map
\[ \Delta_n u = \partial_{rr} u + \frac{1}{r} \partial_r u - \frac{n}{r^2} u.\]
We may then precondition equation \eqref{e:main} with the operator $(1 - \eps^2 D \Delta_1)$. 
The result is the following expression,
\[ 0=d \Delta_1 w + ( 1+ \rmi \lambda) w - (1 + \rmi \beta) |w|^2w + \tilde{N}(w;\eps),\]
where the parameter $d \in \C$ is given by
$$d = d_R + \rmi d_I = (\eta -\eps^2 D ) - \rmi \eps^2 D \lambda.$$
Next, we rescale the radial variable $r = \sqrt{d_R} y$, let $\alpha = d_I/d_R$, and write
\[ 0=(1+ i \alpha) \Delta_{1,y }w + ( 1+ \rmi \lambda) w - (1 + \rmi \beta) |w|^2w + \tilde{N}(w;\eps),\]
where the nonlinearities are now given by
\[ \tilde{N}(w;\eps) =   \frac{ \eps^2 D}{d_R} (1 + \rmi \beta) \Delta_{1,y} |w|^2 w + (1- \frac{\eps^2 D}{d_R} \Delta_{1,y}) N(w;\eps).\]

Finally, we write the complex valued function, $w$, using polar coordinates $w = \rho(y) \rme^{\rmi \phi(y)}$, and separate the equation into its real and imaginary parts. This gives the following system of equations, 
where for convenience we revert back to the original radial variable, $r$, 
\begin{eqnarray}\label{e:real}
0 = &  \left[ \Delta_1 \rho - (\partial_r \phi)^2 \rho \right] - \alpha \left[ \rho \Delta_0 \phi  + 2\partial_r \phi \partial_r \rho \right] +  \rho -  \rho^3 + \mathrm{Re}\left[ \tilde{N}(w; \eps) \rme^{-\rmi \phi} \right]\\ \label{e:imag}
0 = &   \left[ \rho \Delta_0 \phi  + 2\partial_r \phi \partial_r \rho \right] + \alpha \left[ \Delta_1 \rho - (\partial_r \phi)^2 \rho \right] + \lambda \rho - \beta \rho^3 + \mathrm{Im}\left[ \tilde{N}(w; \eps) \rme^{-\rmi \phi} \right].
\end{eqnarray}

To prove the existence of solutions to the above system, we proceed via a perturbation analysis. We rescale the variable $r$ by defining $S = \delta r$, where $\delta$ is assumed to be a small positive parameter such that $\eps \sim \delta^4$. We also use the following expressions for the unknown functions:
\[\begin{array}{r l c l}
\rho = & \rho_0 + \delta^2 (R_0 + \delta R_1), &  \qquad \rho_0= \rho_0(r),& \quad R_i = R_i(\delta r) \quad i=0,1\\
\phi = & \phi_0 + \delta \phi_1, & & \quad  \phi_i = \phi_i(\delta r) \quad i =0,1.
\end{array}\]
Notice that these are the same type of scalings used in \cite{doelman2005} to derive a phase dynamics approximation 
for the local cGL equation.
For the free parameter, $\lambda$, representing the rotational speed of the wave, we choose
$ \lambda = \beta + \delta^2 \Omega,$ with $\beta$ as above and $\Omega$ to be determined.

The result is a hierarchy of equations at different powers of $\delta$ that we present next. In terms of notation, we use the subscript $S$ to distinguish operators that are applied to functions that depend on this variable, i.e. $\Delta_{0,S}$. The absence of this subscript, indicates that the operator is applied to a function of the original variable $r$.

Recalling  that $\alpha \sim \rmO(\eps^2)$, we find the following system of equations.
\begin{itemize}
\item  At order $\rmO(1)$:
\begin{align*}
0 = & \Delta_1\rho_0 +  \rho_0 -  \rho_0^3,\\
0 = &  \beta \rho_0 - \beta \rho_0^3.
\end{align*}
\item At order $\rmO(\delta)$:
\[ 0 = -2 \partial_r \rho_0 \partial_S \phi_0.\]
\item At order $\rmO(\delta^2)$:
\begin{align*}
0 = &  -  \rho_0 (\partial_S \phi_0)^2   +  R_0( 1- 3 \rho_0^2),\\
0 = &  \rho_0 \Delta_{0,S} \phi_0  -2 \partial_r \rho_0 \partial_s \phi_1 +
\beta R_0 ( 1- 3 \rho_0^2) + \Omega \rho_0.
\end{align*}
\item Remaining higher order terms:
\begin{align*}
 0 = &  \delta^2 \Delta_{1,S} R_1 + R_1( 1- 3 \rho_0^2)   -2  \rho_0 \partial_S \phi_0 \partial_S \phi_1 + N_1(R_1, \phi_1;\delta,\eps),\\
 0 = & \rho_0 \Delta_{0,S} \phi_1 + \beta R_1( 1 - 3 \rho_0^2) +N_2(R_1, \phi_1; \delta,\eps),
 \end{align*}
where
\begin{equation}\label{e:nonlinear1}
\begin{split}
N_1(R_1, \phi_1; \delta,\eps) = & \quad
\delta \Big ( \Delta_{1,S} R_0 + ( R_0 + \delta R_1) ( \partial_S \phi_0 + \delta \partial_S \phi_1)^2 -  \rho_0 (\partial_S \phi_1)^2 \Big )\\
& - \delta \Big(  3 \rho_0 ( R_0 + \delta R_1)^2 + \delta^2 (R_0 + \delta R_1)^3 \Big) \\
& - \delta^2 \alpha \Big(  \rho_0 + \delta^2 ( R_0 + \delta R_1\Big ) (\Delta_{0,S} \phi_0 + \delta \Delta_{0,S} \phi_1)\\
& - 2 \delta  \alpha \Big(  \partial_r \rho_0 + \delta^3( \partial_S R_0 + \delta \partial_S R_1 \Big) ( \partial_S \phi_0 + \delta \partial_S \phi_1) \\
& + \mathrm{Re}\left[ \tilde{N}(w; \eps) \rme^{-\rmi \phi} \right],
\end{split}
\end{equation}

\begin{equation}\label{e:nonlinear2}
\begin{split}
N_2(R_1, \phi_1; \delta,\eps) = & \quad
\delta \Big( (R_0 + \delta R_1) ( \Delta_{0,S} \phi_0 + \delta \Delta_{0,S} \phi_1) - 2 (\partial_S R_0 + \delta \partial_S R_1)
(\partial_S \phi_0 + \delta \partial_S \phi_1) \Big) \\
& + \delta \Big( \Omega(R_0 + \delta R_1) - 3 \beta\rho_0 ( R_0 + \delta R_1)^2 - \beta \delta^2 (R_0 + \delta R_1)^3 \Big)\\
& + \alpha \Big( \Delta_1 \rho_0 + \delta^4 ( \Delta_{1,S} R_0 + \delta \Delta_{1,S} R_1)  \Big)\\
& - \delta^2 \alpha \Big( \rho_0 + \delta^2( R_0 + \delta R_1) \Big) ( \partial_S \phi_0 + \delta \partial_S \phi_1)^2\\
& + \mathrm{Im}\left[ \tilde{N}(w; \eps) \rme^{-\rmi \phi} \right].
\end{split}
\end{equation}

\end{itemize}

In the next section we solve the $\rmO(1)$ and $\rmO(\delta^2)$ equations explicitly. 
We then use these results, together with all remaining higher order terms, to define operators $F_i(R_1,\phi_1;\delta)$ with $i=1,2$, noting that the zeros of these operators will then correspond to the solutions we seek. In Section \ref{s:existence} we show that these maps satisfy the conditions of the implicit function theorem, and thus prove the existence of solutions to the nonlocal CGL equation representing spiral waves.

\section{The approximations}\label{s:approximations}

In this section we look in more detail at the  $\rmO(1)$, $\rmO(\delta)$, and $\rmO(\delta^2)$ equations.
In addition, we elaborate on the connection between our approximations and well known results
regarding the existence of spiral waves in reaction diffusion equations,  \cite{kopell1981,greenberg1980, greenberg1981}.

\subsection{The order one equation}\label{ss:one}
We study the system,
\begin{equation}\label{e:order1_syst}
\begin{split}
0 = &  \Delta_1\rho_0 +  \rho_0 -  \rho_0^3,\\
0 = &   \beta \rho_0 - \beta \rho_0^3,
\end{split}
\end{equation}
and to start we concentrate on the first expression, which we view as a  boundary value problem:
\begin{equation}\label{e:rhozero}
 0 =  \Delta_1\rho+ (1  - \rho^2) \rho, \qquad \rho \rightarrow 1\quad \mbox{as}\quad r \rightarrow \infty.
 \end{equation}
Here, we prove the existence of solutions to this equation  and show that they 
 converge to $1$ at a rate of $\rmO(1/r^2)$, a result which will be important for us in later sections. We then use this information to show that the second equation in the above system is really of order $\rmO(\delta^2)$. 
 These results are summarized in the following proposition.
\begin{Proposition}\label{p:decayg}
Consider the system \eqref{e:order1_syst} defined by the order $\rmO(1)$ terms.
There exists an approximation, $\rho_0(r)$, which
solves the first equation and that satisfies:
\begin{enumerate}[i.)]
\item $\rho_0(r) \in C^2((0,\infty]) \cap C^1([0,\infty])$ is positive and increasing for $r>0$.
\item $\rho_0(r) \sim \rmO(r)$ as $r \to 0$, while $\rho_0(r) \to 1$ as $ r \to \infty$.
\item $(1- \rho_0(r)^2) \in H^3_{r,\gamma}(\R^2)$ with $\gamma \in (0,1)$ and satisfies $(1- \rho_0(r)^2) \sim \rmO(1/r^2)$ as $r \to \infty$.

\end{enumerate}
\end{Proposition} 
 
 To prove Proposition \ref{p:decayg} we begin by recalling previous work by Kopell and Howard \cite{kopell1981}, which looks at a more general version of the above boundary value problem, equation \eqref{e:rhozero}, but does not provide the rate at which the solutions converge to one.
 The precise equation and results from \cite{kopell1981} are summarized next.

\begin{Proposition}\label{p:kopell}[ Theorem 3.1 in \cite{kopell1981}] Let $m \in \N$ and consider the function $f(\rho)$, with $f(1) =0$ and $f'<0$. Then, the boundary value problem,
\[ \Delta_m \rho + f(\rho) \rho =0, \qquad \rho(r) \sim b r^m \quad \mbox{as} \quad r \longrightarrow 0,\]
has a unique solution, $\rho_m(r)$, satisfying 
\begin{enumerate}[i)]
\item $0< \rho_m(r) <1 $ for all $r>0$,
\item $\rho'_m(r) >0$ for all $r>0$, and
\item $\rho_m(r) \rightarrow 1$ as $r \rightarrow \infty$.
\end{enumerate}
\end{Proposition}

The above results by Kopell and Howard prove the existence of solutions to our boundary value problem and also give us items i) and ii) in Proposition \ref{p:decayg}. Notice that similar results can be found in Greenberg's papers \cite{greenberg1980, greenberg1981}, where it is assumed that the function $f(\rho)$ in the above proposition is $f(\rho) = 1-\rho$. 

In what follows we present an alternative proof for the existence of solutions to the boundary value problem \eqref{e:rhozero}. Our results not only give us existence, but also the level of decay with which these solutions approach 1 at infinity, proving item iii) in Proposition \ref{p:decayg}.
More precisely, we show the following result.
\begin{Lemma}\label{l:rhodecay}
The boundary value problem 
\[  0 =   \Delta_1\rho+  (1  - \rho^2) \rho, \qquad \rho \rightarrow 1 \quad \mbox{as} \quad r \rightarrow \infty,\]
has a unique solution. Moreover, the difference $( 1- \rho(r) )$ is of order $\rmO(1/r^2)$, and
\[( 1- \rho(r)^2 ) \sim \dfrac{2}{r^2} + \rmo(1/r^2)\]
as $r$ goes to infinity.
\end{Lemma}

\begin{proof}
First, to simplify the analysis,  we consider the ansatz $\rho(\xi) = 1+ u(r)$. The corresponding equation for $u$, is then
\[  \Delta_1 u - 2  u = \frac{1}{r^2} + ( 3 u^2 +u^3).\]
Let $u= u_n + u_\ell$, and notice that this is a solution to the above equation provided $u_\ell$ solves the linear problem,
\[  \Delta_{1} u_\ell - 2  u_\ell - \frac{1}{r^2} =0 , \qquad u_\ell \rightarrow 0 \quad \mbox{as} \quad r \rightarrow \infty,  \]
and $u_n$ solves the nonlinear equation
\begin{equation}\label{e:un}  
\Delta_{1} u_n - 2  u_n -   \Big ( 3 (u_n+u_\ell)^2 + (u_n+u_\ell)^3\Big)=0.
\end{equation}
We will show that the term $1+ u_\ell$ captures the main behavior of the solution, $\rho(r)$, while $u_n$ corresponds to a small correction.

For ease of exposition we prove the existence of solutions, $u_\ell$, to the linear problem in 
Lemma \ref{l:decayu_p}, shown below.
Lemma \ref{l:decayu_p} also shows that the solution $u_\ell$ 
decays at $\rmO(1/r^2)$ as $r$ approaches infinity.

Next, having obtained the asymptotic decay of $u_\ell$, we move on to proving the existence of solutions to equation \eqref{e:un} which bifurcate from  zero using the implicit function theorem. We assume that $u_n$ has a regular expansion $u_n = \eps u_1 + \eps^2 u_2 + \cdots$, and consider the left hand side of equation \eqref{e:un} as an operator $F: H^s_{r, \gamma}(\R^2) \times \R \longrightarrow H^{s-2}_{r, \gamma}(\R^2)$. 

From Corollary  \ref{c:Delta-c} we know that the linearization about the origin $D_uF(0;0): H^s_{r,\gamma}(\R^2) \longrightarrow H^{s-2}_{r,\gamma}(\R^2)$, given by the expression
\[ D_uF(0;0) = \Delta_1  - 2,\]
 defines an invertible operator for all values of $\gamma$.
To complete our argument, we need to pick the correct value of the weight $\gamma$ that guarantees that the operator $F$ is well defined. In particular, we are concerned with showing that the nonlinear terms,
\[\left( 3 (u_n+u_\ell)^2 + (u_n+u_\ell)^3\right)=0.\]
are in $H^s_{r,\gamma}(\R^2)$. 

Notice that if we let $u_n \in H^2_{r,\gamma}(\R^2)$ with $\gamma>0$, then by Sobolev embeddings $u_n$ is a bounded and continuous function. As a result any product $u_n^p$  is in the  space $L^2_{r,\gamma}(\R^2)$. Similarly, because $u_\ell$ is bounded near the origin and decays like $1/r^2$ at infinity, any term of the form $u_n^pu_\ell^q$ is also in $L^2_{r,\gamma}(\R^2)$, for $\gamma>0$. 
Therefore, the level of algebraic localization of the nonlinear terms is controlled by the term $u_\ell^2$. Since we have that $u^2_\ell \sim \rmO(1/r^4)$ in the far field, then the nonlinearities are in the space $L^2_{r,\gamma^*}$ for values of $\gamma^*< 3$.

We may conclude then that the solution $u_n$ is in the space $H^2_{r,\gamma^*}(\R^2)$ and that it has the same level of decay at infinity as $u_\ell^2$. Consequently, the solution, $\rho$, to the original boundary value problem \eqref{e:rhozero} satisfies $\rho(r) -1 \sim \rmO(u_\ell ) = \rmO(1/r^2)$.

Finally, a short computation, together with our result that $u_n \sim \rmO(u_\ell^2)$, shows that
$ (1- \rho^2(r)) = -2 u_\ell (r) + \rmo(1/r^2)$.
It then follows from Lemma \ref{l:decayu_p}  that 
\[  (1- \rho^2(r)) = \frac{2 }{ r^2 } + \rmo(1/r^2).\]
\end{proof}

The following Lemma captures the asymptotic behavior of the solution to the linear inhomogeneous problem mentioned in the above analysis.

\begin{Lemma}\label{l:decayu_p}
There exists solutions to the ordinary differential equation,
\[  \partial_{\xi \xi}u +\frac{1}{\xi}\partial_\xi u -\frac{1}{\xi^2} u -  u = \frac{1}{\xi^2}\]
satisfying $u \sim \rmO(1/\xi^2)$ as $\xi \to \infty$, while $u \sim \rmO(1)$ as $\xi \to 0$.
\end{Lemma}

\begin{proof}
The proof of this lemma relies on Proposition \ref{p:fredholm1} where it is shown that the operator
$(\Delta_1 -\Id) :H^s_{r,\gamma,\sigma}(\R^2) \longrightarrow H^{s-2}_{r,\gamma,\sigma+2}(\R^2)$ is invertible for values of $\sigma \in (-2,0)$ and $\gamma \in \R$.
A short calculation then shows that $f = 1/\xi^2$ is in the space $H^{s-2}_{r,\gamma,\sigma}(\R^2)$, for any  integer $s\geq2$, and for values of $\gamma \in (0,1)$ and $\sigma +2 \in (1,2)$. It then follows that the solution $u$ to
\[ (\Delta_1 -\Id) u  =  \partial_{\xi \xi}u +\frac{1}{\xi}\partial_\xi u -\frac{1}{\xi^2} u -  u = \frac{1}{\xi^2}\]
is in the space $ H^s_{r,\gamma,\sigma}(\R^2) $. In particular, $u$ has the 
same level of decay at infinity as $f$, so $u \sim \rmO(1/\xi^2)$ as $\xi \to \infty$.

To obtain the growth rate of $u$ near the origin, we use the Green's function for the operator $(\Delta_1 -\Id)$ and write
\[ u(\xi) = I_1(\xi) \int_\xi^\infty \frac{ K_1(\rho)}{ \rho} \;d\rho +
  K_1(\xi) \int_0^\xi\frac{ I_1(\rho) }{ \rho} \;d\rho.
  \]
Using the asymptotic approximation for $I_1(\xi)$ near the origin, we may bound the first term by
\[A(\xi)  =  I_1(\xi) \int_\xi^\infty \frac{ K_1(\rho)}{ \rho} \;d\rho \leq C\xi \int_\xi^\infty \frac{ K_1(\rho)}{ \rho} \;d\rho.\]
Because $K_1(\rho)/\rho $ decays at infinity, there is a constant $M$ such that
\[ A(\xi) \leq  C\xi \left[  \int_\xi^1 \frac{ K_1(\rho)}{ \rho} \;d\rho + \int_1^\infty \frac{ K_1(\rho)}{ \rho} \;d\rho  \right] \leq C\xi \left[  \int_\xi^1 \frac{1}{\rho^2} \;d\rho +M\right], \]
and it then follows that $A(\xi) \sim \rmO(1)$ as $\xi \to 0$.

Similarly, using the asymptotic expansions for $K_1$ and $I_1$ for values of $\xi \sim 0$, we may write
\[ B(\xi) = K_1(\xi) \int_0^\xi\frac{ I_1(\rho) }{ \rho} \;d\rho \leq \frac{C}{\xi} \int_0^\xi 1 \;d\rho <C\]
It then follows that $B(\xi) \sim \rmO(1)$  as $\xi \to 0$. The results of the lemma then follow.
\end{proof}

\begin{Remark}
Notice that because $\rho_0$ solves the second order ode \eqref{e:rhozero}, it is an element of $C^2((0,\infty)) \cap C^1([0,\infty))$.
The decay rates for $(1-\rho(r)^2)$ presented in Lemma \ref{l:rhodecay} then imply that this  function is in 
$ C^2([0,\infty)) \subset H^3_{r,\gamma}(\R^2)$ with $0<\gamma <1$, see also Figure \ref{f:decay}. This completes the proof of item iii) in Proposition \ref{p:decayg}.
\end{Remark}

We close this section by showing  that the expression defining the second equation in our order $\rmO(1)$ approximation, that is
\[ \beta \rho_0 - \beta \rho_0^3  = \beta\rho_0 ( 1+ \rho_0)( 1- \rho_0),\]
is of order $\rmO(\delta^2)$. To do this we use the results of the previous lemma.

Indeed, because $\rho_0$ solves equation \eqref{e:rhozero}, 
and the function $\rho_0$ satisfies $(\rho_0 -1) \sim \rmO(1/r^2)$ in the far field, one sees that the above expression is also of order $\rmO(1/r^2)$.
Letting  $S = \delta r$ and recalling that $\rho_0$ is bounded near the origin, one finds that, in terms of the variable $S$,  
\[ \beta  ( 1- \rho_0)(1 + \rho_0) \rho_0 \sim \rmO(\delta^2). \]
This `left over' term will be included in the $\rmO(\delta^2)$ equation in the next subsection. For notational convenience we will also use the following definition.
\begin{Definition}\label{d:g}
We let $\tilde{\rho}_0 = \rho_0(S/\delta)$ and define
\[ g(S) = \frac{1}{\delta^2}  (1- \tilde{\rho}_0^2),\]
which satisfies $g(S) \sim \rmO(1/S^2)$ as $S \to \infty$.
\end{Definition}

\subsection{The order $\delta$ equation}\label{s:delta}
In this short subsection we look at the terms appearing in the order $\rmO(\delta)$ equation, 
\[  -2 \partial_r \rho_0 \partial_S \phi_0.\]
and show that these are really of order $\rmO(\delta^4)$. This result follows from this next lemma.

\begin{Lemma}\label{l:order_delta}
Let $\rho_0$ satisfy the boundary value problem
\[0 = \Delta_1 \rho_0 + (1-\rho_0^2) \rho_0, \quad \rho_0 \to 1 \quad \mbox{as} \quad r \to \infty.\]
Then, given $f(r,S) = \partial_r \rho_0 \partial_S \phi_0$ and  the rescaling  $S = \delta r$, the term
\[ \tilde{f}(S) = f(S/\delta, S)   \sim \rmO(\delta^3) \quad \mbox{as} \quad S \to \infty.\]
\end{Lemma}

\begin{proof}
From Lemma \ref{l:rhodecay} we know that the solution, $\rho_0$, to the given boundary value problem satisfies $(\rho_0 -1) \sim \rmO(1/r^2)$ as $r \to \infty$. As a result, we may conclude that $\partial_r \rho_0 \sim \rmO(1/r^3)$ as $r \to \infty$. Using the rescaling, $r = S/\delta$ we arrive at the desired result.
\end{proof}

\subsection{The order $\delta^2$ equation}\label{ss:delta2}
We now move on to the next set of equations
\begin{align*}
0 = & - \rho_0 (\partial_S \phi_0)^2 + R_0 ( 1- 3\rho_0^2)\\
0 = &\; \rho_0 \Delta_{0,S} \phi_0 - 2 \partial_r \rho_0 \partial_S \phi_1 + \beta R_0(1 - 3 \rho_0^2)
+ \Omega \rho_0 + \rho_0 \beta g(S)
\end{align*}
where the function $g(S)$ is given as in Definition \ref{d:g}.

To simplify the analysis of the above system, we use the information  presented in Proposition \ref{p:decayg}. In particular:
\begin{itemize}
\item 
A similar proof as in Lemma \ref{l:order_delta}, 
shows that the expression $-2 \partial_r \rho_0 \partial_S \phi_1$ is 
of order $\rmO(\delta^3)$.  Therefore, this term can be moved to the next set of higher order equations.
\item Similarly,  because $1- \rho_0^2 \sim \rmO(1/r^2)$ when $r$ is large,
we can write
\[ 1 - 3\rho_0^2 = -2 + 3( 1- \rho_0^2) = -2 + h(S)\]
and conclude that the function $h(S) = 3( 1- \rho_0^2)$ is of oder $\rmO(\delta^2)$.
Thus, we also move these terms to the next set of equations.
\end{itemize}

As a result, in this section we concentrate on 
solving the system
\begin{align}\label{e:first}
0 = & - \rho_0( \partial_S \phi_0)^2 -2 R_0\\
 \label{e:second}
0 = & \rho_0 \Delta_{0,S} \phi_0  -2 \beta R_0 + \Omega \rho_0 + \rho_0 \beta g(S).
\end{align}
More precisely, we prove the following proposition, which also
states smoothness and decay properties of the solutions $R_0$ and $\phi_0$.
In the proposition, we use 
$\chi^D  \in C^{\infty}$ to denote a radial cut-off function, with 
$\chi^D(|x|) =0$ for $|x|<D$, for some positive number $D$, and $\chi^D(|x|) =1$ for $|x|>2D$.
The symbol $\gamma_e$ is also used here to represent the Euler-Mascheroni constant, while $K_0$ denotes the modified Bessel function of the second kind.

\begin{Proposition}\label{p:dacayRphi}
Let $g(S)$ be as in Definition \ref{d:g}, pick $D>0$ and assume $\beta<0$. Then, 
the solutions to the equations \eqref{e:first} and \eqref{e:second},
are given by,
\begin{align*}
 \phi_0(S) = & \underbrace{ \frac{1}{\beta} \chi(\Lambda S) \log( K_0(\Lambda S))  + a c }_{\tilde{\phi}_1} + \tilde{\phi_2}(S ), \qquad \Lambda^2 = -\beta \Omega(a),  \quad c \in \R,\\
R_0(S) = & - \frac{1 }{2} \rho_0 (\partial_S \phi_0)^2,
\end{align*}
where $\Omega(a) = 4 C(a) \rme^{-2 \gamma_e} \exp(2/a)$, with  $C(z)$  a $C^1(\R)$ function in $z$ and
\[ a= - \beta^2 \int_0^\infty g_c(S) S\;dS; \qquad g_c = (1- \chi^D)g.\]
Moreover, we have that
\begin{enumerate}[i.)]
\item $\partial_S \tilde{\phi}_2 \in H^3_{r, \gamma}(\R^2)$  for $\gamma \in (0,1)$, 
\item $ \partial_S \phi_0 \to \kappa$ and $R_0 \to \frac{1}{2} \kappa ^2$ as $S \to \infty$, with $\kappa = -\dfrac{\Lambda}{\beta} >0$, and
\item  $ (\partial_S \phi_0 - \kappa), ( R_0 - \frac{1}{2} \kappa^2) \in H^3_{r,\bar{\gamma}}$ for $\bar{\gamma} \in (-1,0)$
\end{enumerate}

\end{Proposition}

Before proving Proposition \ref{p:dacayRphi}, notice first that we can easily solve 
for the unknown $R_0$ using equation \eqref{e:first},
\[R_0 = - \frac{1}{2}  \rho_0( \partial_S \phi_0)^2.\]
Inserting this result into the expression \eqref{e:second} then leads to the  viscous eikonal equation,
\[0 = \Delta_{0,S} \phi_0  + \beta ( \partial_S \phi_0)^2 + \Omega  +\beta g(S).\]
For convenience, in the analysis that follows we let $b = -\beta >0$ (see Hypothesis (H1) ) and define $\bar{g}(S) = b g(S)/\epsilon$, for some arbitrarily small number $\epsilon>0$.
With this notation, the above viscous eikonal equation then reads
\begin{equation}\label{e:eikonal}
0 = \Delta_{0,S} \phi_0  - b ( \partial_S \phi_0)^2 + \Omega  -\epsilon \bar{g} (S).
\end{equation}

Notice that, not surprisingly, we have arrived at the same phase dynamics approximation that can be formally derived for systems undergoing a Hopf bifurcation, see  for example \cite{kollar2007, jaramillo2018}.
Our goal in this subsection is to solve this nonlinear equation.
It is important to note that this equation constitutes a nonlinear eigenvalue problem, since we need to find the solution, $\phi_0$, together with the corresponding value of $\Omega$.
This is because $\Omega $ is related to the unknown parameter $\lambda$, which represents the rotational speed of  the spiral pattern, through the expression $\lambda = \beta + \delta^2 \Omega$. 
 In particular, we are interested in solutions, $\phi_0(S)$, which in the far field behave like planar waves, since as mentioned in the introdcution these solutions would then represent spiral waves. That is, we require that $\partial_S \phi_0 \rightarrow k $ as $S$ goes to infinity, for some constant $k$. We also ask that the group velocity of these solutions be positive, i.e. $\partial_S \phi_0 \cdot S >0$, indicating that these planar waves move outwards, away from the spiral's core.

 To prove the existence of solutions to equation \eqref{e:eikonal}, we  use  the results from reference \cite{jaramillo2022can}, where the same viscous eikonal equation 
 is used to model target patterns in 2-d oscillatory media.
 In this context, the function $g(S)$ that appears in equation \eqref{e:eikonal} represents a defect, or impurity, that
is present in the medium. This defect acts as a pacemaker entraining the rest of the system and generating a
series of concentric waves that propagate away from it.
n contrast to previous results that also deal with the existence of target pattern solutions (see for example \cite{kuramoto1976, kopell1981, hagan1981, kollar2007, jaramillo2016, jaramillo2018}), in reference \cite{jaramillo2022can} the function $g(S)$ is assumed to represent a large defect, in the sense that it does not have finite mass. More precisely, it is assumed that $g(S) \sim \rmO(1/S^m)$ as $S \to \infty$, with $1<m\leq 2$.
 Since the function $g(S)$ considered here is of order $\rmO(1/S^2)$, and 
  because in the far field our spiral wave solutions, as well as target patterns,  
satisfy $\nabla \phi \cdot \frac{{\bf x}}{| {\bf x}|}  \rightarrow k $ as $|{\bf x}| \rightarrow \infty$,
the results from \cite{jaramillo2022can} are precisely the ones we need in this section. 
We therefore summarize them in the next theorem using the notation already introduced.
 We start by stating the hypothesis placed on $g$ in reference \cite{jaramillo2022can}.

\begin{Hypothesis}\label{h:g}
The inhomogeneity, $g$, lives in $H_{ \gamma} ^k(\R^2)$, with $k\geq2$ and $\gamma \in (0,1)$,
is radially symmetric, and positive. In addition, this function can be split into the sum of  two positive functions, $g_c, g_f$, satisfying
\begin{itemize}
\setlength \itemsep{2ex}
\item The function $g_f$ is in $ H^k_\gamma(\R^2)$ for $0<\gamma<1$. In particular,
 $g_f \sim \rmO(1/r^m)$ as $r \to \infty$, with $1< m \leq 2$, while near the origin  $g_f(|x|) =0$ for $|x|<1$.
\item The function $g_c$ is in $H^k_{\tilde{\gamma}}(\R^2)$ for $\tilde{\gamma}>1$. In particular,
$g_c \sim \rmO(1/r^{d})$ with $d>2$ as $r \to \infty$.
\end{itemize}
\end{Hypothesis}

\begin{Theorem}\label{t:targetpatterns}
Let $k \geq 2$ and $\gamma \in (0, 1)$ and consider a function $g \in H^k_{r,\gamma}(\R^2)$ satisfying Hypothesis \ref{h:g}. Then, there exists a constant $\epsilon_0>0$ and a $C^1 ([0,\epsilon_0))$  family of eigenfunctions $\phi = \phi(S; \epsilon)$ and eigenvalues $\Omega= \Omega(\epsilon)$ that bifurcate from zero and solve
the equation
\[\Delta_0 \phi - b(\partial_S \phi) ^2 + \Omega - \epsilon g(S)  =0 \qquad S= |x| \in [0,\infty), \quad b>0.\]
Moreover, this family has the form
\begin{equation}\label{e:target-pattern}
\phi(S;\epsilon) = -\frac{1}{b} \chi(\Lambda S) \log( K_0(\Lambda S)) + \phi_2(S; \epsilon ) + \epsilon c , \qquad \Lambda^2 = b \Omega(\epsilon) 
\end{equation}
where
\begin{enumerate}[i)]
\item $c$ is a constant that depends on the initial conditions of the problem,
\item $K_0(z)$ represents the zeroth-order Modified Bessel function of the second kind,
\item $ \partial_S \phi_2 \in H^{k}_{r,\gamma}(\R^2) $, and
\item $\Omega= \Omega(\epsilon) = C(\epsilon) 4 e^{-2 \gamma_\epsilon}\exp[ 2/ a ]$, with
\[a = -\epsilon b \int_0^\infty g_c(S) \;S \;dS,\]
 and $C(\epsilon)$ a $C^1$ constant that depends on $\epsilon$.
\end{enumerate}
\end{Theorem}

As already mentioned, a complete proof of the above theorem can be found in \cite{jaramillo2022can}.
In the theorem, the function $\chi$ is a radial $C^\infty$ cut-off function
satisfying $\chi(x) = 0$ for $|x|<1$ and $\chi(x) =1$ for $|x|>2$.

\begin{table}[t]
\begin{center}
\begin{tabular}{ m{2cm} m{5cm} m{4.5cm}  } 
\specialrule{.1em}{.05em}{.05em} 
  & $z \to 0$ & $z \to \infty $\\
  \hline
 $K_0(z) $ &$ - \log(z/2) - \gamma_e + \rmO(z^2) $ &$ \sqrt{ \frac{\pi}{2 z} } \rme^{-z} \Big( 1 + \rmO(1/z) \Big)  $ \\ [3ex]
$ K_1(z) $ &$ \frac{1}{z} + \rmO(z)$ &$ \sqrt{ \frac{\pi}{2 z} } \rme^{-z} \Big( 1 + \rmO(1/z) \Big)  $\\      
\specialrule{.1em}{.05em}{.05em} 
\end{tabular}
\end{center}
\caption{ 
Asymptotic behavior for the Modified Bessel functions of the second kind of zero-th and first-order, taken from \cite[(9.6.8), (9.6.9), (9.7.2)]{abramowitz} }
\label{t:bessel}
\end{table}

{\bf Proof of Proposition \ref{p:dacayRphi}}
\begin{proof}
We first confirm that the function $ \bar{g}(S) = b g(S)/\epsilon $ in \eqref{e:eikonal} satisfies 
Hypothesis \ref{h:g}. From the Definition \ref{d:g} and Proposition \ref{p:decayg}, it follows that $\bar{g}(S) \sim (1-\rho_0^2) \sim \rmO(1/S^2)$ as $S$ approaches infinity. 
Similarly, from Proposition \ref{p:kopell} we know that the term  $(1- \rho^2_0)$ is always positive,
so that $\bar{g}(S)$ is a positive function provided $b = -\beta >0$.
In addition, Definition \ref{d:g} specifies that $g$ is in $H^3_{r,\gamma}(\R^2)$ with $\gamma \in (0,1)$, as desired.
Recalling that   $\chi^D  \in C^{\infty}$ denotes a radial cut-off function, with 
$\chi^D(|x|) =0$ for $|x|<D$, for some positive number $D$, and $\chi^D(|x|) =1$ for $|x|>2D$,
we can then define
\begin{equation}\label{e:g}
 \bar{g}_c = (1-\chi^D) \bar{g} \qquad \bar{g}_f= \chi^D \bar{g},
 \end{equation} 
with both functions, $\bar{g}_f$ and  $\bar{g}_c$ satisfying the rest of Hypothesis \ref{h:g}.

We can therefore invoke Theorem \ref{t:targetpatterns} so that for  $\epsilon \in (0,\epsilon_0)$, target-pattern solutions, $\phi(S; \epsilon)$, given as in expression \eqref{e:target-pattern} to
equation \eqref{e:eikonal}, exist.   
The value of $\epsilon$ is then related to the constant $a$ by
\[ a = -\epsilon b \int_0^\infty \bar{g}_c(S)\;S \;dS = -\beta^2 \int_0^\infty g_c(S)\;S \;dS,\]
which is fixed once the value of $D$ is picked. 
As a result, we arrive at the form  of the solutions $\phi_0$ and $R_0$ stated in Proposition \ref{p:dacayRphi}.

We are left with showing the decay properties of these functions, which we restate here for convenience:
\begin{enumerate}[i.)]
\item $\partial_S \tilde{\phi}_2 \in H^3_{r, \gamma}(\R^2)$  for $\gamma \in (0,1)$, 
\item $ \partial_S \phi_0 \to \kappa$ and $R_0 \to \frac{1}{2} \kappa ^2$ as $S \to \infty$, with $\kappa = -\dfrac{\Lambda}{\beta}$, and
\item  $ (\partial_S \phi_0 - \kappa), ( R_0 - \frac{1}{2} \kappa^2) \in H^3_{r,\bar{\gamma}}$ for $\bar{\gamma} \in (0,1)$. \end{enumerate}
Item i) follows from Theorem \ref{t:targetpatterns}. To prove items ii) and iii) for $\phi_0$,  we recall the form of this solution  shown in Proposition \ref{p:dacayRphi}, and calculate its derivative,
\[ \partial_S \phi_0(S) = \partial_S \tilde{\phi}_1 + \partial_S \tilde{\phi}_2 = \frac{\Lambda}{b}\frac{K_1(\Lambda S)}{K_0(\Lambda S)} + \partial_S \tilde{\phi}_2.\]
Since $\partial_S \tilde{\phi}_2 \in H^3_{r,\gamma}$ with $\gamma \in (0,1)$, it then follows from Lemma \ref{l:decay} that as $S$ approaches infinity, the function $\partial_S \tilde{\phi}_2 \sim \rmO(S^\alpha)$ with $\alpha < -\gamma -1$.
Similarly, using the decay properties of the Modified Bessel functions (or Mathematica), we find that
\begin{equation}\label{e:bessel}
 \frac{K_1(\xi)}{K_0(\xi)} \sim 1 +  \frac{1}{2\xi} - \frac{1}{8\xi^2} + \frac{1}{8\xi^3} + \rmO(1/\xi^4), \quad \mbox{as} \; \xi \to \infty.
 \end{equation}
Consequently, $\partial_S \phi_0 \to \kappa$ as $S$ approaches infinity, with $\kappa = \dfrac{\Lambda}{b} = - \dfrac{\Lambda}{\beta}$. In addition, the term $(\partial_S \phi_0 - \kappa) \sim \rmO(1/S)$ in the far field, and we may conclude that this function is in $H^3_{r,\bar{\gamma}}(\R^2)$ with $\bar{\gamma} \in (-1,0)$, see Lemma \ref{l:decay} and Figure \ref{f:decay}.

Next, recall that $R_0(S) = \frac{1}{2} \rho_0 (\partial_S \phi_0)^2$. 
Expanding the quadratic term, we can write
\[ ( \partial_S \phi_0)^2 = ( \partial_S \tilde{\phi}_1)^2 + 2 \partial_S \tilde{\phi}_1 \partial_S \tilde{\phi}_2  +  ( \partial_S \tilde{\phi}_2)^2.  \]
Since elements in $H^3_{r,\gamma}(\R^2)$, with $\gamma \in (0,1)$, are bounded it then follows that $(\partial_S \tilde{\phi}_2)^2 \in H^3_{r,\gamma}(\R^2) $.
In addition, because the cut-off function $\chi$ removes the singularity at the origin, the function $\partial_S  \tilde{\phi}_1$ is  smooth and uniformly bounded. As a result, the product $\partial_S \tilde{\phi}_1 \partial_S \tilde{\phi}_2 $ is also in the space $H^3_{r,\gamma}(\R^2)$. 
Thus, the far field behavior of $R_0$ is determined by $(\partial_S  \tilde{\phi}_1)^2$. Items ii) and iii) for $R_0$ then follow from the definition of $\partial_S  \tilde{\phi}_1$ , the expansion \eqref{e:bessel},
and the fact that $(R_0 - \frac{1}{2} \kappa^2) \sim \rmO(1/S)$ for large $S$.
\end{proof}

We end this section with some remarks regarding the value of $\Omega$ and the behavior of $\partial_S \Phi_0$ near the origin.

\begin{Remark}
Notice that the results of Theorem \ref{t:targetpatterns}, and consequently Proposition \ref{p:dacayRphi}, imply that the frequency, $\Omega$, of our spiral waves only 
depends on $g_c$, the core of the function $g$. Unfortunately, the theorem does not provide us with a way of determining
the bounds for what is consider the core.  This is why one can pick any cut-off function $\chi^D$,
as defined in the proof above, to determine $g_f$ and $g_c$. 
This choice affects the accuracy of our approximation for $\Omega$ and $\kappa$, but Theorem 1 makes it clear that no matter what  choice we make for $\chi^D$, the form of the phase solution, $\phi_0$, does not change.
\end{Remark}

\begin{Remark}
Notice that we do not have precise details regarding the asymptotic behavior for $\partial_S \phi_0$ near the origin. Nonetheless, Theorem \ref{t:targetpatterns} and Proposition \ref{p:dacayRphi} allow us to conclude at leas that $\partial_S \phi_0 \sim \rmO(1)$ near the origin. This follows from the fact that $\tilde{\phi}_1$ involves the cut off function $\chi$, which removes the singularity near the origin for this logarithmic term, while the fact that $\partial_S\tilde{\phi}_2 \in H^3_{r,\gamma}(\R^2) \subset C(\R)$ implies that this function is bounded near the origin, therefore at least of order $\rmO(1)$.
\end{Remark}

\section{Existence of spiral waves}\label{s:existence}

In this section we gather all terms of order $\rmO(\delta^3)$ and higher,  and write them as a system of equations.
Rearranging and separating linear and nonlinear terms we arrive at,
 \begin{align}\label{e:ordercube1}
 0 =  &(\delta^2 \Delta_{1,S} -2)R_1  -2 \partial_S \phi_0 \partial_S \phi_1 + \tilde{N}_1(R_1,  \partial_S \phi_1;\delta),\\
 \label{e:ordercube2}
 0 =  & -2\beta R_1 + \Delta_{0,S} \phi_1 + \tilde{N}_2(R_1, \partial_S \phi_1; \delta).
\end{align}
This is the same system appearing in Section \ref{s:normalform}, except that now the nonlinear terms
 are given by
 \begin{align*}
 \tilde{N}_1(R_1,\partial_S \phi_1;\delta) &= N_1(R_1, \phi_1;\delta,\delta^4) +(1-\rho_0)[ 2 \partial_S \phi_0 \partial_S \phi_1]+ 3(1- \rho_0^2)(R_0 + \delta R_1) / \delta\\
 \tilde{N}_2(R_1, \partial_S \phi_1; \delta) &=N_2(R_1, \phi_1; \delta,\delta^4) -(1-\rho_0)[   \Delta_{0,S} \phi_1 ] + 3 \beta (1- \rho_0^2)(R_0 + \delta R_1)/ \delta,
 \end{align*}
 where $N_1$ and $N_2$ are as in expressions  \eqref{e:nonlinear1} and \eqref{e:nonlinear2}, respectively, with $\eps =\delta^4$,
 and the last terms come from  the order $\rmO(\delta^2)$ equations after rewriting
 \begin{equation*}
  (1 - 3 \rho_0^2) = -2 + 3( 1- \rho_0^2),
  \end{equation*}
(see  Subsection \ref{ss:delta2}). Because the nonlinearities in the expressions $\tilde{N}_i$, $i=1,2$, depend on $\phi_1$ through its derivative, $\partial_S \phi_1$, we write them as functions of this last variable. Also, notice from the scaling $S = \delta r$ and  Lemma \ref{l:rhodecay} that 
the terms $(1- \rho_0)$ and $(1- \rho_0^2)$ are order $\rmO(\delta^2)$.

Our goal is to prove the existence of solutions  to the above system which
 bifurcate from zero.
Using the results from Section \ref{s:approximations} we assume  the functions $\rho_0$ and $R_0, \phi_0$ are given by the expressions
stated in Propositions \ref{p:decayg} and \ref{p:dacayRphi}, respectively.
 Because $\rho_0, R_0, \phi_0$ then represent the first order approximations to spiral wave solutions,  this immediately implies the existence of solutions to the nonlocal complex Ginzburg-Landau equation 
\eqref{e:main} representing these patterns. 
Thus, Theorem \ref{t:existence}, shown below, together with Propositions  \ref{p:decayg} and \ref{p:dacayRphi}, then imply the main results of this paper, Theorem  \ref{t:main}.

\begin{Theorem}\label{t:existence}
There exists a real number $\delta_0>0$, such that equations \eqref{e:ordercube1} and \eqref{e:ordercube2}
have a family of solutions $(R_1(S; \delta), \partial_S \phi_1(S; \delta))$ valid for $\delta \in [0, \delta_0)$. Moreover, this family is $C^1$ in the parameter $\delta$  and satisfies,
\begin{itemize}
\item $R_1(S; 0) =0 $ and $\partial_S \phi_1(S; 0) =0$, as well as having
\item both $\partial_S \Phi_1 (S;\delta)$ and $R_1(S;\delta )$ of order $\rmO(1)$ as $S \to \infty$, 
\end{itemize}

\end{Theorem}

To prove Theorem \ref{t:existence} we work with the equivalent system,
 \begin{align}\label{e:ordercube_new1}
 0 = & (\delta^2 \Delta_{1,S} -2)R_1  -2 \partial_S \phi_0 \partial_S \phi_1 + \tilde{N}_1(R_1, \partial_S \phi_1;\delta),\\
 \label{e:ordercube_new2}
 0 = & -2\beta \delta^2 \Delta_{1,S} R_1 + \Delta_{0,S} \phi_1 + 2 \beta \partial_S \phi_0 \partial_S \phi_1 + \tilde{N}_2(R_1, \partial_S \phi_1; \delta) - \beta \tilde{N}_1(R_1, \partial_S \phi_1;\delta).
\end{align}
which can be obtained by adding a $-\beta$ multiple of equation \eqref{e:ordercube1} to equation \eqref{e:ordercube2} to arrive at the second equation, \eqref{e:ordercube_new2}.
We carry out the proof as a series of steps.
In Step 1, we use Fredholm properties of the operator $\mathcal{L} \psi = (\partial_S + \frac{1}{S} - k ) \psi$ (Lemma \ref{l:fredholm2}), a result that is known as a  Bordering Lemma (Lemma \ref{l:bordering}), and the implicit function theorem to obtain a family of solutions $\partial_S\phi_1 = \partial_S\phi_1(R_1, \delta) $ to equation \eqref{e:ordercube_new2}.
This result is then summarized in Proposition \ref{p:phi}.
In Step 2,  we use Proposition \ref{p:fredholm1}, which implies that the operator $(\delta^2 \Delta_{1,S} - 2)$ is invertible, together with our family of solutions $\partial_S\phi_1$ and the implicit function theorem, to prove existence of solutions, $R_1=R_1(\delta)$, to equation  \eqref{e:ordercube_new1}. This in turn proves Theorem \ref{t:existence}.
In the course of the analysis we find that it is necessary to show that the nonlinear terms $\tilde{N}_1, \tilde{N}_2$ define
 bounded operators between appropriate weighted Sobolev spaces, and that they depend continuously on the variables $R_1, \partial_S \phi_1$ and the parameter $\delta$. This is shown in Step 3.

\begin{Remark}
Notice that all terms appearing in equations \eqref{e:ordercube_new1} and \eqref{e:ordercube_new2} involve the derivative $\partial_S \phi_1$ and not variable $\phi_1$.
We therefore view $\partial_S \phi_1 $, and not $\phi_1$, as our unknown. 
\end{Remark}

Before presenting our result, let us properly define the spaces that we will be using.
Here, and in the rest of this section, we let $\gamma_1 \in (-1,0)$, $\sigma_1 \in (-2,-1)$ and consider
\begin{equation}\label{e:spaces}
\begin{array}{ c c}
X = & H^2_{r,\gamma_1,\sigma_1}(\R^2) \oplus \left\{\zeta(S) \right \} \oplus \R\\[2ex]
Y = & H^1_{r,\gamma_1,\sigma_1+1}(\R^2) \oplus \{ \frac{1}{S} \}  \oplus \R\\[2ex]
Z = & L^2_{r,\gamma_1, \sigma_1 +2} (\R^2) \oplus \left\{ \frac{1}{S^2} \right \} \oplus \R
\end{array}
\end{equation}
where $\{ u \}$ denotes the span of  the function $u$ over the reals, and the function $\zeta(S)$ satisfies
\[ (\delta^2 \Delta_{1,S} -2) \zeta = -1.\]
These spaces are then endowed with the norms
\[ \begin{array}{r l l}
\| x \|_X = & \| R \|_{H^2_{r,\gamma_1, \sigma_1}(\R^2)}  + |x_1| + |x_2| \quad & \forall x  := (R,x_1, x_2) \in X\\
\| y \|_Y = & \| \phi \|_{H^1_{r,\gamma_1, \sigma_1+1}(\R^2)}  + |y_1| + |y_2| \quad & \forall  y  := (\phi ,y_1, y_2) \in Y\\
\| z \|_Z = & \|  N \|_{L^2_{r,\gamma_1, \sigma_1+2}(\R^2)}  + |z_1| + |z_2| \quad & \forall z := (N ,z_1, z_2) \in Z.
\end{array}
\]
The decay rates for $\partial_S\phi_1 $ and $R_1$ follow from this choice of Banach spaces.

In this section we will also use the following result, which establishes decay properties for the function $\zeta(S)$.

\begin{Lemma}\label{l:zeta}
Let $\zeta$ be the solution to the equation
\[ (\delta^2 \Delta_{1,S} - 2) \zeta = -1\]
Then, $\zeta(S) = \zeta_1(S) + \zeta_f$ where $\zeta_1 \in H^3_{r,\gamma}(\R^2)$, with $\gamma \in (0,1)$, and $\zeta_f \in \R$.
\end{Lemma}
\begin{proof}
Notice that with the change of variables $\xi = \sqrt{2}S/\delta$ we are able to write the 
equation $(\delta^2 \Delta_{1,S} - 2) \zeta = -1$ as
\[ \partial_{\xi \xi} \zeta + \frac{1}{\xi} \partial_\xi \zeta  - \frac{1}{\xi^2} \zeta - \zeta = -1/2\]
Letting $\zeta_f = 1/2$ so that $\zeta = \zeta_1(\xi) + 1/2$, we find that $\zeta_1(\xi)$ solves
\[ \partial_{\xi \xi} \zeta_1 + \frac{1}{\xi}  \partial_\xi \zeta_1  - \frac{1}{\xi^2} \zeta_1 - \zeta_1 = \frac{1}{2} \frac{1}{\xi^2}\]
The proof of Lemma \ref{l:decayu_p}  then shows that solutions to the above equation
satisfy $\zeta_1 \sim \rmO(1/\xi^2)$ as $\xi$ goes to infinity. It then follows from Lemma \ref{l:decay} that $\zeta_1 \in L^2_{r,\gamma}(\R^2)$, with $\gamma \in (0,1)$, see also Figure \ref{f:decay}.
Lemma \ref{l:decayu_p} also shows that $\zeta_1 \sim \rmO(1)$ as $\xi \to 0$.
Because, $\zeta_1$ satisfies a second order ODE and it is bounded near the origin, it follows that $\zeta_1 \in C^2([0,\infty)) \cap C^1([0,\infty))$ and is consequently an element in $
H^3_{r,\gamma}(\R^2)$.

\end{proof}

\subsection{Step 1}

In this section we use the implicit function theorem to find solutions to equation \eqref{e:ordercube_new2} that bifurcate from zero and are of the form $\partial_S\phi_1 = \partial_S\phi_1(R_1, \delta)$. 
To do this, we first look at the nonlinear terms in this equation,
\begin{align*}
\mathcal{N}(R_1,  \partial_S \phi_1; \delta) = & -2 \beta \delta^2 \Delta_{1,S} R_1 + \tilde{N}_2(R_1, \phi_1; \delta) - \beta \tilde{N}_1(R_1, \phi_1;\delta ),\\
 = &  -2 \beta \delta^2 \Delta_{1,S} R_1  -(1-\rho_0)[ \Delta_{0,S} \phi_1 + 2\beta \partial_S \phi_0 \partial_S \phi_1)\\
&  +  N_2(R_1, \phi_1; \delta, \delta^4) - \beta N_1(R_1, \phi_1; \delta, \delta^4),
\end{align*}
where $N_1$ and $N_2$ are given as in expressions \eqref{e:nonlinear1} and \eqref{e:nonlinear2}, respectively, with $\eps = \delta^4$.
In Step 3, we prove that  these nonlinear terms give us a well defined map $\mathcal{N}: X \times Y \times \R^2 \longrightarrow Z$. This result is summarized in Proposition \ref{p:nonlinear}. We can therefore write
\begin{equation}\label{e:newnon}
 \mathcal{N} (R_1, \partial_S \phi_1; \delta) = C(R_1)( \partial_S \phi_1)^2 + D(R_1, \partial_S \phi_1) \frac{1}{S^2} + M(R_1, \partial_S \phi_1; \delta),
 \end{equation}
where, using the notation $R_1 = R_n + \zeta(S) + R_f\in X $ and $\partial_S \phi_1 = \partial_S \phi_n + \frac{a}{S} + \partial_S \phi_f \in Y $, we have defined
\begin{itemize}
\setlength \itemsep{2ex}
\item $C(R_1) \in \R$ with
\begin{equation}\label{e:C}
C(R_1) = \beta \delta - \beta \delta^3 ( \frac{1}{2} k^2 + \delta \zeta_f + \delta R_f) -\delta^4 \alpha( 1 + \delta^2( \frac{1}{2}k^2 + \delta \zeta_f + \delta R_f)) + \rmO(\eps = \delta^4) \quad \mbox{as} \quad \delta \to 0.
\end{equation}
\item $D(R_1,\partial_S  \phi_1) \in \R$ with,
\begin{equation}\label{e:D}
D(R_1,\partial_S  \phi_1) = (- \beta \delta - \beta \delta^2 a^2 + \alpha \delta^2 - \delta^6 \alpha a^2 ) [ \delta R_f +  \delta \zeta_f + \frac{1}{2}k^2] + \rmO(\eps= \delta^4) \quad \mbox{as} \quad \delta \to 0.
\end{equation}
\item $M: X \times Y \times \R \longrightarrow \tilde{Z} = L^2_{r,\gamma_1, \sigma_1 + 2}(\R^2) \oplus \R$,
\end{itemize}
Essentially, equation \eqref{e:newnon} comes from extracting all constant multiples of $(\partial_S \phi_1)^2 $ and $\frac{1}{S^2}$ from the nonlinearities $\mathcal{N}$. This calculation is done in Appendix \ref{a:nonlinear_split}. The notation $C(R_1)$ and $D(R_1, \partial_S \phi_1)$ is meant to indicate that these numbers
depend on the constant elements, $\zeta_f, R_f$ and $a,\phi_f$, which appear in the definition of the unknowns, $R_1$ and $\partial_S \phi_1$, respectively.
Here and in the rest of the paper, we also define the space
$\tilde{Z} = L^2_{r,\gamma_1,\sigma_1+2}(\R^2) \oplus \R$, with $\gamma_1 \in (-1,0)$, $\sigma_1 \in (-2,-1)$, and  norm
\begin{equation}\label{e:tildeZ1}
 \| z \|_{\tilde{Z}} = \|N \|_{L^2_{r,\gamma_1,\sigma_1+2}(\R^2)} + |z_1| \quad \forall z := (N,z_1) \in \tilde{Z}.
 \end{equation}

 Next, to prove the results of Step 1, we let $ \partial_S  \phi_1 = \partial_S  \tilde{\phi}_1  +  q$ with 
 $$ \partial_S \tilde{\phi}_1:= \partial_S \phi_n + \frac{a}{S} + \partial_S \phi_f \in Y \qquad a, \partial_S \phi_f \in \R,$$
  and $ q$ satisfying the equation,
  \begin{equation}\label{e:q_eq}
\partial_Sq + \frac{1}{S} q + 2 a  \frac{1}{S} q + 2 \beta \partial_S \phi_0 q  + C q^2 + ( a^2 C + D ) \frac{1}{S^2} =0.
\end{equation}
Here,  the constants $C$ and $ D$ are as in \eqref{e:C} and \eqref{e:D}, respectively.
Inserting the anstaz $\partial_S  \phi_1 = \partial_S  \tilde{\phi}_1  +  q$     into expression \eqref{e:ordercube_new2} we obtain
\begin{equation}\label{e:F_ift}
\begin{split}
0  = &  \Delta_{0,S} \tilde{\phi}_1  + 2\beta \partial_S \phi_0 \partial_S \tilde{\phi}_1 + M(R_1, \partial_S \phi_1; \delta) \\
& + C(R_1) \left[ 2 (q +\frac{a}{S} )   ( \partial_S \phi_n + \partial_S \phi_f) + (\partial_S \phi_n + \partial_S \phi_f)^2 \right].
\end{split}
\end{equation}

The right hand side of this equation then defines an operator $F: Y \times X \times \R \longrightarrow \tilde{Z}$.
Our goal is  to show that the map $F$ satisfies the assumptions of the implicit function theorem. The result is the following proposition.

\begin{Proposition}\label{p:phi}
Let $X$ and $Y$ be Banach spaces defined as in \eqref{e:spaces}, let $\tilde{Z} $ be defined as in \eqref{e:tildeZ1}, and consider the map $F: Y \times X \times \R \longrightarrow \tilde{Z}$ given by \eqref{e:F_ift}. Then, there exist a constant $\delta_0>0$, a neighborhood $X_0 \subset X$ containing 0, and a $C^1$ function $\partial_S \phi: X_0 \times (-\delta_0,\delta_0)  \longrightarrow Y$ satisfying: $\partial_S\phi(0,0) =0$ and  $F(\partial_S \phi(R, \delta) ; R, \delta)=0$ whenever $(R,\delta) \in X_0 \times  (-\delta_0,\delta_0) $. Moreover, if  $q$  satisfies equation \eqref{e:q_eq}, then 
the vector $(q+ \partial_S\phi(R, \delta), R, \delta)$ solves equation \eqref{e:ordercube_new2} 
in the same neighborhood of zero, i.e. $(R,\delta) \in X_0 \times  (-\delta_0,\delta_0)  $.
\end{Proposition}
\begin{proof}
We need to check that the operator $F: Y \times X  \times \R \longrightarrow \tilde{Z}$,

1) is well defined and $C^1$ with respect to all its variables,

2) that $F(0; 0,0) =0$, and

3) that its derivative $D_{\partial_S  \phi_1} F\mid_{(0;0,0)} : Y \longrightarrow \tilde{Z}$ is invertible.

Because both $M(R_1, \partial_S  \phi_; \delta)$ and $C(R_1)$ are order $\rmO(\delta)$ as $\delta \to 0$, it is easy to check that item 2) holds. Given that the derivative map is
\[ D_{\partial_S  \phi_1} F\mid_{(0;0,0)} u = \mathcal{L}_\phi u = \partial_S u + \frac{1}{S} u + 2 \beta \partial_S \phi_0 u, \]
 item 3) then follows from Proposition \ref{p:invL} given at the end of this subection, which shows that the operator $\mathcal{L}_\phi: Y \longrightarrow \tilde{Z}$ has a bounded inverse. 
 We are left with showing that $F$ is well defined and continuously differentiable with respect to all its variables. We begin by showing that $F$ is well defined. 
 
 First, Proposition \ref{p:invL}, shows that the linear part of the operator,  $\mathcal{L}_\phi$, maps to $\tilde{Z}$. Next, from the definition of the map $M$ and Proposition \ref{p:nonlinear} we may conclude that these nonlinearities are well defined in $\tilde{Z}$, provided $q$ is also an element in $Y$. This last point is shown in Lemma \ref{l:qinY} given below.
  Thus, we only need to check that the remaining terms,
\[  C(R_1) \left[ 2 (q +\frac{a}{S} )   ( \partial_S \phi_n + \partial_S \phi_f) + (\partial_S \phi_n + \partial_S \phi_f)^2 \right] \]
 are elements in the space $\tilde{Z}$.
 
 Notice that because the terms in $C(R_1)$ belong to $\R$, we can concentrate on showing that the elements in the brackets 
are in $\tilde{Z}$.
 We start by noticing that the expression $(q +\frac{a}{S} )$ is an element in $Y$. This follows from the definition of this space and Lemma \ref{l:qinY}, which shows $q$ is also in $Y$.
 As a result, we may write
\begin{align*}
  (q +\frac{a}{S} )   ( \partial_S \phi_n + \partial_S \phi_f) =  &( h + \frac{\tilde{a}}{S} + c)( \partial_S \phi_n + \partial_S \phi_f)\\
=&  h \partial_S \phi_n + \frac{\tilde{a}}{S} \partial_S \phi_n + c\partial_S \phi_n + h \partial_S \phi_f+ \frac{\tilde{a}}{S} \partial_S \phi_f + c\partial_S \phi_f,
\end{align*}
 where $h, \partial_S \phi_n \in H^1_{r,\gamma_1, \sigma_1 +1}(\R^2)$ while $c, \tilde{a}, \partial_S \phi_f$ are in $\R$. This observation then allows to look at each term individually, and conclude that
 \begin{align*}
  c\partial_S \phi_n + h \partial_S \phi_f & \in H^1_{r,\gamma_1, \sigma_1+1}(\R^2)  \subset  L^2_{r,\gamma_1, \sigma_1+2}(\R^2)\\
  \frac{\tilde{a}}{S} \partial_S \phi_f & \in L^2_{r,\gamma_1,\sigma_1+2}(\R^2)\\
  \frac{\tilde{a}}{S} \partial_S \phi_n & \in H^1_{r,\gamma_1,\sigma_1+2}(\R^2)\\
  c \partial_S \phi_f & \in \R\\
  h \partial_S \phi_n &\in L^2_{r,\gamma_1,2 \sigma_1 +3} (\R^2) \subset  L^2_{r,\gamma_1,\sigma_1+2}(\R^2)
  \end{align*}
 where the first embedding of the last line follows from Lemma \ref{l:product_rule}, while the second embedding is a consequence of Definition \ref{d:doubly-weighted} and our assumption that  $\sigma \in (-2,-1)$.
More precisely, this last condition shows that  $2 \sigma_1 +3< \sigma_1 + 2$ so that the embedding,  $L^2_{r,\gamma_1,2 \sigma_1 +3} (\R^2) \subset  L^2_{r,\gamma_1,\sigma_1+2}(\R^2)$, holds. It then follows that the term $(q +\frac{a}{S} )   ( \partial_S \phi_n + \partial_S \phi_f) $ is in $\tilde{Z}$.
 
 A similar reasoning allows us to show that $(\partial_S \phi_n + \partial_S \phi_f)^2 \in \tilde{Z}$, so we omit the details.
 
Because all terms in the definition of the operator $F$ involve linear or polynomial functions of the variables $R, \delta, \eps,$ and $\partial_s \phi_1 $, and because $\partial_S \phi_1= \partial_S\tilde{\phi} + q$ with $q$ a $C^2$ function of $\eps, \delta$, the map $F$ is continuously differentiable. This concludes the proof of this proposition.
\end{proof}

Our next task is to show that the function $q$ satisfying equation \eqref{e:q_eq} lives in the space $Y$.
This follows from Lemma \ref{l:qinY} shown below, which establishes  properties for solutions to this type of equations.
One can check that the coefficients in \eqref{e:q_eq} satisfy the assumptions presented in this lemma.

\begin{Lemma}\label{l:qinY}
Let $a, b_\infty, c,d $ be positive constants and take $b(x)$ to be a function in $C^1([0,\infty),\R)$ such that
$b(x) \sim -b_\infty + \rmO(1/x)$ as $x \to \infty$. Then, there exists a solution, $q(x)$, to the ordinary differential equation,
\[q' + \frac{a}{x} q + b(x) q - c q^2 + \frac{d}{x^2} =0  \quad x\geq 0\]
satisfying,
\[ q(x) \sim \rmO(1/x) \quad \mbox{as} \quad x \to 0, \qquad q(x) \sim q_\infty + \rmO(1/x) \quad \mbox{as} \quad x \to \infty \]
for some $q_\infty \in \R$. Moreover, there is a constant $ \nu $ and a function $q_2 \in H^1_{r,\gamma_1, \sigma_1+1}(\R^2)$, with $\gamma_1\in (-1,0)$ and $\sigma_1 \in (-2,-1)$, such that
\[ q(x) = q_2(x) + \frac{\nu }{x} +q_\infty \in Y.\]
\end{Lemma} 

The proof of this lemma is presented in Appendix \ref{a:ode}. The idea is to use the Hopf-Cole transform, $q = \alpha \frac{y'}{y}$, with $\alpha = -1/c$, to arrive at the linear ordinary differential equation,
\[ y'' + \frac{a}{x} y' + b(x) y' - \frac{cd}{x^2} y=0,\]
One can then use the Frobenious method to determine the asymptotic behavior of solutions, $F$, to this linear equation as $x \to 0$ and as $x \to \infty$. The results of Lemma \ref{l:qinY} then follow from the relationship between $F$ and $q$ and the definition of the space  $H^1_{r,\gamma_1, \sigma_1+1}(\R^2)$.

To conclude this section, we show that the linear operator $\mathcal{L}_\phi: Y \longrightarrow \tilde{Z}$ is invertible. This is done in Proposition \ref{p:invL}, which relies in Lemma \ref{l:fredphi} and Lemma \ref{l:bordering}, shown next.
We start with the proof  of Lemma \ref{l:fredphi}, which shows that 
the operator $\mathcal{L}_\phi = \partial_S + \frac{1}{S} - \lambda \partial_S \phi_0 $ is Fredholm when its domain is defined appropriately.

\begin{Lemma}\label{l:fredphi}
Let $\gamma$ and $\sigma$ be real numbers, $s$  an integer satisfying $s\geq 1$, and take $\lambda >0$. 
If the function $\phi_0(r) \in L^2_{r,\tilde{\gamma}-1}(\R^2)$, with $\tilde{\gamma} \in (0,1)$,
satisfies:
\begin{itemize}
\item $\partial_r\phi_0 \rightarrow \kappa$ as $r \to \infty$, with $\kappa >0$,
\item $(\partial_r \phi_0 -\kappa) \sim \rmO(1/r)$ as $r\to \infty$, and
\item $(\partial_r \phi_0 -\kappa) \in H^3_{r,\tilde{\gamma}}(\R^2)$,
\end{itemize}
then the operator $\mathcal{L}_\phi: H^s_{r,\gamma,\sigma}(\R^2) \longrightarrow H^{s-1}_{r,\gamma, \sigma+1}(\R^2)$, given by
\[ \mathcal{L}_\phi u = \partial_r u + \frac{1}{r} u - \lambda \partial_r \phi_0 u\]
is Fredholm. If moreover, $\phi_0(r) \to c$ with $c \in \R$, as $r \to 0$, then
\begin{enumerate}[i)]
\item if $\sigma <0$, the operator has index $i =-1$, it is injective and its cokernel is spanned y $e^{-\lambda \phi_0(r)}$; and
\item if $\sigma >0$, it has index $i =0 $ and it is invertible.
\end{enumerate}
\end{Lemma}
\begin{proof}
Recall first the results from Lemma \ref{l:fredholm2}, which shows that for $\tilde{\lambda}>0$, the operator
 $\mathcal{L}:H^s_{r,\gamma,\sigma}(\R^2) \longrightarrow H^{s-1}_{r,\gamma, \sigma+1}(\R^2)$, defined by $\mathcal{L} =  \partial_r + \frac{1}{r} - \tilde{\lambda}$, is Fredholm. 
In particular, we have that for $\sigma<0$ the operator $\mathcal{L}$ has index $i =-1$,
while for $\sigma >0$ it is invertible.
In what follows we show that $\mathcal{L}_\phi$ is a compact perturbation of $\mathcal{L}$. As a result,  $\mathcal{L}_\phi$ is also Fredholm with the same index as $\mathcal{L}$.

To this end, we let $\tilde{\lambda} = \lambda \kappa$ and   write
\[ \mathcal{L}_\phi = \partial_r + \frac{1}{r} -\tilde{ \lambda} +  \lambda( \kappa - \partial_r \phi_0) = \mathcal{L} + T,\]
where $T$ is the multiplication operator $T = \lambda( \kappa - \partial_r \phi_0)$.
Since $( \kappa - \partial_r \phi_0) \in \rmO(1/r)$ as $r \to \infty$ and $L^2_{r,\gamma,\sigma}(\R^2) \subset L^2_{r,\gamma,\sigma+1}(\R^2)$, it follows 
that $T: H^1_{r,\gamma,\sigma}(\R^2) \longrightarrow L^2_{r,\gamma,\sigma+1}(\R^2)$
is bounded. 
To prove that $T$ is compact, we show that it can be approximated by a sequence of compact operators given by $T_N =\lambda \chi_N ( \kappa - \partial_r \phi_0)$. Here,
$\chi_N$ is a radial function in $C^\infty(\R^2)$ satisfying $\chi_N(r) = 1$ for $r<N$ and $\chi_N(r) = 0$ for $r >N+1$.

First, it is clear from the following inequalities that $T_N \to T$ in the operator norm
\begin{align*}
 \| (T- T_N)u\|^2_{L^2_{r,\gamma, \sigma+1}(\R^2)} 
 & \leq 
 \lambda C \int_0^\infty |1- \chi_N|^2 \left| \frac{1}{r}\right|^2 |u|^2 m(r)^{2(\sigma+1)} \langle r\rangle^{2\gamma} r\;dr\\
 &\leq
  \lambda  C \int_N^{\infty}  \left| \frac{1}{r}\right|^2 |u|^2  \langle r\rangle^{2\gamma} r\;dr\\
  & \leq 
  \frac{ \lambda  C}{N^2} \| u \|^2_{L^2_{r,\gamma, \sigma}(\R^2)}. 
    \end{align*}
    
 Next, to show that the multiplication operators $T_N$ are compact notice that the image, $T_Nu$, is a function with compact support inside the ball $B_{2N}$ of radius $2N$.
In particular, given a sequence $\{ u_n\}  \subset H^1_{r,\gamma,\sigma}(\R^2) \subset  H^1_{r,\gamma,\sigma+1}(\R^2)$, we have that
$T_N u_n(r)\; \langle r\rangle^\gamma m(r)^{\sigma+1} \in H^1(B_{2N}) \subset \subset L^2(B_{2N})$, where this last embedding is compact by the Rellich-Kondrachov theorem.
Thus, there is a subsequence $u_{n_k}$ such that $T_Nu_{n_k} \to \bar{v}$ with $\bar{v}(r) \; \langle r\rangle^\gamma m(r)^{\sigma+1}  \in L^2(\R^2)$, as desired.

To complete the proof of this lemma, we need to show that given any $\sigma \in \R$, the operator has trivial kernel, and that its cokernel is spanned by $e^{-\lambda \phi_0}$ only if $\sigma <0$.
First, it is clear that the only function satisfying $\mathcal{L}_\phi u = 0$ is $u(r) = \frac{e^{\lambda \phi_0(r)}}{r}$. Since $\phi_0(r) \to \kappa r$ with $\kappa >0$, this function grows exponentially and thus, it is not in the domain of the operator no matter what the values of $\gamma$ and $\sigma $ are. 
On the other hand, the adjoint of the operator is 
$\mathcal{L}_\phi^* : L^2_{r,-\gamma+1,-(\sigma +1)}(\R^2) \longrightarrow H^{-1}_{r,-\gamma, -\sigma}(\R^2)$, is given by 
\[ \mathcal{L}_\phi^* = - (\partial_r + \lambda \partial \phi_0).\]
The kernel of this map is then spanned by $e^{-\lambda \phi_0(r)}$.
Since $e^{-\lambda \phi_0(r)}  \to e^{-\lambda c}$ as $r$ goes to zero, it then follows that this function is in the domain of $\mathcal{L}_\phi^*$ provided $\sigma<0$. The results of the lemma then follow.
 \end{proof}

This next result, which is known as a Bordering Lemma, gives us conditions under which 
a Fredholm operator becomes invertible.
\begin{Lemma}\label{l:bordering}[Bordering Lemma]
Let $X$ and $Y$ be Banach spaces, and consider the operator
\[ S = \begin{bmatrix}
A & B\\
C& D
\end{bmatrix} : X \times \R^p \longrightarrow Y \times \R^q,\]
with bounded linear operators $A: X \longrightarrow Y$, $B:\R^p \longrightarrow Y$, $C: X  \longrightarrow \R^q$,
$D: \R^p \longrightarrow \R^q$. If $A$ is Fredholm of index $i$, then $S$ is Fredholm of index $i +p-q$.
\end{Lemma}

\begin{proof}
One can write $S$ as the sum of a block diagonal operator with the indicated index, $i + p-q$,
and a compact operator  consisting of the off-diagonal elements.
Since compact perturbations do not alter the index of a Fredholm operator, the result then follows.
\end{proof}

This next proposition shows that the linear operator in equation \eqref{e:ordercube_new2}, when considered as a map between $Y$ and $\tilde{Z}$, is invertible. 
\begin{Proposition}\label{p:invL}
Let $\phi_0$ be as in Lemma \ref{l:fredphi}, take $\beta < 0$.
Then, the map
\[
\begin{array}{r c c}
\mathcal{L}_\phi : Y & \longrightarrow & \tilde{Z}\\
u & \longmapsto & \partial_S u + \frac{1}{S} u + 2\beta \partial_S \phi_0 u
\end{array}\]
with $Y = H^1_{r,\gamma_1,\sigma_1+1}(\R^2) \oplus \{ \frac{1}{S} \}  \oplus \R$ and $\tilde{Z} = L^2_{r,\gamma_1,\sigma_1+2}(\R^2) \oplus \R $, is invertible.
\end{Proposition}
\begin{proof}
We first take a closer look at the operator
\[
\begin{array}{c c c}
\tilde{\mathcal{L}}_\phi: H^1_{r,\gamma_1,\sigma_1+1}(\R^2) 
&\longrightarrow& L^2_{r,\gamma_1,\sigma_1+2}(\R^2)\\
u & \longmapsto & \partial_S u + \frac{1}{S} u + 2\beta \partial_S \phi_0 u.
\end{array}\]
The following calculations show that the function $\mathcal{L}_\phi \frac{1}{S} =(\partial_S + \frac{1}{S}  + 2\beta \partial_S \phi_0 ) \frac{1}{S} $ spans the cokernel of the operator $\tilde{\mathcal{L}}_\phi: H^1_{r,\gamma_1, \sigma_1+1} \longrightarrow L^2_{r,\gamma_1, \sigma_1 +2}$, which is given by $\rme^{2\beta \phi_0(S)}$ .  Since these Hilbert spaces come endowed with the inner product
\[ \langle u,v \rangle = \int_0^\infty u(S)v(S) S \;dS,\]
it suffices to show that 
$\langle \mathcal{L}_\phi \frac{1}{S}, \rme^{2\beta \phi_0(S)} \rangle \neq 0.$  This last result follows easily from an integration by parts:
\begin{align*}
\langle \mathcal{L}_\phi \frac{1}{S}, \rme^{2\beta \phi_0(S)} \rangle & = 
\frac{1}{S} \rme^{2\beta \phi_0(S)} S \Big|_0^\infty + \int_0^\infty \frac{1}{S} \mathcal{L}^*_{\phi} (\rme^{2\beta \phi_0(S)} ) S \; dS\\
& = - \rme^{2\beta  \phi_0(0)} \\
& \neq 0.
\end{align*}
Using the Borderling Lemma  we can therefore conclude that the extended operator
\[ \tilde{\mathcal{L}}_\phi: H^1_{r,\gamma_1, \sigma_1+1}(\R^2) \times \{ \frac{1}{S} \} \longrightarrow L^2_{r,\gamma_1, \sigma_1 +2},\]
is Fredholm and has index $i =0$. One can now check that the only element in the kernel of this operator is the trivial solution. Consequently, the operator is invertible. Indeed, letting $u_0 = u_1 + \frac{a}{S}$ with $ a \in \R$, we find that
the equation $\tilde{\mathcal{L}}_\phi u_0 =0$ can be written as
\[ \mathcal{L}_\phi u_1 = - 2\beta \partial_S \phi_0 \frac{a}{S}. \]
Since $\partial_S \phi_0$ converges to a constant, both near the origin and in the far field, it follows that the right hand side is in $L^2_{r,\gamma, \sigma+2}$ with $\gamma \in (-1,0)$ and $ \sigma + 2 \in (0,1)$, see Lemma \ref{l:decay} and Figure \ref{f:decay}. The results of Lemma \ref{l:fredphi} then imply that $u_1 \in H^1_{r,\gamma, \sigma+1}$, but with $\sigma \in (-2,-1)$. Therefore, this solution $u_0$ is not in the space $Y$, which uses $\sigma_1 \in (-1,0)$ (see definition of $Y$ in \eqref{e:spaces}).

 To complete the proof of this proposition, we need to show that the following operators are bounded,
\begin{equation}\label{e:op_phi}
\begin{array}{c c c}
 \mathcal{L}_\phi: \R & \longrightarrow & \tilde{Z} \\
 \mathcal{L}_\phi^{-1}: \R & \longrightarrow & Y\\
 \end{array}
\end{equation}
Here again, $\mathcal{L}_\phi u = \partial_S u + \frac{1}{S} u + 2 \beta \partial_S \phi_0 u$.
Letting $c \in \R$, we may write
\begin{equation*}
 \mathcal{L}_\phi c = \frac{c}{S} +2\beta c( \partial_S \phi_0 -k) + 2 \beta  ck.  
\end{equation*}
Since $ \left( \frac{c}{S} + 2\beta c( \partial_S \phi_0 -k) \right) \in L^2_{r,\gamma_1, \sigma_1+2}$,
it follows that the right hand side in the above expression is in $\tilde{Z}$. As a result, the first operator defined in \eqref{e:op_phi} is bounded.
On the other hand, notice that proving that the second operator appearing in \eqref{e:op_phi} is bounded, is equivalent to solving the equation $\mathcal{L}_\phi u = 1$ and showing that $u \in Y$.
To that end, let $u = u_1 + c$, with $c = 1/ 2\beta k$ and notice that  $u_1$ then has to solve:
\[ \mathcal{L}_\phi u_1= - \frac{c}{S} - 2\beta c (\partial_S \phi_0 - k). \]
Since the right hand side is an element in $L^2_{r, \gamma_1, \sigma_1+2}(\R^2)$, and since we just showed that the operator
\[ \tilde{\mathcal{L}}_\phi: H^1_{r,\gamma_1, \sigma_1+1}(\R^2) \times \{ \frac{1}{S} \} \longrightarrow L^2_{r,\gamma_1, \sigma_1 +2}\]
is invertible, we then have that $u_1 \in  H^1_{r,\gamma_1, \sigma_1+1}(\R^2) \times \{ \frac{1}{S} \}$. Consequently $u = u_1 + c \in Y$, as desired.

\end{proof}

\subsection{Step 2}

From Proposition \ref{p:phi} we know that there is a function $\partial_S \phi_1(R_1,\delta)$ representing solutions to \eqref{e:ordercube_new2} that bifurcate from zero. We can now insert this function into equation \eqref{e:ordercube_new1} with the goal of proving existence of solutions to the resulting expression.
To do this we precondition equation \eqref{e:ordercube_new1}  by $\mathcal{L}^{-1}(\delta) = (\delta^2 \Delta_{1,S} - 2)^{-1}$. The right hand side then defines an operator $F: X  \times \R \longrightarrow X$, given by
\begin{equation}\label{e:opF}
F(R_1; \delta) = R_1 + \mathcal{L}^{-1}(\delta) \left[ -2 \partial_S \phi_0\; \partial_S \phi_1(R_1,\delta) + \tilde{N}_1(R_1, \partial_S \phi_1(R_1, \delta) ;\delta) \right]. 
\end{equation}
In what follows we show that $F$ satisfies the conditions of the implicit function theorem, thus proving the results of this next proposition.

\begin{Proposition}\label{p:R_existence}
Let $X$ be the  Banach space defined as in \eqref{e:spaces}, and consider the operator $F: X \times \R \longrightarrow X$ defined in \eqref{e:opF}. Then there exist a constant $\delta_0 >0$,
and a $C^1$ function, $R: [0, \delta_0)  \longrightarrow X$, such that $R(0) =0$. Moreover, the function  $R( \delta) $ is a zero of the operator $F$, and thus it is also a solution to equation \eqref{e:ordercube_new1},
whenever $\delta \in  [0, \delta_0) $ and $\partial_S \phi$ is as in Proposition \ref{p:phi}.
\end{Proposition}

The results of Proposition \ref{p:R_existence} rely on:
\begin{itemize}
\item  Lemma \ref{l:bounded_inverse}, which shows that the operator $\mathcal{L}^{-1}(\delta) : Z \longrightarrow X$ is bounded, and $C^1$ with respect to the parameter $\delta$, for $\delta>0$.
\item Proposition \ref{p:nonlinear}, which shows that the nonlinear terms define bounded operators of the form
 $\tilde{N}_i: X \times Y \times \R^2 \longrightarrow Z$ (this also includes the term $2 \partial_S \phi_0\; \partial_S \phi_1$).
 \end{itemize}
 For ease of exposition at this stage we only summarize these results. The proof of Lemma \ref{l:bounded_inverse} is given at the end of this subsection, while the proof of Proposition \ref{p:nonlinear} appears in Subsection \ref{ss:step3}.

 \begin{proof}
The results from Lemma \ref{l:bounded_inverse} and Proposition \ref{p:nonlinear} show that the operator $F$ is well defined. To show that  $F(0;0,0) =0$ notice first that the nonlinear map $\tilde{N}_1$ is order $\rmO(\delta)$. At the same time, Proposition \ref{p:phi}  shows that the function $\partial_S \phi_1(R_1, 0)$ solves  equation \ref{e:F_ift} with $\delta =0$, i.e.
\begin{align*}
0  = &  \Delta_{0,S} \tilde{\phi}_1  + 2\beta \partial_S \phi_0 \partial_S \tilde{\phi}_1 + M(R_1, \partial_S \phi_1; \delta) \\
& + C(R_1) \left[ 2 (q +\frac{a}{S} )   ( \partial_S \phi_n + \partial_S \phi_f) + (\partial_S \phi_n + \partial_S \phi_f)^2 \right].
\end{align*}
Since both, $ M(R_1, \partial_S \phi_1; \delta)$ and $C(R_1)$ are order $\rmO(\delta)$, and because by Proposition \ref{p:invL}
 the operator $ \mathcal{L}_\phi:Y \longrightarrow  \tilde{Z} \subset Z $ defined by
 $$\mathcal{L}_\phi \; \partial_S \tilde{\phi}_1  = (\partial_S + \frac{1}{S} + 2 \beta \partial_S \phi_0) \partial_S \tilde{\phi}_1 =  \Delta_{0,S} \tilde{\phi}_1  + 2\beta \partial_S \phi_0 \partial_S \tilde{\phi}_1,$$
 is invertible, if follows that $\partial_S \phi_1(R_1, 0) =0$. Hence $F(0;0,0) =0$.
Because the norm $\| \mathcal{L}^{-1}(\delta)\|$ is bounded, the above remarks also show that the derivative map, $D_XF\mid_{(0;0,0)} : X \longrightarrow X$, is just the identity and is therefore invertible.

To show that $F$ satisfies the conditions of the implicit function theorem, we are left with checking that the operator defines a $C^1$ map with respect to all its variables.  This easily follows for $R_1 \in X$, and $\partial_S \phi = \partial_S\phi( R_1, \delta) \in Y$ as in Proposition \ref{p:phi}, since all nonlinear terms in $\tilde{N}_1(R_1,\partial_S \phi_1; \delta)$ involve linear or polynomial functions of these variables and because, by Proposition \ref{p:phi}, the function $\partial_S \phi $ is $C^1$ in all its variables.

To check that the operator $F$ is $C^1$ with respect to $\delta$, notice first that from Lemma \ref{l:bounded_inverse} we have that the linear map $\mathcal{L}^{-1}(\delta)$ is $C^1$ in the parameter $\delta$ for values of $\delta>0$. Because the nonlinear terms in $\tilde{N}_1$ are smooth with respect to $\delta$, it then follows that $F$ is also $C^1$ with respect to this parameter when $\delta >0$. 

To establish this same result for $\delta =0$, we first check that $F$ is continuous at this point. This can be done by showing there is a constant $C>0$ such that
\[ \| F(R, h) - F(R, 0)\|_X \leq C|h|.\]
To this end, we determine first that
\begin{align*}
F(R,  h) - F(R,0)  = & \Big ( R + \mathcal{L}^{-1}(h) [  -2 \partial_S \phi_0\; \partial_S \phi_1(R_1,h)+  \tilde{N}_1(R, \partial_S \phi (R, h) ; h) ] \Big )\\
& -  \Big( R + \mathcal{L}^{-1}(0) [ -2 \partial_S \phi_0\; \partial_S \phi_1(R_1,0) + \tilde{N}_1(R,  \partial_S \phi (R, 0) ; 0) \Big) \\
& = \mathcal{L}^{-1}(h) [ -2 \partial_S \phi_0\; \partial_S \phi_1(R_1,h) +\tilde{N}_1(R, \partial_S \phi (R, h);h )].
\end{align*}
Then
\begin{align*}
 \| F(R, h) - F(R,0)\|_X & \leq  h \| \mathcal{L}^{-1}(h) \| \left[ \|  2 \partial_S \phi_0\; \partial_S \phi_1(R_1,h)/h\|_Z+  \|  \tilde{N}_1(R, \phi(R, h) , h) /h \|_Z \right] \\
 \| F(R,  h) - F(R,0)\|_X & \leq h  C(\gamma),
\end{align*}
where  the last inequality holds since both, $\tilde{N}$ and $\partial_S \phi_1(R,\delta)$, are order $\rmO(\delta)$ as $\delta \to 0$.
From the above calculation we also deduce that $D_\delta F\mid_{(R,  0)}$ is bounded.
As a result the map $F$ is $C^1$ for $\delta \in \R_+$, and the results of the proposition then follow.

\end{proof}

We now proceed with the proof of Lemma \ref{l:bounded_inverse} which shows that the operator $\mathcal{L}(\delta):X \longrightarrow Z$ is invertible.

\begin{Lemma}\label{l:bounded_inverse}
Let $\delta>0$, and take $\gamma_1 \in \R$ and $\sigma_1 \in (-2,0)$ in the definition of the spaces $X$ and $Z$ given in \eqref{e:spaces} . Then, the operator
\[
\begin{array}{r c c c}
\mathcal{L}^{-1}(\delta):& Z & \longrightarrow & X\\
&(f + a/S^2 + b) &\longmapsto & (\delta^2 \Delta_{1,S} -2)^{-1}(f + a/S^2 +b)
\end{array}
\]
is bounded and $C^1$ with respect to the parameter $\delta$.
\end{Lemma}
\begin{proof}
We split the proof of this lemma into two steps. In Step A, we show that the operator is bounded, and in Step B we show that it is $C^1$ with respect to $\delta.$

{\bf Step A:} 
Suppose
 \[ (f + a/S^2 + b) \in \; Z = L^2_{r,\gamma_1, \sigma_1 +2}(\R^2) \oplus \left\{ \frac{1}{S^2} \right\} \oplus \R.\]
 In what follows we show that 
 \[ (\delta^2 \Delta_{1,S} - 2)^{-1} (f + a/S^2 + b) \in \; X = H^2_{r,\gamma_1, \sigma_1 }(\R^2) \oplus \left\{ \zeta(S) \right\} \oplus \R.\]
 
 First, since $\delta>0$ by Proposition \ref{p:fredholm1}  we know that the operator $(\Delta_1 - \Id): H^2_{r,\gamma_1, \sigma_1}(\R^2) \longrightarrow L^2_{r,\gamma_1, \sigma_1+2} (\R^2)$ has a bounded inverse provided $\sigma_1 \in (-2,0)$. Therefore, by a suitable re-scaling, this same lemma implies that
 \[ (\delta^2 \Delta_{1,S} - 2)^{-1} f \in H^2_{r,\gamma_1,\sigma_1}(\R^2).\]
Next, taking 
$$ ( a/S^2 +b) = \underbrace{ (a/S^2 + 2a/\delta^2)}_{f_2} + \underbrace{(b - 2a/\delta^2)}_{f_3}$$
we show that $(\delta^2 \Delta_{1,S} - 2)^{-1}f_2 \in \R$ and $(\delta^2 \Delta_{1,S} - 2)^{-1}f_3 \in \{ \zeta(S) \}$.

First, notice that the first term
\[ f_2 := \frac{a}{\delta^2}( \delta^2/ S^2 + 2) \in L^2_{r,\bar{\gamma}, \bar{\sigma}+2}(\R^2)\] 
for $\bar{\gamma} \in (-2, -1)$ and $\bar{\sigma} + 2 \in(1,2)$.
By Proposition \ref{p:fredholm1}, it again follows that the function 
\[ (\delta^2 \Delta_{1,S} - 2)^{-1} f_2 =: R_f \; \in H^2_{r,\bar{\gamma}, \bar{\sigma}}(\R^2)\]
exists and is unique.
Since
\[  (\delta^2 \Delta_{1,S} - 2)1 = - ( \frac{\delta^2}{S^2} + 2)  = - \frac{\delta^2}{a} f_2\]
we then obtain that $R_f = - a/ \delta^2 \in \R$.

Finally, because $f_3 := (b - 2a/\delta^2) \in \R$, then
\[ (\delta^2 \Delta_{1,S} - 2)^{-1} f_3 =  f_3 (\delta^2 \Delta_{1,S} - 2)^{-1} 1 \;\in \{ \zeta(S) \} \]
where, we recall that the function $\zeta$ is defined as satisfying  $ - \zeta= (\delta^2 \Delta_{1,S} - 2)^{-1} 1$, see \eqref{e:spaces}.

{\bf Step B:}
We start by showing the operator  $\mathcal{L}^{-1}(\delta)$ is continuous with respect to $\delta$. 
This is equivalent to showing that there exists a constant $C>0$ such that
\[ \sup_{\| f\|_Z=1} \| \left( \mathcal{L}^{-1}(\delta+h) - \mathcal{L}^{-1}(\delta)\right ) f \|_X \leq |h| C.\]

Suppose then that $f \in Z$ with $\| f\|_Z =1$, and let $u(\delta) = \mathcal{L}^{-1}(\delta) f$. Notice that
\begin{align*}
\left( \mathcal{L}^{-1}(\delta+h) - \mathcal{L}^{-1}(\delta) \right) f & =
- \left( u(\delta) - u(\delta +h) \right)\\
& = - \mathcal{L}^{-1}(\delta) \left[  \mathcal{L}(\delta+h) -  \mathcal{L}(\delta) \right] u(\delta + h) \\
& = -(2 \delta h - h^2) \mathcal{L}^{-1}(\delta)   \Delta_{1,S} \mathcal{L}^{-1} (\delta +h) f.
\end{align*}
If we show that the operator 
$ \Delta_{1,S} : X \longrightarrow Z$ is bounded, we then obtain
\begin{align} \label{e:hint_derivative}
 \| \left( \mathcal{L}^{-1}(\delta+h) - \mathcal{L}^{-1}(\delta)\right ) f \|_X &\leq 
 | 2h \delta + h^2| \| \mathcal{L}^{-1}(\delta) \| \| \Delta_{1,S}\| \| \mathcal{L}^{-1}(\delta + h) \| \| f\|_Z
\\
\nonumber
& \leq |h| C,
\end{align}
as desired.

To check that the map $\Delta_{1,S}: X \longrightarrow Z$ is indeed well defined, let $f  = f_1 + a \zeta + b \in X$.
Because $f_1 \in H^2_{r,\gamma_1, \sigma_1}(\R^2)$, it then follows from the definition
of this space that $\Delta_{1,S} f_1 \in L^2_{r,\gamma_1, \sigma_1+2}(\R^2)$.
From Lemma \ref{l:zeta}  we know that the function $\zeta$ can be decomposed as $ \zeta= \zeta_1 + \zeta_f$, where
$\zeta_1 \in H^2_{r,\gamma}(\R^2)$ with $\gamma \in (0,1)$, and $\zeta_f \in \R$.
The definition of this last weighted space then implies that the term $\Delta_{1,S}\; \zeta_1$ is in $L^2_{r,\gamma}(\R^2) \subset L^2_{r,\gamma_1,\sigma_1+2}(\R^2) $. This last embedding holds since $\gamma> \gamma_1$ and $\sigma_1 + 2 >0$ .
Lastly, combining $\zeta_f + b \in \R$ we obtain that $\Delta_{1,S} (\zeta_f +b) = -(\zeta_f+b)/S^2$, showing that $\Delta_{1,S} f  \in Z$.

To complete the proof of this Lemma, notice that from inequality \eqref{e:hint_derivative}  we are able to deduce that the derivative of the operator $\mathcal{L}^{-1}(\delta)$ with respect to $\delta $ is given by
\[ -2 \delta \mathcal{L}^{-1}(\delta) \Delta_{1,S} \mathcal{L}^{-1} (\delta): Z \longrightarrow X.\]
Since this map is the composition of bounded operators, it is well defined. We can therefore conclude that the operator $\mathcal{L}^{-1}(\delta): Z \longrightarrow X$ is $C^1$ with respect to $\delta$, for $\delta>0$.
\end{proof}

\subsection{Step 3}\label{ss:step3}
In this section we consider the nonlinear terms, 
\begin{align*}
 \tilde{N}_1(R_1, \partial_S \phi_1;\delta) &= N_1(R_1, \phi_1;\delta,\delta^4) +(1-\rho_0)[ 2 \partial_S \phi_0 \partial_S \phi_1]+ 3(1- \rho_0^2)(R_0 + \delta R_1) / \delta\\
 \tilde{N}_2(R_1,  \partial_S \phi_1; \delta) &=N_2(R_1, \phi_1; \delta,\delta^4) -(1-\rho_0)[   \Delta_{0,S} \phi_1 ] + 3 \beta (1- \rho_0^2)(R_0 + \delta R_1)/ \delta,
 \end{align*}
 where $N_1$ and $N_2$ are as in expressions  \eqref{e:nonlinear1} and \eqref{e:nonlinear2}, respectively, with $\eps = \delta^4$.
We show that these nonlinearities define bounded operators between appropriate weighted Sobolev spaces. 
We start with two preliminary results followed by two lemmas regarding terms of the form $(\partial_S \phi_0 + \delta \partial_S \phi_1)$ and $(R_0 + \delta R_1)$. We then move on to prove the main result of this subsection in Proposition \ref{p:nonlinear}.

{\bf Notation:} Throughout this section $\phi_0$ and $R_0$ are as in Proposition \ref{p:dacayRphi}. We also let 
\begin{align}\label{e:phi}
\partial_S \phi_1(S) & = \partial_S \phi_n(S) +  \frac{a}{S} +  \partial_S \phi_f \in Y,\\
\label{e:R}
R_1(S) & = R_n(S) +a \zeta(S)+  R_f \in X, \\
\end{align}
where the spaces $X = H^2_{r,\gamma_1,\sigma_1}(\R^2) \oplus \left\{\zeta(S) \right \} \oplus \R $ and $Y= H^1_{r,\gamma_1,\sigma_1+1}(\R^2) \oplus \{ \frac{1}{S} \}  \oplus \R $ are defined in \eqref{e:spaces}. (Recall that we assume $\gamma \in (-1,0)$, while $\sigma \in (-2,-1)$).

\begin{Lemma}\label{l:embedding}
Let $k$ be an integer satisfying $k \leq 3$, and take $\gamma\geq \gamma_1 >-1$, 
and $\sigma >-1$. Then, the embedding
 $$H^3_{r,\gamma}(\R^2) \subset H^k_{r,\gamma_1, \sigma}(\R^2)$$ holds.
\end{Lemma}

\begin{proof}
Let $f \in H^3_{r,\gamma}(\R^2)$. We first study the behavior of this function in the far field $\R^2 \setminus B_1$, where $B_1$ is the unit ball. Since $ \gamma_1\leq \gamma$, we have that the inequality $\langle x \rangle^{\gamma_1}< \langle x \rangle^\gamma$ holds. It  then follows that  for all multi-index $\alpha$ with $0\leq | \alpha |\leq k$, the function $D^\alpha f \in L^2_{r,\gamma_1,\sigma+ |\alpha|}(\R^2 \setminus B_1)$. 

To show that $f$ is in $H^k_{r,\gamma_1, \sigma}(B_1)$, notice first that we have the following relations,
\[ H^3_{r,\gamma}(B_1) \equiv H^3(B_1) \subset C^1(B_1),\] 
Here, the first equivalence is a consequence of the weight $\langle x \rangle$ being bounded near the origin, and the last
inclusion follows from standard Sobolev embeddings.
As a result, we can bound the norm
\[ \| f \|_{L^2_{r,\gamma_1, \sigma}(B_1) } \leq \|f\|_{L^\infty(B_1)}^2 \int_0^1 r^{2 \sigma } r \;dr <\infty,\]
where the last inequality follows from our assumption that $\sigma >-1$.
On the other hand, for any derivative of order $1\leq s\leq k$, which for simplicity we represent as $D^sf$, we have
that the norm
\[ \| D^s f\|_{L^2_{r,\gamma_1, \sigma +s }(B_1)} \leq C \int_0^1 |D^s f|^2 r^{2(\sigma +s)} \;r \;dr \leq C \int_0^1 |D^s f|^2 \;r \;dr <\infty.\]
In this case, the second inequality is a consequence of $\sigma +s >0$, while the last inequality follows from our
assumption that $f \in H^3_{r,\gamma}(\R^2)$. The results of the Lemma then follow.
\end{proof}

\begin{Lemma} \label{l:product2}
Let  $\sigma \in \R$, and take $\gamma_1, \gamma_2>-1$.
Then, given $f \in H^3_{r,\gamma_1}(\R^2)$ and $g \in H^2_{r,\gamma_2, \sigma}(\R^2)$, their product
\[ fg \in H^2_{r,\gamma, \sigma}(\R^2),\]
where $\gamma = \min\{ \gamma_1, \gamma_2\}$.
\end{Lemma}
\begin{proof}
Because the value of $\gamma_1,\gamma_2>-1$ functions in the spaces $H^3_{r,\gamma_1}(\R^2)$ and  $H^2_{r,\gamma_2,\sigma}(\R^2)$, as well as their derivatives, decay at infinity. As a result, given an index $\alpha$ with $0\leq | \alpha| \leq k$, the expressions $D^\alpha(fg)$ are bounded by an
$L^2_{r,\gamma}(\R^2\setminus B_1)$ function, where $\gamma =\min\{ \gamma_1,\gamma_2\}$. We may thus conclude that the product $fg$ is in  
 $H^k_{r,\gamma,\sigma}(\R^2 \setminus B_1)$. 
 
To show that the function $fg$ has the desired properties near the origin we use the following inclusions,
\[ f \in H^3_{r,\gamma_1}(B_1) \subset H^3(B_1) \subset C^1(B_1)\]
\[g r^{\sigma +2} \in H^k(B_1) \subset C^{k-2}(B_1).\]
These follow from Sobolev embeddings and the 
inequality $|g| r^{\sigma +s} <|g| r^\sigma$,
which is valid for values of $s>0$ and $r<1$.
A short computation then shows that any 
function of the form $D^\alpha (fg) r^{\sigma +|\alpha|} $, where $0\leq | \alpha | \leq k$, 
can be written as the product of a $C(B_1)$ function
and an element of $L^2(B_1)$. For example, letting $D^2(fg)$ denote a derivative of order $2$, 
we can use the product rule to write
\[ D^3(fg) r^{\sigma+2} = (D^2 f ) g r^{\sigma +2} +  (Df)(Dg r^{\sigma+2}) + f (D^2 g r^{\sigma+2} ).\]
In this case, the functions $ g r^{\sigma +2} , (Df)$ and $f$ are bounded, while the remaining terms are in $L^2(B_1)$.
 The results of the Lemma then immediately follow.
\end{proof}

With the help of the above Lemmas, and properties of the spaces $H^k_{r,\gamma, \sigma}(\R^2)$ summarized
in Subsection \ref{ss:spaces}, we now bound the terms $(\partial_S \phi_0 + \delta \partial_S \phi_1)^2 $
and $(R_0 + \delta R_1)^m$ in appropriate weighted spaces. These lemmas will then be used to prove the main result of this
subsection, Proposition \ref{p:nonlinear}. When looking at the proofs of these results, it is worth keeping in mind Figure \ref{f:decay}, which summarizes decay properties of the spaces, $H^k_{r,\gamma, \sigma}(\R^2)$ and $H^k_{r,\gamma}(\R^2)$.

\begin{Lemma} \label{l:phi_0}
Let $\gamma_1 \in (-1,0)$, $\sigma_1 \in (-2,-1)$, define $Y = H^1_{r,\gamma_1, \sigma_1+1}(\R^2) \oplus \left \{ \frac{1}{S}\right\} \oplus \R$, and take
\[ \partial_S \phi_1 = \partial_S \phi_n + \frac{a}{S} + \partial_S \phi_f \in Y.\] 
Then, the expression
\[ (\partial_S \phi_0 + \delta \partial_S \phi_1)^2 \in L^{2}_{r,\gamma_1,\sigma_1+2}(\R^2) \oplus \left\{ \frac{1}{S^2} \right\} \oplus \R.\]
\end{Lemma}

\begin{proof}
From Proposition \ref{p:dacayRphi} we know that the function $\partial_S \phi_0 \to \kappa$ as $S \to \infty$, with $\kappa \in \R$; while  the function $(\partial_S \phi_0 -\kappa ) \in H^3_{r,\gamma_1}(\R^2)$.
We can then expand
\begin{align*}
(\partial_S \phi_0 + \delta \partial_S \phi_1)  = 
& (\partial_S \phi_0 - \kappa ) + \delta \partial_S \phi_n + \delta  \frac{a}{S} + (\kappa + \delta \partial_S \phi_f).
\end{align*}
Because $\sigma_1+1 > -1$, Lemma \ref{l:embedding} shows that $H^3_{r,\gamma_1}(\R^2) \subset H^1_{r,\gamma_1, \sigma_1 +1}(\R^2)$. We can therefore write
\begin{align*}
(\partial_S \phi_0 + \delta \partial_S \phi_1)^2 & = ( h + \delta \frac{a}{S} + c)^2\\
& = [ h^2 + 2\delta h \frac{a}{S} + 2c h + 2 \delta c \frac{a}{S}] + \delta^2 \frac{a^2}{S^2} + c^2
\end{align*}
where $h = (\partial_S \phi_0 - \kappa ) + \delta \partial_S \phi_n  \in H^1_{r,\gamma_1, \sigma_1 +1}(\R^2)$ and $c=(\kappa + \delta \partial_S \phi_f) \in \R$.
To finish the proof of this lemma we must show that the terms in the brackets are in $L^2_{r,\gamma_1, \sigma_1 + 2}(\R^2)$.

Notice first that $ \frac{a}{S} \in L^2_{r,\gamma_1, \sigma_1 +2}(\R^2)$, while the embedding $H^1_{r,\gamma_1, \sigma_1 +1}(\R^2) \subset L^2_{r,\gamma_1,\sigma_1+2}(\R^2)$, which holds since $\sigma_1+1< \sigma_1+2$, shows that $h$ is the correct space.
Recalling the definition of our weighted Sobolev spaces, we also find that $  \frac{a}{S} h \in L^2_{r,\gamma_1, \sigma_1+2}(\R^2)$. Finally, Lemma \ref{l:product_rule} shows that the product $h^2$ is in the space  $H^1_{r,\gamma_1, 2\sigma_1+3}(\R^2) \subset L^2_{r,\gamma_1,\sigma_1+2}(\R^2)$. Notice that this last embedding follows from the
inequality $ 2 \sigma_1 + 3< \sigma_1 +2$, which holds thanks to the assumption 
that $\sigma_1\in (-2,-1)$.
\end{proof}

In this next lemma, the function $\zeta(S)$ is as in \eqref{e:spaces}.

\begin{Lemma} \label{l:R_0}
Let $\gamma_1 \in (-1,0)$, $\sigma_1 \in (-2,-1)$, define $X = H^2_{r,\gamma_1, \sigma_1}(\R^2) \oplus \left\{ \zeta(S) \right \} \oplus \R$, and take
\[ R_1 = R_n + a \zeta(S) + R_f \in X.\]
Then, for any $m \in \N \cup \{0\}$, the expression
\[ (R_0 + \delta R_1)^m \in  H^3_{r,\gamma_1}(\R^2) \oplus H^{2}_{r,\gamma_1,\sigma_1}(\R^2)  \ \oplus \R.\]
\end{Lemma}

\begin{proof}
From Proposition \ref{p:dacayRphi} we know that the function $R_0  \to -\frac{1}{2} \kappa^2$ as $S \to \infty$, with $\kappa \in \R$; while $(R_0 -\frac{1}{2} \kappa^2 ) \in H^3_{r,\gamma_1}(\R^2)$.
Thus, we may expand,
\[ (R_0 + \delta R_1) = (R_0 - \frac{1}{2} \kappa^2)  + \delta R_n + \delta a \zeta(S) + ( \frac{1}{2} \kappa^2 + \delta R_f).\]
Then from Lemma \ref{l:zeta} we know that $\zeta(S) = \zeta_1 + \zeta_f$ where $\zeta_1 \in H^3_{r,\gamma}(\R^2)$ with $\gamma \in (0,1)$, and $\zeta_f$ is a constant. We can therefore write
\[ (R_0 + \delta R_1) = f+ g + d\]
with $f = (R_0 - \frac{1}{2} \kappa^2)  + \zeta_1 \in H^3_{r,\gamma_1}(\R^2)$, $ g = \delta R_n \in H^2_{r,\gamma_1,\sigma_1}(\R^2)$, and $ d = ( \frac{1}{2} \kappa^2 + \delta R_f) \in \R$.

Now, looking at the product,
\[ (f+g)^m = \sum_{\ell =0}^m {m \choose \ell} f^\ell g^{m-\ell},\]
and using Lemma \ref{l:product2}, one can show that all terms  of the form $f^\ell g^{m-\ell}$ are in $H^2_{r,\gamma_1,\sigma_1}(\R^2)$, whenever
$1< \ell , m-\ell <m$. Indeed, since  $H^3_{r,\gamma_1}(\R^2) \subset C^1(\R^2)$, it follows that this space is a Banach algebra and therefore $f^p \in H^3_{r,\gamma_1}(\R^2) $ for any integer $p$.
Similarly, because $\sigma_1 <-1$, it follows from Lemma \ref{l:banach_alg} that the space $H^2_{r,\gamma_1, \sigma_1}(\R^2)$ is also closed under multiplication, so that $g^p$ is in $H^2_{r,\gamma_1, \sigma_1}(\R^2)$, again for any integer $p$. Therefore, $f^\ell g^{m-\ell}$ is always the product of a bounded function, $f^\ell$, times a function in $g^{m-\ell} \in H^2_{r,\gamma_1,\sigma_1}(\R^2)$. If we now include terms with $\ell=0, m$, we may conclude that $(f+g)^m \in  H^3_{r,\gamma_1}(\R^2) \oplus H^{2}_{r,\gamma_1,\sigma_1}(\R^2)$.
If we then write
\[ (f + g + c)^m = \sum_{\ell =0}^m {m \choose \ell} ( f+g)^\ell c^{m-\ell}\]
we immediately see that the results of the lemma hold.

\end{proof}
We now state and prove the main result of this section.

\begin{Proposition}\label{p:nonlinear}
Let $X,Y,Z$ be Banach spaces defined as in 
\eqref{e:spaces}.
Then, the maps $\tilde{N}_i: X \times Y \times \R \longrightarrow Z$, for $i =1,2$, defined by
\begin{align*}
 \tilde{N}_1(R_1, \partial_S \phi_1;\delta) &= N_1(R_1, \phi_1;\delta,\delta^4) +(1-\rho_0)[ 2 \partial_S \phi_0 \partial_S \phi_1]+ 3(1- \rho_0^2)(R_0 + \delta R_1) / \delta\\
 \tilde{N}_2(R_1,\partial_S \phi_1; \delta) &=N_2(R_1, \phi_1; \delta,\delta^4) -(1-\rho_0)[   \Delta_{0,S} \phi_1 ] + 3 \beta (1- \rho_0^2)(R_0 + \delta R_1)/ \delta,
 \end{align*}
 where $N_1$ and $N_2$ are as in expressions  \eqref{e:nonlinear1} and \eqref{e:nonlinear2}, respectively,
 are bounded.
 \end{Proposition}
\begin{proof}
To prove the proposition  one establishes decay properties at infinity and near the origin for each the nonlinear terms in $\tilde{N}_i$, $i=1,2$.
This analysis is very similar to the one done in the proofs of Lemmas \ref{l:phi_0} and \ref{l:R_0}. We therefore omit most of the details and just state the results. We will only show that terms $\tilde{N}(w;\eps)$ appearing in the expressions $N_i$, $i=1,2$ are in the space $Z$. Here $\eps = \delta^4$, but for convenience we omit this detail in the proof and just write $\eps$.

\begin{itemize}

\setlength \itemsep{2ex}
\item Elements in $L^2_{r,\gamma_1, \sigma_1+2}(\R^2)$:\\

\begin{itemize}

\setlength \itemsep{2ex}
\item[\ding{118}] $( 1- \rho_0) ( \Delta_{0,S} \phi_1 + 2 \partial_S \phi_0 \partial_S \phi_1)$
\item[\ding{118}] $( \rho_0 + R_0 + \delta R_1) ( \Delta_{0,S} \phi_0 + \delta \Delta_{0,S} \phi_1)$
\item[\ding{118}] $ (\partial_S R_0 + \delta \partial_S R_1) (\partial_S \phi_0 + \delta \partial_S \phi_1)$
\end{itemize}

\item Elements in $L^2_{r,\gamma_1, \sigma_1 + 2 }(\R^2) \oplus \left \{ \frac{1}{S^2} \right \} \oplus \R$:\\

\begin{itemize}
\setlength \itemsep{2ex}
\item[\ding{118}] $( R_0 + \delta R_1)( \partial_S \phi_0 + \delta \partial_S \phi_1)^2$
\end{itemize}

\item Elements in $L^2_{r,\gamma_1, \sigma_1+1}(\R^2)  \oplus \R$:\\

\begin{itemize}
\setlength \itemsep{2ex}
\item[\ding{118}] $ \rho_0 ( R_0 + \delta R_1)^2$
\item[\ding{118}] $(R_0 + \delta R_1)^3$
\end{itemize}

\item Elements in $H^2_{r,\gamma_1, \sigma_1+1}(\R^2)$:\\

\begin{itemize}
\item[\ding{118}] $ (1- \rho_0) (R_0 + \delta R_1)$
\end{itemize}

\item Elements in $L^2_{r,\gamma_1, \sigma_1+2}(\R^2) \oplus \left\{ \frac{1}{S^2} \right \}$:\\

\begin{itemize}
\item[\ding{118}] $ \Delta_{1,S} ( R_0 + \delta R_1)$
\end{itemize}

\item Elements in $L^2_{r,\gamma_1, \sigma_1 +2 }(\R^2) \oplus \R$:\\

\begin{itemize}
\item[\ding{118}] $ \rho_0( \partial_S \phi_0 + \delta \partial_S \phi_1)^2$
\end{itemize}

\item Since $\rho_0 \in H^3_{r,\gamma}(\R^2) $ we then have that  $\partial_r \rho_0 \in H^2_{r,\gamma}(\R^2)$ and
$\Delta_1 \rho_0 \in H^1_{r,\gamma}(\R^2)$ and $\partial_r \rho \partial_S \phi_0 \in H^2_{r,\gamma}(\R^2)$.

\end{itemize}

The statement of the proposition then follows from the above results, the embeddings $H^s_{r,\gamma_1,\sigma_1}(\R^2)  \subset H^k_{r,\gamma_1, \nu}(\R^2)$, which hold whenever $\sigma_1 <\nu$ and $s>k$,
and the inclusion $H^3_{r,\gamma}(\R^2) \subset H^2_{r,\gamma_1,\sigma_1+2}(\R^2)$ which follows from Lemma \ref{l:embedding}.

We conclude by showing that $\tilde{N}(w; \eps) \rme^{-\rmi \phi} \in Z$. 

First, recall that
\[ \tilde{N}(w;\eps) =   \frac{ \eps^2 D}{d_R} (1 + \rmi \beta) \Delta_{1} |w|^2 w + (1- \frac{\eps^2 D}{d_R} \Delta_{1}) N(w;\eps).\]
where $N(w;\eps)$ is given as in Hypothesis (H1), and is thus composed of 
elements of the form $|w|^{2n} w$ with $n \in \{ 1,2,3, \cdots\}$.
We show that $\tilde{N}(w; \eps) \rme^{-\rmi \phi} \in Z$ in three steps.

{\bf Step A.} We consider first the term  $ (\Delta_{1} |w|^2 w ) \rme^{-\rmi \phi}$. Writing $|w|^2w = \rho^3 \rme^{\rmi \phi}$ and letting $W = \rho^3$, we can then expand
\begin{equation}\label{e:deltaW1}
 ( \Delta_1 W \rme^{\rmi \phi} ) \rme^{-\rmi \phi}
= \Delta_1 W + \rmi W \Delta_0 \phi  - 2\rmi \partial_r \phi \partial_r W - W ( \partial_r \phi)^2.
\end{equation}
In what follows we check that each element  is in the space $Z = L^2_{r,\gamma_1, \sigma_1+2}(\R^2) \oplus \left\{ \frac{1}{S^2} \right\} \oplus  \R$,
keeping in mind that $S = \delta r$. 

We start by noticing that 
\begin{equation}\label{e:W_space}
 W(r) := \rho^3 = ( \rho_0(r) + \delta^2 (R_0 + \delta R_1))^3 \in H^3_{r,\gamma}(\R^2)  \oplus H^2_{r,\gamma_1, \sigma_1}(\R^2) \oplus \R. 
 \end{equation}
This follows from the fact that $(\rho_0-1) \in H^3_{r,\gamma}(\R^2)$ and the results of Lemma \ref{l:R_0}, which show that the space
$H^3_{r,\gamma}(\R^2) \oplus H^2_{r,\gamma_1, \sigma_1}(\R^2) \oplus \R$ is a Banach algebra.
Similarly, in the proof of Lemma \ref{l:phi_0} it is shown that  
\begin{equation}\label{e:phi_space}
\partial_S \phi = \partial_S \phi_0 + \partial_S \phi_1 \in H^1_{r,\gamma_1, \sigma_1+1}(\R^2) \oplus \left \{ \frac{1}{S} \right \} \oplus \R.
\end{equation}
We will use these simplifications in the analysis for each of the terms appearing in expression \eqref{e:deltaW1}.

To prove that $ \Delta_1 W \in Z$, we  use \eqref{e:W_space} and write
\[ \Delta_1 W = \Delta_1 ( h_1 + h_2 + c) = \Delta_1 h_1 + \Delta_1 h_2 + \frac{c}{r^2} \]
where $c \in \R$ and the functions, $h_1$ and $h_2$ are in $H^3_{r,\gamma}(\R^2)$ and $H^2_{r,\gamma_1, \sigma_1}(\R^2)$, respectively. The result then follows from the definition of these spaces, and the embedding $H^1_{r,\gamma}(\R^2) \subset L^2_{r,\gamma_1, \sigma_1+2}(\R^2)$, given by Lemma \ref{l:embedding}.

Next, because the function $W$ is uniformly bounded, to show $W \Delta_0 \phi \in Z$, we only need  to show that $ \Delta_0 \phi \in Z$. To do this, we use \eqref{e:phi_space} and write
\[  \Delta_0 \phi =    \partial_r ( h + \frac{a}{r} + c) + \frac{1}{r} ( h + \frac{a}{r} + c) = \partial_r h + \frac{1}{r} h + \frac{c}{r} \]
where $h \in H^1_{r,\gamma_1, \sigma_1 + 1}(\R^2)$ and $a, c\in \R$. The result again follows from the definition of these spaces, and the fact that the function $\frac{1}{r} \in L^2_{r,\gamma_1, \sigma_1 +2}(\R^2)$ (see Figure \ref{f:decay} in Section \ref{s:fredholm}).

To study the term $\partial_r W \partial_r \phi$, we consider \eqref{e:phi_space} and expand,
\[ \partial_r W \partial_r \phi  = \partial_r W( h + \frac{a}{r} + c) = \partial_r W h  + \frac{a}{r} \partial_r W + c \partial_r W\]
with $h \in H^1_{r,\gamma_1, \sigma_1+1}(\R^2)$ and $a,c \in \R$.
Because 
$$\partial_r W \in H^2_{r,\gamma}(\R^2) \oplus H^1_{r,\gamma_1, \sigma_1 +1}(\R^2) \subset H^1_{r,\gamma_1, \sigma_1 +1}(\R^2)$$ 
(see Lemma \ref{l:embedding}), then $\frac{a}{r} \partial_r W \in L^2_{r,\gamma_1,\sigma_1+2}(\R^2)$.
Similarly, we have that
$\partial_r W \in H^1_{r,\gamma_1, \sigma_1 +1}(\R^2) \subset L^2_{r,\gamma_1, \sigma_1 +2}(\R^2)$, which follows  from the definition of these doubly-weighted Sobolev spaces. 
Lastly, using Lemma \ref{l:product_rule} we see that term $\partial_r W h \in L^2_{r,\gamma_1, 2 \sigma_1 +3}(\R^2) \subset L^2_{r,\gamma_1,\sigma_1+2}(\R^2) $, where this last embedding follows from the inequality $2\sigma_1 +3 \leq \sigma_1 +2$.

Finally, because $W$ is uniformly bounded, the results from Lemma \ref{l:phi_0} allow us to conclude that,
$W(\partial_r \phi)^2 \in L^2_{r,\gamma_1,\sigma_1 +2}(\R^2)\oplus \left \{ \frac{1}{S^2} \right\} \subset Z$.

{\bf Step B.}
We now look at the nonlinear terms included in the expression $N(w;\eps) \rme^{-\rmi \phi}$, which as mentioned above, are of the form
$|w|^{2n}w \rme^{-\rmi \phi} $ with $n \in \{ 1, 2,3, \cdots\}$. Given that $w = \rho \rme^{\rmi \phi}$, we may write,
\[ |w|^{2n}w \rme^{-\rmi \phi}  = \rho^{2n+1}.\]
Thus, it suffices to show that $\rho^p \in Z$ for all odd integers $p \geq 3$. Notice that this easily follows from the definition of $\rho$,
\[ \rho = \rho_0(r) + \delta^2 (R_0 + \delta R_1))^3 \in H^3_{r,\gamma}(\R^2)  \oplus H^2_{r,\gamma_1, \sigma_1}(\R^2) \oplus \R,\]
the fact that this last space is closed under multiplication, and the embeddings
\[ H^3_{r,\gamma}(\R^2)  \oplus H^2_{r,\gamma_1, \sigma_1}(\R^2) \subset H^2_{r,\gamma_1, \sigma_1}(\R^2)  \subset L^2_{r,\gamma_1, \sigma_1 +2}(\R^2),\]
which follow from Lemma \ref{l:embedding} and the definition of these spaces.

{\bf Step C.} 
To show $[\Delta_1 N(w;\eps) ] \rme^{-\rmi \phi}$, or equivalently $[\Delta_1 \rho^p \rme^{\rmi \phi}] \rme^{-\rmi \phi}$, is in the space $Z$, we use the results from Step 2 to write $\rho^p \rme^{\rmi \phi} = V(r) \rme^{\rmi \phi}$, for some function $V(r) \in H^2_{r,\gamma_1,\sigma_1}(\R^2) \oplus \R$.
Then,
\[ [\Delta_1  V(r) \rme^{\rmi \phi} ]   \rme^{-\rmi \phi} =  [ \Delta_1 (V \rme^{\rmi s \phi}) ] \rme^{-\rmi \phi} 
=  \Delta_1 V + i V \Delta_0 \phi + 2 i  \partial_r \phi \partial_r V  -  V( \partial_r \phi)^2 .   \]
We can then use a similar argument as in Step A, to show that these terms  are in the space $Z$.

\end{proof}

\appendix

\section{}\label{a:nonlinear}
In this appendix we justify hypothesis (H1), reproduced below.

\begin{Hypothesis*}[H1]
The nonlinear term $N(w; \eps)$ is of order $\rmO(\eps |w|^4w)$, and every term in this expression is of the form
$c |w|^{2n} w$, with $c\in \C$ and $n \in \{1,2,3, \cdots\}$.
\end{Hypothesis*}

To justify Hypothesis (H1) we `briefly' review the results from \cite{jaramillo2022rotating}, where the nonlocal complex Ginzburg-Landau equation 
\eqref{e:main} is derived as an amplitude equation for systems that undergo a Hopf bifurcation and that include nonlocal diffusion.

As in the case of reaction diffusion equations, the approach from \cite{jaramillo2022rotating} takes advantage of the separation in scales between the frequency of time oscillations that emerge from the Hopf bifurcation, $t$, and the long spatial, $R = \eps |x|$, and time scales, $T = \eps^2 t$, that dictate changes in the amplitude of 
these oscillation. To derive a reduce equation for spiral waves it is then assumed that these patterns bifurcate from the zero solution and have an amplitude that is proportional to the small parameter $\eps$. 
In polar coordinates $(r, \vartheta)$, these patterned solutions are also assumed to take the form $U(r,\theta,R; \eps, \mu) = U(r, R; \vartheta + \omega t + \eps^2 \mu t)$, where $\omega$ and $\mu$ are constants related to the rotational speed of the spiral wave and where $\theta = \vartheta + \omega t + \eps^2 \mu t$.

With these two assumptions one can then write the solution as a regular perturbation in $\eps$, i.e.
\[ U(r, \theta, R;  \eps,\mu) = \eps U_1(\theta, R; \eps, \mu) + \eps^2 U_2(\theta, R; \eps, \mu) + \eps^3 U_3(r,\theta; \eps,\mu).\]
  Inserting this ansatz into the original nonlocal system and separating terms of equal powers in $\eps$ one obtains a sequence of equations.
 Assuming that the first order correction term is of the form
 \[ U_1(\theta, R;\eps,\mu) = W_1 w(R;\eps,\mu) \rme^{\rmi \theta} + \bar{W}_1 \bar{w}(R;\eps,\mu) \rme^{-\rmi \theta},\]
 with $(\rmi \omega, W_1)$ the eigenvalue-eigenvector pair associated with the Hopf bifurcation, one can then show that $U_1$ solves the 
  $\rmO( \eps)$ equation, with $w(R)$ an arbitrary complex valued function. Moreover, it is then possible to prove that the $\rmO(\eps^2)$ equation has a solution, $U_2\sim \rmO(U_1^2)$, which depends smoothly on $U_1$.

Finally, the nonlocal Ginzburg-Landau equation is derived after gathering all terms of order $\rmO(\eps^3)$ and higher into a single equation, which for ease of exposition in this summary we call the `main equation'.
The results from \cite{jaramillo2022rotating}  then show that there exist a Banach space, $X$, and projections, $P: X \longrightarrow X_\parallel$ and $(I-P): X \longrightarrow X_\perp$, such that the linear part of this main equation, $L$, can be split into an invertible operator $L_\perp: X_\perp \longrightarrow X_\perp$ and a bounded operator $L_\parallel : X_\parallel \longrightarrow X_\parallel$. 
Using these projections and Lyapunov-Schmidt reduction, the main equation can then be split into an invertible equation for the variable $U_3$, and a reduced equation for the unknown amplitude $w(r)$. 

The invertible equation can then be written as
\begin{equation}\label{e:Ginv}
 0 = G(U_3; U_1, \eps, \mu)
 \end{equation}
with $G: X_\parallel \times X_\perp \times \R^2 \longrightarrow X_\perp$ a smooth well defined map  with a derivative operator $D_{U_3} G= L_\perp: X_\perp \longrightarrow X_\perp$ that is invertible. 
The implicit function theorem then proves the existence of a smooth map $\Psi: \mathcal{U} \subset X_\parallel \times \mathcal{V} \subset \R^2 \longrightarrow X_\perp$, such that $U_3= \Psi(U_1, \eps, \mu)$ solves equation \eqref{e:Ginv}, and satisfies $\Psi(0, \eps,\mu) =0$ and $D_{U_1} \Psi(0;\eps,\mu)$ for $(\eps,\mu) \in \mathcal{V}$. Inserting this map $\Psi$ into the reduced equation and projecting onto the space $X_\parallel$  then results in the nonlocal complex Ginzburg-Landau equation \eqref{e:main}.

A couple of remarks are in order:
 
 \begin{itemize}
 \item Notice that unlike a formal multiple scales approach, where this reduced equation comes from applying a solvability condition to the $\rmO(\eps^3)$ equation,
the nonlocal  Ginzburg-Landau equation \eqref{e:main} derived in \cite{jaramillo2022rotating} includes terms of order $\rmO(\eps^3)$ and higher. In equation \eqref{e:main} the higher order terms are then summarized in the expression $N(w;\eps)$, which therefore includes powers of the function $U_3 = \Psi(U_3, \eps, \mu) \sim \rmO(U_1^2)$.
\item The projection $P: X \longrightarrow X_\parallel$ takes the form
\[ P u = \int_0^{2 \pi} \langle u, W_1 \rangle \rme^{\rmi \theta} \;d \theta + \int_0^{2 \pi} \langle u, \bar{W}_1 \rangle \rme^{-\rmi \theta} \;d \theta \]
but  because we are interested in real solutions, it suffices to consider only the projection onto the span of $\{ \rme^{\rmi \theta}\}$.

\item The terms in the expression $N(w, \eps)$ are generated from  the nonlinearities that appear in the original nonlocal equation. 
We assume that these terms are polynomial functions of the unknown variable.

\end{itemize}

The above discussion then implies that elements in $N(w,\eps)$ come from terms of the form $U_2^2 \sim U_1^4$ or  $U^p_1U^q_3$ with $p,q \geq 1$. 
Looking first at $U_2$ and using the definition of $U_1$ we find
\begin{equation}\label{e:series1}
 U_2 \sim U_1^4 = \sum_{s=0}^4 {4 \choose s} w^s \bar{w}^{4-s} \rme^{\rmi \theta( 4-2s)}.
 \end{equation}
On the other hand, since $U_3 = \Psi(U_1,\eps,\mu) \sim \rmO(U_1^2)$ 
is a smooth function, we can Taylor expand it about the point $U_1 = 0$. As a result, we may write 
\[
U_1^p U_3^q =  U_1^p \Psi^q = \sum_{s=0}^ \infty a_s U_1^{2q+p+s},
\]
where, again from the definition of $U_1$, we have that
\begin{equation}\label{e:series2} 
U_1^{m}  = \sum_{k=0}^m {m \choose k} w^k\bar{w}^{m-k} \rme^{\rmi \theta( 2k -m)}
\end{equation}
for $ m \geq 2q+p \geq 3$.
Notice then that projecting terms of the form \eqref{e:series1} or \eqref{e:series2} into the span of $\{ \rme^{i \theta}\}$, gives
\begin{align*}
 P U_1^{m} & = \sum_{k=0}^m {m \choose k} P w^k\bar{w}^{m-k} \rme^{\rmi \theta( 2k -m)} \\
 & = \sum_{\substack{ k\geq (1+m)/2, \\ 2k - m =1} }^m {m \choose k} w^k\bar{w}^{k-1} \rme^{\rmi \theta} \\
 &= \sum_{\substack{ k\geq (1+m)/2, \\ 2k - m =1} }^m {m \choose k} |w|^{2(k-1)} w \rme^{\rmi \theta}.
 \end{align*}
 Because we have the restrictions $2k -m = 1$, $k\geq (1+m)/2$, and $m \geq 3$, we may conclude $k -1 \geq 1$. Therefore,
 elements in $N(w;\eps)$ are given by $|w|^{2n} w$ with $ n = \{ 1, 2, 3, \cdots\}$.

\section{}\label{a:Fredholm}

To prove Lemma \ref{l:Delta-c} will use a result by Kato \cite{kato}, which we state next. 
\begin{Proposition}[Kato, p.370]\label{p:kato}
Let $T(\gamma)$ be a family of compact operators in a Banach space $X$ which are holomorphic for all $ \gamma \in \C$. Call $\gamma$ a singular point if 1 is an eigenvalue of $T(\gamma)$. Then either all $\gamma \in D$ are singular points or there are only finitely many singular points in each compact subset of $D$.
\end{Proposition}

With the above result in hand, we prove that $\Delta -\mathrm{Id}$ is invertible when posed over weighted 
and doubly-weighted Sobolev spaces.

\begin{Lemma*}\ref{l:Delta-c}
Let $s\geq 2$ and suppose $\gamma \in \R$. Then, the operator
\[
  \Delta-\mathrm{Id}:  H^s_\gamma(\R^2) \longrightarrow H^{s-2}_\gamma(\R^2)\]
 is invertible.
\end{Lemma*}
\begin{proof}
We work with the commutative diagram,
\begin{center}
\begin{tikzcd}[column sep = large, row sep= huge]
H^s_{\gamma}(\R^d) \arrow[r, "(\Delta-\Id)"] \arrow[d, "\langle x \rangle^\gamma "] 
& H^{s-2}_{\gamma}(\R^d) \arrow[d, "\langle x \rangle^\gamma"] \\
H^{s}(\R^d) \arrow[r, "\mathcal{L}(\gamma) "] & H^{s-2}(\R^d)
\end{tikzcd}
\end{center}
where the operator $\mathcal{L}$ is given by $\mathcal{L}(\gamma)u = (\Delta-\Id) u + T(\gamma) u$ with
\[T(\gamma)u = 2  \gamma \langle x \rangle^{-2} \; \nabla u \cdot x  +  ( \gamma( \gamma-1) |x|^2 \langle x \rangle^{-4} + \gamma \langle x\rangle^{-2} ) \; u.\]
 Each term in the map $T(\gamma)$ involves a multiplication operator that can be approximated by a function with compact support. It then follows, by 
 Sobolev embeddings,  that $T(\gamma)$ is a compact perturbation of the invertible operator $(\Delta-\Id): H^{s}(\R^d) \rightarrow H^{s-2}(\R^d)$. Consequently, $(\Delta-\Id): H^{s}_{\gamma}(\R^d) \rightarrow H^{s-2}_{\gamma}(\R^d)$ is Fredholm of index zero.

To check that the operator has a trivial kernel, we proceed by contradiction.
Suppose for the moment that there is a number $\gamma^* \in \R$
  and a function $u \in H^{s}_{\gamma^*}(\R^d)$ such that $\mathcal{L}(\gamma^*)u =0$. Then,  the commutative diagram and the embedding $H^{s}_{\gamma^*}(\R^d) \subset H^{s}_{\gamma}(\R^d)$, with $\gamma < \gamma^*$, imply that $\mathcal{L}(\sigma)u = 0$ for all  $\gamma < \gamma^*$.  

On the other hand, the operator 
\[
\begin{array}{c c c}
H^{s}(\R^d) & \longrightarrow &H^{s-2}(\R^d)\\
u &\longmapsto  & (\Delta- \Id)^{-1} T(\gamma) u,
\end{array}\]
which is  compact and analytic for all $\gamma \in \C$, has $\lambda =1$ as an eigenvalue if and only if  $\mathcal{L}(\gamma)$ has a non trivial kernel. It then follows, by our previous argument, that $\lambda=1$ is an eigenvalue of $(\Id-\Delta)^{-1} T(\gamma)$ for all $\gamma< \gamma^*$. Using Proposition \ref{p:kato} we may then conclude that $\lambda =1$ is an eigenvalue 
for all $\gamma \in \C$. In particular, this result holds for $\gamma=0$, which is a contradiction since $(\Delta-\Id)^{-1}T(0) =(\Id-\Delta)^{-1} $ is invertible. As a result, $\ker \mathcal{L}(\gamma) =\{0\}$ for all $\gamma \in \R$ and the map $(\Delta-\Id): H^{s}_{\gamma}(\R^d) \rightarrow H^{s-2}_{\gamma}(\R^d)$ is therefore an isomorphism.

\end{proof}


\section{}\label{a:nonlinear_split}

In this section we work on re-grouping the nonlinear terms
\begin{align*}
\mathcal{N}(R_1,  \partial_S \phi_1; \delta) = & -2 \beta \delta^2 \Delta_{1,S} R_1 + \tilde{N}_2(R_1, \phi_1; \delta) - \beta \tilde{N}_1(R_1, \phi_1;\delta ),\\
 = &  -2 \beta \delta^2 \Delta_{1,S} R_1  -(1-\rho_0)[ \Delta_{0,S} \phi_1 + 2\beta \partial_S \phi_0 \partial_S \phi_1)\\
&  +  N_2(R_1, \phi_1; \delta, \delta^4) - \beta N_1(R_1, \phi_1; \delta, \delta^4),
\end{align*}
to arrive at equation \eqref{e:newnon},
\begin{equation*}
 \mathcal{N} (R_1, \partial_S \phi_1; \delta) = C(R_1)( \partial_S \phi_1)^2 + D(R_1, \partial_S \phi_1) \frac{1}{S^2} + M(R_1, \partial_S \phi_1; \delta).
 \end{equation*}
We recall that here, $N_1$ and $N_2$ are given as in expressions \eqref{e:nonlinear1} and \eqref{e:nonlinear2}, respectively, with $\eps = \delta^4$, and
\begin{itemize}
\setlength \itemsep{2ex}
\item $C(R_1) \in \R$ with
\begin{equation}\label{e:C}
C(R_1) = \beta \delta - \beta \delta^3 ( \frac{1}{2} k^2 + \delta \zeta_f +  \delta R_f) -\delta^4 \alpha( 1 + \delta^2( \frac{1}{2}k^2 + \delta \zeta_f + \delta R_f)) + \rmO(\eps = \delta^4) \quad \mbox{as} \quad \delta \to 0.
\end{equation}
\item $D(R_1,\partial_S  \phi_1) \in \R$ with,
\begin{equation}\label{e:D}
D(R_1,\partial_S  \phi_1) = (- \beta \delta - \beta \delta^2 a^2 + \alpha \delta^2 - \delta^6 \alpha a^2 ) [ \delta R_f +  \delta \zeta_f + \frac{1}{2}k^2] + \rmO(\eps= \delta^4) \quad \mbox{as} \quad \delta \to 0.
\end{equation}
\item $M: X \times Y \times \R \longrightarrow \tilde{Z} = L^2_{r,\gamma_1, \sigma_1 + 2}(\R^2) \oplus \R$.
\end{itemize}

As mentioned in Section \ref{s:existence}, this is the result of extracting multiples of $(\partial_S \phi_1)^2$ and $\frac{1}{S^2}$ from $\mathcal{N}$. We start with the constant $C(R_1)$.

Using the expressions for $N_1$,  given in \eqref{e:nonlinear1}, and $N_2$, given in \eqref{e:nonlinear2}, we see that there are three terms in  $N_2 - \beta N_1$,  which involve the function $(\partial_S \phi_1)^2$ (ignoring the order $\rmO(\eps = \delta^3)$ terms, since these are already higher order). These are,
\begin{itemize}
\setlength \itemsep{2ex}
\item $- \delta^2 \alpha ( \rho_0 + \delta^2( R_0 + \delta R_1)) (\delta \partial_S \phi_1)^2$,
\item $- \beta \delta( R_0 + \delta R_1)(\delta \partial_S \phi_1)^2$
\item $ \beta \delta \rho_0 ( \partial_S \phi_1)^2$
\end{itemize}

Since $\rho_0$ and $R_0$ are given as in Propositions \ref{p:decayg} and \ref{p:dacayRphi}, we may re-arrange,
\begin{align*}
 ( \rho_0 + \delta^2( R_0 + \delta R_1)) = & (\rho_0 -1) + 1 + \delta^2[ (R_0 - \frac{1}{2} \kappa^2) + \delta(R_n +\zeta(S) + R_f) + \frac{1}{2} \kappa^2 ] \\
 = & (\rho_0 -1) + \delta^2( R_0 - \frac{1}{2}\kappa^2) + \delta^3 (\zeta_1(S) + R_n) +  \left[ 1 + \delta^2 \left ( \frac{1}{2} \kappa^2 + \delta \zeta_f + \delta R_f \right ) \right ]
 \end{align*}
where $\kappa \in \R$ is given in Proposition \ref{p:dacayRphi},  and we have also used the notation $R_1 = R_n +  \zeta(S) + R_f \in X$. In addition, $\zeta(S) = \zeta_1(S) + \zeta_f$ is as in Lemma \ref{l:zeta}.
As a result 
\[ - \delta^2 \alpha ( \rho_0 + \delta^2( R_0 + \delta R_1))  = h  - \delta^2 \alpha \left[ 1 + \delta^2 \left ( \frac{1}{2} \kappa^2 + \delta \zeta_f + \delta R_f \right ) \right ],
\]
with $h \in \tilde{Z}$ and the remaining terms in $\R$.

A similar analysis then shows that
\begin{itemize}
\setlength \itemsep{2ex}
\item $- \beta \delta( R_0 + \delta R_1)= h - \beta \delta (\frac{1}{2} \kappa + \delta ( \zeta_f + R_f) ) $
\item $ \beta \delta \rho_0 = h + \beta \delta $
\end{itemize}
where again we use $h$ as a generic function in $\tilde{Z}$. The results for $C(R_1)$ then follow.

Next, we gather all terms that are multiples of $\frac{1}{S^2}$. These are,
\begin{itemize}
\setlength \itemsep{2ex}
\item $ (\alpha \delta^2 - \beta \delta) [ \Delta_{1,S} R_0 + \delta \Delta_{1,S} R_1]$
\item $ - ( \alpha \delta^4 + \beta \delta) [ (R_0 + \delta R_1) ( \partial_S \phi_0 + \delta \partial_S \phi_1)^2]$
\end{itemize}
To see why,  using Proposition \ref{p:dacayRphi}  and Lemma \ref{l:zeta}, we first expand 
\begin{align*}
 (R_0 + \delta R_1) = & (R_0 - \frac{1}{2} \kappa^2) + \delta(R_n +\zeta_1(S) + \zeta_f + R_f) + \frac{1}{2} \kappa^2 \\
= &\underbrace{ (R_0 - \frac{1}{2} \kappa^2) + \delta \zeta_1(S) + \delta R_n}_{h}  + \left[ \frac{1}{2} \kappa^2 + \delta (\zeta_f + R_f) \right]
 \end{align*}
 with $h \in H^3_{r,\gamma}(\R^2) \oplus H^2_{r,\gamma_1, \sigma_1}(\R^2)$, $\gamma \in (0,1), \gamma_1 \in (-1,0)$ and $\sigma_1 \in (-2,-1)$. It then follows that
 \[ \Delta_{1,S} (R_0 + \delta R_1) = \Delta_{1,S} h -  \left[ \frac{1}{2} \kappa^2 + \delta (\zeta_f + R_f) \right] \frac{1}{S^2} \]
 where, from the definition of these spaces, we may conclude that $ \Delta_{1,S} h \in H^1_{r,\gamma}(\R^2) \oplus L^2_{r,\gamma_1, \sigma_1 + 2}(\R^2) \subset \tilde{Z}$.

To analyze the second expression, $[ (R_0 + \delta R_1) ( \partial_S \phi_0 + \delta \partial_S \phi_1)^2]$, we first write
\[ (R_0 + \delta R_1)  = h_1 +  \underbrace{\left[ \frac{1}{2} \kappa^2 + \delta (\zeta_f + R_f) \right] }_{d}
 \]
with $h_1 =h $ as above and $ d\in \R$. Similarly, using the notation $\partial_S \phi_1 = \partial_S \phi_n + \frac{a}{S} + \partial_S \phi_f \in Y$ introduced in Section \ref{s:existence}, together with Lemma \ref{l:phi_0} stated in Subsection \ref{ss:step3}, we may write
\[ (\partial_S \phi_0 + \delta \partial_S \phi_1)^2 = h_2 + \delta^2 \frac{a^2}{S^2} + c\]
with $h_2 \in L^2_{r,\gamma_1,\sigma_1+2}(\R^2)$ and $a, c \in \R$. We then have
\begin{align*}
 (R_0 + \delta R_1) ( \partial_S \phi_0 + \delta \partial_S \phi_1)^2 = & (h_1 +d) ( h_2  \delta^2 \frac{a^2}{S^2} + c)\\
 & = (h_1 + d) h_2 + (h_1+d) c + \delta^2 \frac{ a^2}{S^2} h_1 + \delta^2 \frac{d a^2}{S^2}
 \end{align*}
where, thanks to Proposition \ref{p:nonlinear} proved in Subsection \ref{ss:step3}, we may conclude that the first three terms are in $\tilde{Z}$. The results for $D(R_1, \partial_S \phi_1)$ then follow.


\section{}\label{a:ode}

\begin{Lemma*}
Let $a, b_\infty, c,d $ be positive constants and take $b(x)$ to be a function in $C^1([0,\infty),\R)$ such that
$b(x) \sim -b_\infty + \rmO(1/x)$ as $x \to \infty$. Then, there exists a solution, $q(x)$, to the ordinary differential equation,
\[q' + \frac{a}{x} q + b(x) q - c q^2 + \frac{d}{x^2} =0  \quad x\geq 0\]
satisfying,
\[ q(x) \sim \rmO(1/x) \quad \mbox{as} \quad x \to 0, \qquad q(x) \sim q_\infty + \rmO(1/x) \quad \mbox{as} \quad x \to \infty \]
for some $q_\infty \in \R$. Moreover, there is a constant $ \nu $ and a function $q_2 \in H^1_{r,\gamma_1, \sigma_1+1}(\R^2)$, with $\gamma_1\in (-1,0)$ and $\sigma_1 \in (-2,-1)$, such that
\[ q(x) = q_2(x) + \frac{\nu }{x} +q_\infty \in Y.\]
\end{Lemma*}

\begin{proof}
We start by using the Hopf-Cole Transform, $ q = \alpha \dfrac{y'}{y}$, with $ \alpha = 1/c$, to arrive at the linear equation
\[ y'' + \frac{a}{x} y' +b(x) y' - \frac{dc}{x^2} y =0.\]
We first use the Frobenius Method \cite{boyce2021elementary} to determine the behavior of solutions near the origin.

It is straight forward to check that $x =0$ is a regular singular point of the equation. We can therefore look for solutions of the form
$ y(x) = x^r \sum_{k=0}^\infty a_k x^k.$
Multiplying the equation by $x^2$ and inserting this guess, we obtain,
\[ 0 = \sum_{k=0}^\infty a_k (r+k)(r+k-1) x^{r+k} + a_k a (r+k) x^{r+k} + a_k b(x) (r+k) x^{r+k+1} - a_k cd  x^{k+1}, \]
and one finds that the indicial equation is
\[ r^2 +(a-1) r -cd =0.\]
Since the discriminant $ (a-1)^2+ 4 cd >0$, it follows that the roots, $r_1,r_2$, of this equation are real. We then have two cases, $r_1 - r_2$ is or is not an integer. If the difference between the roots is not an integer, we then have the two linearly independent solutions of the form
\[ y_1(x) = x^{r_1} \sum_{k=0}^\infty a_k x^k, \qquad y_2(x) = x^{r_2} \sum_{k=0}^\infty b_k x^k.\]
If, on the other hand, $r_1-r_2 $ is an integer then the two linearly independent solutions are
\[ y_1(x) = x^{r_1} \sum_{k=0}^\infty a_k x^k, \qquad y_2(x) = x^{r_2}  \sum_{k=0}^\infty b_k x^k + C \log(x) y_1(x), \]
for some constant $C$. In either case, we see that
$ q(x) = \alpha \frac{y'}{y} \sim \rmO(1/x) \quad \mbox{as} \quad x \to 0.$

Next, to determine the behavior of solutions at infinity, we use the change of coordinates $\xi = 1/x$, which leads to the equation
\begin{equation}\label{e:y_infty}
 \frac{d^2 y}{d \xi^2} + \frac{(2-a)}{\xi} \frac{d y}{d \xi} - \frac{\tilde{b}(\xi)}{\xi^2} \frac{d y}{d\xi} - \frac{dc}{\xi^2} y =0.
 \end{equation}
The point $\xi = 0$ is now an irregular singular point, so we first look for solutions of the form
\[y(\xi) = \rme^{A / \xi} Y(\xi) \quad \mbox{with} \quad Y(\xi) = \xi^r \sum_{k =0}^\infty a_k \xi^{k}.\]
This leads to the following equation for $Y(\xi)$
\[ Y'' + \left( \frac{(2-a)}{\xi} - \frac{(\tilde{b}(\xi) +2A)}{\xi^2} \right) Y' + \left(\frac{A(\tilde{b}(\xi) +A)}{\xi^4} + \frac{a A}{\xi^3} - \frac{cd}{\xi^2} \right) Y =0. \]
Multiplying the equation by $\xi^2$ and inserting the power series for $Y(\xi)$, we obtain
\begin{align*}
 0 =  \sum_{k=0}^\infty &\; a_k (r+k)(r+k-1) \xi^{r+k} + a_k (2- a) (r+k) \xi^{r+k} - a_k (\tilde{b}(\xi) +2A)(r+k) \xi^{r+k-1} \\
 & + a_k \frac{A( \tilde{b}(\xi) +A)}{\xi} \xi^{r+k-1} + a_k a A \xi^{r+k-1} - a_k cd x^{r+k}.
 \end{align*}
Letting $A=- b_\infty$, we find that the term $\frac{A(\tilde{b}(\xi) +A)}{\xi} \sim \rmO(1)$ as $\xi \to 0$. Therefore, the indicial equation, which is determined by the expression
\[  - (\tilde{b}(\xi) +2A)r + \frac{A( \tilde{b}(\xi) +A)}{\xi} + a A =0,\]
gives us $r = a + \rmO(1) >0$. 

Although, the sum describing $Y(\xi)$ might not converge, it still provides us with an approximation for $y(\xi)$ when $\xi \sim 0$. Keeping in mind that $\xi = 1/x$, we conclude that
\[ y(x) \sim \rme^{-b_\infty x} x^r \sum_{k =0}^\infty  \frac{a_k}{x^k} \qquad x\to \infty .\]
It then follows that $q(x) = \alpha \frac{y'}{y} \sim - \alpha b_\infty + \rmO(1/x)$ as $x \to \infty$.

Lastly, to find a second solution to equation \eqref{e:y_infty} that is linearly independent from the solution found above, we assume
$ y(\xi) = \sum_{k=0}^\infty b_k \xi^{r+k}$. Multiplying the equation by $\xi^2$ and inserting this guess, gives us
\[ \sum_{k=0}^\infty b_k (r+k)(r+k-1) \xi^{r+k} + b_k (2-a)  (r+k) \xi^{r+k} - b_k \tilde{b}(\xi) \xi^{r+k-1}
- b_k cd \xi^{r+k}.\]
The lowest term in the above expression must satisfy $ \tilde{b}(\xi) r =0$, which shows that $ r=0$.
We conclude that this second solution to the equation satisfies
\[ y(x) \sim \sum_{k =0}^\infty b_k \frac{1}{x^k} \quad \mbox{as} \quad x \to \infty\]
Consequently, 
$q(x) = \alpha \frac{y'}{y} \sim  \rmO(1/x)$ as $x \to \infty$.

Because the function $q(x)$ solves a first order differential equation, it is a function in $C^1((0,\infty),\R)$. The asymptotic expansions found above, also show that it is possible to find constants $q_\infty, \nu \in R$ such that
\[ q(x) = q_2(x) +\frac{\nu}{x} + q_\infty\]
with $q_2(x) \sim \rmO(1)$ near the origin and $q_2(x) \sim \rmO(1/x)$ at infinity. It then follows that $q_2 \in H^1_{r,\gamma_1, \sigma_1+1}(\R^2)$ with $\gamma_1 \in (-1,0)$ and $\sigma_1 \in (-2,0)$. 

\end{proof}


\noindent {\bf E-mail:} gabriela@math.uh.edu
\bibliographystyle{plain}
\bibliography{spirals}

\end{document}